\documentclass[11pt]{amsart}

\input{LPPBuseProcMacros}

\allowdisplaybreaks[1] 
\author[C.~Janjigian]{Christopher Janjigian}
\address{Christopher Janjigian\\ Purdue University\\  Department of Mathematics \\ 150 N University St\\ West Lafayette, IN 47901\\ USA.}
\email{cjanjigi@purdue.edu}
\urladdr{http://www.math.purdue.edu/~cjanjigi}
\thanks{C.~Janjigian was partially supported by National Science Foundation grant DMS-2125961 and Simons Foundation grant MPS-TSM-00012155.}

\author[F.~Rassoul-Agha]{Firas Rassoul-Agha}
\address{Firas Rassoul-Agha\\ University of Utah\\  Department of Mathematics\\ 155S 1400E\\   Salt Lake City, UT 84112\\ USA.}
\email{firas@math.utah.edu}
\urladdr{http://www.math.utah.edu/~firas}
\thanks{F.\ Rassoul-Agha was partially supported by National Science Foundation grants DMS-2054630 and DMS-2450951 and Simons Foundation grant MPS-TSM-00013661.}

\author[T.~Sepp\"al\"ainen]{Timo Sepp\"al\"ainen}
\address{Timo Sepp\"al\"ainen\\ University of Wisconsin--Madison\\  Department of Mathematics\\ 480 Lincoln Drive \\  Madison, WI 53706\\ USA.}
\email{seppalai@math.wisc.edu}
\urladdr{http://www.math.wisc.edu/~seppalai}
\thanks{T.~Sepp\"al\"ainen was partially supported by National Science Foundation grants DMS-2152362 and DMS-2448375, and by the Wisconsin Alumni Research Foundation.}

\title[Rooted Gibbs-DLR Measures in Planar Directed Polymers]{Rooted Gibbs-DLR Measures in Planar Directed Polymers}
\begin{document}

\begin{abstract} 
We study rooted Gibbs-DLR measures in the directed polymer model on $\bbZ^2$ with an ergodic disorder distribution which satisfies an additional mild hypothesis. We prove that the set of extremal rooted Gibbs-DLR measures is closed and totally ordered, characterize extremality in terms of path coalescence, and show that each fully supported extremal rooted Gibbs measure canonically generates a globally consistent and coalescing family of extremal rooted Gibbs measures indexed by all lattice sites. These families are, moreover, totally ordered.

Building on this structure, we prove strong existence and strong uniqueness of the associated Busemann process, together with an $L^1$ continuity theorem for the shift-covariant Busemann cocycles which is joint in the inverse temperature, the tilt parameter, and the random environment. This yields, as a corollary, in-probability continuity of the generated extremal Gibbs measures corresponding to directions of differentiability under bounded i.i.d.\ perturbations of the weights. In positive temperature, it shows in-probability convergence of  the generated extremal Gibbs measures. At zero temperature, it also yields quenched subsequential large deviation principles for the corresponding positive-temperature rooted Gibbs-DLR measures on path space. The rate functions are determined by a zero-temperature Busemann cocycle and vanish precisely on the infinite geodesics generated by that cocycle.
\end{abstract}

\maketitle

\tableofcontents

\section{Introduction}
A basic question in mathematical statistical mechanics is to understand the infinite-volume Gibbs-Dobrushin-Lanford-Ruelle (Gibbs-DLR) measures associated to a given specification. In disordered systems, this question can be especially rich because the random environment removes the deterministic symmetries that often organize Gibbs states in homogeneous models and instead can give rise to genuinely random, environment-dependent structure. This paper focuses on directed polymers, which originated as effective models of domain walls in random media \cite{Hus-Hen-85}. See \cite{Com-17} for a recent pedagogical introduction. Directed polymers are among the simplest models that exhibit the rich quenched structure characteristic of disordered statistical mechanics, while retaining enough structure from their simple path geometry to make the infinite-volume Gibbs-DLR boundary somewhat tractable. Despite this relative tractability, many open problems remain.

In the setting of directed polymers, there are two natural types of infinite-volume Gibbs-DLR measures. The type that has perhaps attracted the most interest previously are Gibbs-DLR measures on bi-infinite up-right paths, which are the analogue of bi-infinite geodesics in first- and last-passage percolation. In first-passage percolation, such bi-infinite paths define interfaces that are dual to non-trivial ground states in the disordered Ising ferromagnet \cite{New-97}. A heuristic argument due to Newman, recorded in \cite[Section 5]{Auf-Dam-Han-17}, predicts non-existence of bi-infinite geodesics and Gibbs-DLR measures whenever the path transverse fluctuation exponent strictly exceeds $1/2$. This prediction has been made rigorous in several exactly solvable or integrable models \cite{Bas-Hof-Sly-22,Bal-Bus-Sep-20,Gro-Jan-Ras-25-jsp,Bus-Sep-22-ejp}, and, for first-passage percolation in general dimension, under strong unproven hypotheses on the model \cite{Ale-23}. It is unclear to us if the transverse fluctuation exponent being greater than $1/2$ should hold for all of the models we study in the present work.

This paper focuses on a second type of Gibbs-DLR measures, which always exist by compactness arguments. These are measures on semi-infinite up-right paths rooted at a fixed lattice site. Such measures arise naturally as Doob $h$-transforms of the finite volume measures and can also be viewed as encoding harmonic functions (also known as eternal solutions) of the associated discretization of the KPZ equation. See \cite{Bak-Cat-Kha-14, Bak-13, Bak-16, Bak-Li-19, Car-Sou-17, Jan-Ras-Sep-23-1F1S-} for work related to this point of view concerning random and stochastic PDEs.  These rooted Gibbs-DLR measures have a rich structure similar to the structure of semi-infinite geodesic rays in first- and last-passage percolation. 

In planar settings, when the limiting free energy is sufficiently regular, all extremal rooted Gibbs measures are supported on paths that satisfy a strong law of large numbers. Results of \cite{Jan-Ras-20-aop} show that for every value of the limiting direction  there exists at least one such measure, constructed from the so-called Busemann process (which we discuss momentarily), and that uniqueness holds for typical directions, with non-uniqueness confined to an exceptional set of random directions where two measures constructed from the Busemann process differ. If the Busemann process furnishes two measures with the same direction, then they stochastically bracket all other such measures. Whether or not such measures exist was left open in \cite{Jan-Ras-20-aop}, but it was later shown in \cite{Bat-Fan-Sep-25} that this directional non-uniqueness does occur in the exactly solvable log-gamma polymer. In that setting, the set of non-uniqueness directions is random and dense. It remains open in positive temperature whether any extremal Gibbs measures besides  those furnished by the Busemann process can exist, but the expectation is no, in line with the picture proved in the exactly solvable exponential last-passage percolation model and for the renormalization fixed point of the KPZ universality class, the directed landscape. See \cite{Bus-Sep-Sor-24, Bus-25-, Cou-11,Jan-Ras-Sep-23}. See also \cite{Bas-Mj-26} for recent related results in the setting of first-passage percolation on hyperbolic groups.

This paper focuses on exploring the structure of the set of rooted Gibbs-DLR measures in a fixed realization of the random environment, without invoking any exactly solvable inputs or unproven hypotheses.  We show first that for each root, the set of extremal rooted Gibbs-DLR measures is closed and totally ordered under stochastic monotonicity. We also characterize extremality in terms of path coalescence and prove that each fully supported extremal rooted Gibbs measure canonically generates a globally consistent and totally ordered family of extremal Gibbs measures, one rooted at each lattice site.  In this sense, the rooted Gibbs structure inherits the same kind of planar order and coalescence phenomena that is expected to underlie the geometry of semi-infinite geodesics in zero temperature. See \cite{Jan-Ras-Sep-23} for the current state of the art in zero temperature. 

Building on this structure, we turn to the tilt-indexed Busemann process. Its marginals are shift covariantly defined cocycles, equivalently the transition probabilities of shift covariantly selected Gibbs-DLR measures. Earlier work \cite{Jan-Ras-20-aop, Gro-Jan-Ras-25-tams} constructed such objects on extended probability spaces; in the planar setting, \cite{Jan-Ras-20-aop} used monotonicity to organize them into a process. What remains is to identify when this process is canonically defined on the original environment and whether it is unique. In this project, we first show that the Gibbs-DLR measures generated by these shift-covariant cocycles are extremal in planar directed polymers. We then use the closed total order on extremal rooted Gibbs-DLR measures to prove strong existence and strong uniqueness of the Busemann process on the canonical weight space. This means that we show that every probability space that supports the weight field also supports a canonical realization of the Busemann process and any two realizations of this process agree almost surely.

Focusing on shift covariantly defined Gibbs-DLR measures can be thought of as analogous to studying shift-invariant Gibbs-DLR measures in models without disorder. This philosophy dates back at least to the pioneering work of Aizenman and Wehr \cite{Aiz-Weh-90} and their introduction of \textit{metastates}. It is believed that in KPZ-class directed polymer models, the Busemann process should generate all extremal Gibbs measures, though this remains open even in the exactly solvable setting in positive temperature. We remark that while the Busemann process is shift covariantly defined, in the exactly solvable log-gamma model, it is known to generate certain extremal Gibbs measures, corresponding to random tilts, which are provably not shift-covariant; see \cite{Bat-Fan-Sep-25}. 
In brief, the expectation is that while shift covariantly defined Gibbs-DLR measures are not expected to be exhaustive, they are believed to be dense.

One of the motivations for proving strong existence and uniqueness is that it provides the inputs needed to access certain stability questions, which were among our main goals in pursuing this project.  Specifically, we prove an $L^1$ continuity theorem for the Busemann process that is joint in the inverse temperature, the weights, and the environment. A corollary shows, for example, that the Busemann cocycles corresponding to directions of differentiability, along with their associated Gibbs-DLR measures, are stable in probability under bounded perturbations of the weights. It was very far from obvious to us that such a statement should be true. A second application concerns regularity of the Gibbs-DLR measures as the inverse temperature varies and in particular in the zero-temperature limit. We use strong existence and uniqueness to prove convergence of rooted Gibbs measures to infinite geodesics in the zero temperature limit. We quantify this with a new (subsequential) large deviation principle (LDP): along deterministic subsequences with inverse temperature tending to infinity, the corresponding finite-temperature rooted Gibbs-DLR measures on the path space satisfy quenched large deviation principles whose rate functions are determined by the zero-temperature Busemann cocycle and vanish precisely on infinite geodesics.

\subsection{Connections to some prior works}
Work on the structure of rooted Gibbs-DLR measures in directed lattice polymers traces back at least to a seminal paper of Comets and Yoshida \cite{Com-Yos-06}, which proved diffusivity of certain rooted Gibbs-DLR polymer measures in weak disorder in dimensions three and higher. The next significant contribution appeared in \cite{Geo-etal-15}, specialized to the exactly solvable planar log-gamma polymer introduced in \cite{Sep-12-corr}. Systematic study of rooted Gibbs-DLR measures in general weight planar polymer models using the DLR formalism and the associated convex analytic questions which are a main focus of the present work first appeared in \cite{Jan-Ras-20-aop}.

\textit{Weak existence} of the planar Busemann process under an i.i.d.~weight hypothesis appeared previously in \cite{Jan-Ras-20-aop}. In this setting, weak existence is used in the usual sense of stochastic analysis, to mean existence of a probability space supporting the Busemann process. This argument was partially inspired by earlier work of Damron and Hanson \cite{Dam-Han-14} on lattice first-passage percolation. See also \cite{Gar-Mar-05, Gou-07, Hof-08} for a chain of related works. We refer the reader to the introduction of \cite{Gro-Jan-Ras-25-tams} for a more complete accounting of the history. \cite{Gro-Jan-Ras-25-tams} subsequently generalized weak existence to a very wide class of random walks in ergodic random potentials generalizing directed polymers, random walks in random environments, and first- and last-passage percolation in all dimensions, with arbitrary admissible steps. The weak existence result in \cite{Gro-Jan-Ras-25-tams} combined with monotonicity furnishes the weak existence input to our proof of strong existence and strong uniqueness.

Weak uniqueness of the Busemann cocycles, meaning uniqueness of the distribution of a Busemann cocycle with a given mean, was proven in \cite{Jan-Ras-20-jsp} under an i.i.d.~weight hypothesis by adapting an argument due to Chang from the queueing literature \cite{Cha-94}. See also the uniqueness argument in \cite{Fan-Sep-20} for another application of that type. The results of \cite{Jan-Ras-20-jsp} are not used in the present work and are fully subsumed by the main results of this paper.

Although not phrased in this language, the first strong uniqueness proof of the type we apply here, without extra hypotheses on the limit shape/free energy, in a planar growth model appears in the work of Ahlberg and Hoffman \cite{Ahl-Hof-16-} on first-passage percolation. A vast prior literature, tracing back to a seminal paper of Newman \cite{New-95}, relied on strong shape hypotheses to realize Busemann cocycles uniquely on the canonical space. The ideas of \cite{Ahl-Hof-16-} were later adapted to planar directed last-passage percolation by the authors in \cite{Jan-Ras-Sep-25-strong-}.

Bates, Fan, and the third author \cite{Bat-Fan-Sep-25} investigated the direction-indexed  Busemann process in the general planar directed polymer model with i.i.d.\ weights, based on its weak existence proved in \cite{Jan-Ras-20-aop}, and in the solvable log-gamma polymer where the strong existence of the Busemann process can be derived from its special properties. In the general case the  direction-indexed Busemann process is not known to equal the full tilt-indexed Busemann process, as explained in Remark \ref{rmk:dir_Bus}. The present paper strengthens some results of \cite{Bat-Fan-Sep-25} and simplifies some proofs in the general case, where \cite{Bat-Fan-Sep-25} had to prove partial ergodicity properties. One significant benefit of the present strong existence result is that the Busemann process inherits properties of the weight field, so in particular is mixing under the i.i.d.\ weight hypothesis of \cite{Bat-Fan-Sep-25}.   Remark 2.5 of \cite{Bat-Fan-Sep-25} is now moot: no regularity assumption on the shape function is required for the conclusion that the Busemann process is a function of the weights. 
Connections to earlier results in \cite{Bat-Fan-Sep-25} are explained  in Remarks \ref{rmk:Dspace} and \ref{rmk:BFS5} below. 

One of our results, the characterization of extremality in terms of existence of a coalescing coupling, is largely a consequence of Goldstein's \cite{Gol-79} celebrated theorem concerning maximal couplings. We highlight that this argument in no way relies on planarity and so this can be expected to hold in general lattice random walks in random environments at positive temperature. It is a significant open question whether the geodesics and Gibbs measures generated by shift-covariant recovering cocycles (as described below) should coalesce in high dimensions. In polymer language, this now can be seen to correspond to asking whether shift covariantly selected Gibbs-DLR measures are supported on extremal Gibbs measures or if they are generically mixtures. 
Using a version of the Licea-Newman \cite{Lic-New-96} argument as adapted in \cite{Jan-Ras-20-aop}, we show that in planar polymers, shift covariantly selected Gibbs measures are always extreme.

\subsection{Methodology and main proof ideas}
We now sketch the main ideas in the proofs. Our pointwise results on Gibbs-DLR measures hold whenever the weights are taken to be ergodic, provided that a certain random walk constructed from weight differences diverges; see Condition \ref{cond:standing}. Once we discuss the Busemann cocycles and process, extra regularity hypotheses are needed. Every result in the paper holds unconditionally if the weights are taken to be i.i.d., non-degenerate, and to have $p>2$ moments. Most of the results hold significantly more generally, under our Condition \ref{cond:moment-mixing}, which is essentially a trade-off between mixing and moment hypotheses. Some of our final applications concerning regularity of the Busemann cocycles jointly in the inverse temperature and weight field would require either restricting to i.i.d.\ weights throughout throughout or extending one of our results from our previous work \cite{Jan-Ras-Sep-25-strong-} when applied in the zero temperature limit. See the discussion below.

The first fixed-environment result is a new characterization of extremality. We show that a fully supported rooted Gibbs-DLR measure is extremal exactly when the path measures it induces in the quadrant ahead of its root admit a coalescing coupling. This uses tail triviality together with Goldstein's coupling criterion \cite[Theorem 2.1]{Gol-79}: equality on the tail $\sigma$-algebra is equivalent to existence of a coupling under which two paths eventually agree. We next show that the set of extremal rooted Gibbs-DLR measures is totally ordered under stochastic monotonicity using a coupling argument. This is analogous to the total order on geodesics in last-passage percolation. 

These characterizations are key tools in the proof that the set of extremal rooted Gibbs-DLR measures is closed; in general, one can only expect that this set is $G_\delta$. This is a significant difference from zero temperature, where the analogous closure is immediate from the definition of the geodesic property. Starting from a convergent sequence of extremal rooted Gibbs-DLR measures, we construct a monotone coupling which can be used to verify a second new characterization of extremality phrased in terms of certain conditional expectations of partition functions. 

Starting with an extremal rooted Gibbs-DLR measure, a martingale convergence argument produces a family of Gibbs-DLR measures, one rooted at each lattice site. This construction is analogous to the way a classical Busemann function built from a single infinite geodesic induces infinite geodesics from all other points. We then prove that each of these Gibbs measures has trivial tail $\sigma$-algebra. Goldstein's theorem then yields a coalescing coupling of the entire family, implying that every member is extremal. Consequently, the rooted total order extends to a global total order on rooted Gibbs-DLR measures, independent of the choice of root.

Closure allows us to view the set of extremal Gibbs measures as a random totally ordered compact space. This makes measurable selections of largest elements in closed lower sets possible. We use such selections to build a shift-covariant quantile function for a shift-covariant extremal Gibbs measure, viewed as a random variable in the totally ordered compact space of extremal Gibbs measures, with respect to which we can use inverse CDF sampling. This idea was motivated by a construction in \cite{Ahl-Hof-16-} and previously used in last-passage percolation in \cite{Jan-Ras-Sep-25-strong-}. Strong existence then comes from a coupling argument where the key point is that the aforementioned coalescence forces the inverse CDF sample to not depend on the independent uniform used in the sampling process. Strong uniqueness comes quickly from the total order. 

With strong existence in hand in positive temperature we use a coupling argument combined with classical results concerning weak-strong convergence (see \cite{Jac-Mem-81} for a survey) to bootstrap from tightness to $L^1$ convergence of the Busemann functions. This result is partially conditional on an extension of one of our results from \cite{Jan-Ras-Sep-25-strong-} which we expect to hold, but all of the results in the paper hold unconditionally if the weights are taken to be i.i.d.~throughout; see Remark \ref{rem:unnecessary-condition}. In-probability convergence of extremal Gibbs-DLR measures comes as a corollary of the $L^1$ regularity of the Busemann process. In the zero temperature limit, the convergence of Busemann cocycles is equivalent to computing a large deviation rate function for certain cylinder probabilities. Diagonalizing, these can be shown to give a full LDP under which zeroes of the rate function are precisely infinite geodesics.

\subsection{Organization}
The paper is organized as follows. Section \ref{sec:notation} collects our notation and our assumptions on the probability space. Section \ref{sec:results} introduces the polymer model, rooted Gibbs-DLR measures, and then states our main results. Section \ref{sec:structure} develops the structure theory of the set of Gibbs-DLR measures and its extreme points, including closedness, total ordering, coalescence, and the extension of rooted extremal Gibbs measures to the full lattice. Section \ref{sec:strong} proves strong existence and strong uniqueness of the Busemann process. Sections \ref{sec:L1-cont} and \ref{sec:zero-temp} treat the $L^1$ continuity theorem and the zero-temperature asymptotics of rooted Gibbs-DLR measures. Appendix \ref{app:meas} collects the measurability arguments, and Appendix \ref{app:aux} records some auxiliary results.

\section{Notation and setting}\label{sec:notation}
\subsection{Notation}
Some probability measures and expectations will be distinguished throughout the paper, but a generic expectation under a probability measure $\mu$ is denoted by $\bfE^\mu$.

For $a\in\R$, we write $\bbR_{\geq a}=[a,\infty)$ and $\bbR_{>a}=(a,\infty)$. Then the open first quadrant is $\bbR_{>0}^2 = (0,\infty)^2$. For $k \in \bbZ$, the integers greater than or equal to $k$ are denoted $\bbZ_{\geq k}=\{k,k+1,k+2,\dots\}$.
  Integer intervals are expressed as 
$\lzb m,n\rzb=\{k\in\Z: m\le k\le n\}$. 
Given points in the plane with integer coordinates, $x,y\in\bbZ^2$, we say that $x \leq y$ if $x \cdot e_i \leq y \cdot e_i$ for both $i\in\{1,2\}$. We denote the closed line segment connecting $x$ and $y$ by $[x,y]$. Segments with endpoints removed are denoted by reversing the orientation brackets; for example, $]x,y] = [x,y]\backslash\{x\}$. Given integers $m,n\in\bbZ$, we write $m \vee n$ for the maximum and $m \wedge n$ for the minimum. We similarly denote by $x \vee y$ and $x \wedge y$ the coordinate-wise maximum and minimum of sites $x,y\in\bbZ^2$.

$\cM_1(\sS,\mathscr{B}(\sS))$ denotes the space of Borel probability measures on a metric space $\sS$, endowed with its standard weak topology. At times we denote by $d_{\cM_1}$ a metric generating this topology.

The set of extreme points of a convex set $K$ is denoted by $\ext K.$

\subsubsection{Paths and path spaces}
A path $\pi_{m:n}=(\pi_i)_{i=m}^n$ of points in $\Z^2$ is  \emph{up-right} or \emph{admissible} if it only takes steps in $\{e_1,e_2\}$, meaning $\pi_{i+1} - \pi_{i} \in \{e_1, e_2\}$ for all integers $i \in \lzb m, n-1\rzb$.  This definition extends to semi-infinite and bi-infinite up-right paths. Up-right paths are indexed by antidiagonal levels, that is, 
$\pi_i\cdot(\evec_1+\evec_2)=i$ for all $i$ in the relevant range. 
When $x\in\bbZ^2$ is a generic point, we will at times implicitly denote by $k$ its level  $x\cdot(e_1+e_2)=k$.
 
For $x \leq y$, $\Path{x}{y}$ denotes the collection of up-right paths between $x$ and $y$. Throughout the paper, all paths are taken to be up-right. We therefore omit the ``up-right'' qualifier at times. If $x\in\bbZ^2$, $k=x\cdot(e_1+e_2)$, and $n\ge k$, define
\[\pathsp_{x:n}=\bigcup_{\substack{y \geq x \\  y\cdot(e_1+e_2)=n}}\Path{x}{y}.\]

When $u\le v$ are points on a path $\gamma$,   $\gamma_{u:v}$ is the segment of $\gamma$ from $u$ to $v$, including the endpoints $u$ and $v$. Then $\gamma_{\gamma_m:\gamma_n}$ is abbreviated by $\gamma_{m:n}$. If $m$ is an index and $u\in\gamma$, then mixtures $\gamma_{m:u}$ and $\gamma_{u:m}$ are also entirely unambiguous. 

Fix $k\in\bbZ$. Let $\pathsp_{k:\infty}$ denote the space of semi-infinite up-right paths $\gamma_{k:\infty}$ on $\bbZ^2$ indexed by $\{k,k+1, k+2, \dots\}$= $\Z_{\ge k}$. Equip $\pathsp_{k:\infty}$ with the product-discrete topology, metrized for example by
\[
d_{\pathsp_{k:\infty}}(\gamma,\pi)=\sum_{i=k}^\infty 2^{-(i-k+1)}\one_{\{\gamma_i\ne \pi_i\}} \in [0,1].
\]
Under this metric, $\pathsp_{k:\infty}$ is Polish. For $k\in\Z$ and $x\in\Z^2$ with $x\cdot(e_1+e_2)=k$, $\pathsp_x$ denotes the space of semi-infinite up-right paths $\gamma=\gamma_{k:\infty}$ on $\Z^2$ that start at $\gamma_k=x$. $\pathsp_x$ is a compact metric subspace of $\pathsp_{k:\infty}$. 

Our arguments will at times work with concatenations of paths. Given two paths $\pi_{\ell:m}$ and $\gamma_{m:n}$ with $\pi_m=\gamma_m$, we denote by $\pi \concat \gamma$ the concatenation.

We suppress this point in the sequel, but one should view all of the spaces of paths that we study as closed subspaces of the Polish space  $\pathsp_{-\infty:\infty}$ of $\bbZ^2$-valued paths indexed by times $k\in\bbZ$. This can be achieved by fixing coordinates outside of the relevant index set via a deterministic extension rule. Each such space then carries the subspace topology and the trace Borel $\sigma$-algebra. This is helpful for viewing objects like the $\sigma$-algebras we introduce momentarily as living on the same space.

We denote the canonical process on the path spaces by $X$ with appropriate subscripts. For $m\in\bbZ$, call
\[
\pathsa_{m:\infty}=\sigma(X_m,X_{m+1},\dots) \qquad \text{ and }\qquad \ptail = \bigcap_{m\in\bbZ} \pathsa_{m:\infty} = \bigcap_{m>k} \pathsa_{m:\infty} \text{ for any }k\in\bbZ.\]
Finite path $\sigma$-algebras are defined similarly: if $k<n$, $\pathsa_{k:n}=\sigma(X_k,\dots,X_n)$.

On a product space of two paths, we set
\[
\pathsa^{(2)}_{m:\infty}=\sigma((X^1_m,X^2_m),(X^1_{m+1},X^2_{m+1}),\dots),
\qquad
\ptail^{(2)} = \bigcap_{m\in\bbZ}\pathsa^{(2)}_{m:\infty} =\bigcap_{m>k}\pathsa^{(2)}_{m:\infty}\text{ for any }k\in\bbZ.
\]
On a countable product space of paths, we set
\[
\pathsa^{(\infty)}_{m:\infty}
=
\sigma\bigl((X^i_m)_{i\geq1},(X^i_{m+1})_{i\geq1},\dots\bigr),
\qquad
\ptail^{(\infty)}
=
\bigcap_{m\in\bbZ}\pathsa^{(\infty)}_{m:\infty}
=
\bigcap_{m>k}\pathsa^{(\infty)}_{m:\infty}\text{ for any }k\in\bbZ.
\]

\subsubsection{Shift maps}
For $y\in\bbZ^2$ define the shift map
\be
\theta_y:\pathsp_{k:\infty}\longrightarrow \pathsp_{k+y\cdot(e_1+e_2):\infty}
\qquad\text{by}\qquad
(\theta_y\gamma)_i=\gamma_{i-y\cdot(e_1+e_2)}+y,
\qquad i\ge k+y\cdot(e_1+e_2).
\label{eq:pathshift}
\ee
If $x\cdot(e_1+e_2)=k$, then $\theta_y$ restricts to a map $\theta_y:\pathsp_x\to \pathsp_{x+y}$.

For $\mu\in\sM_1(\pathsp_{k:\infty},\pathsa_{k:\infty})$ define the pushforward 
\be
(\theta_y)_\#\,\mu=\mu\circ\theta_y^{-1}\in\sM_1(\pathsp_{k+y\cdot(e_1+e_2):\infty},\pathsa_{k+y\cdot(e_1+e_2):\infty}).
\label{eq:pathshift-pushforward}
\ee
A measure $\mu$ is \emph{rooted at $x\in \bbZ^2$} if it is supported on $\pathsp_x$, equivalently 
$\mu(X_k=x)=1$. In this case $(\theta_y)_\#\mu$ is rooted at $x+y$.

\subsubsection{Path orders}
We use the level-wise path partial order for paths with a common root: for two up-right paths $\gamma,\pi\in\pathsp_x$ with $k=x\cdot(e_1+e_2)$, 
\[
\gamma \preceq \pi \quad \Longleftrightarrow \quad
\gamma_i\cdot e_1 \le \pi_i\cdot e_1 \ \text{ for all } i\ge k,
\]
and for finite segments $\gamma_{k:m},\pi_{k:m}$ we define $\gamma_{k:m}\preceq \pi_{k:m}$ analogously. We say that $\gamma \precneq \pi$ if $\gamma \preceq \pi$ and $\gamma \neq \pi$. We say that $\gamma \prec \pi$ if $\gamma \preceq \pi$ and there exists an index $N \geq k$ so that $\gamma_{k:N}=\pi_{k:N}$ and for $n > N,$ $\gamma_n \cdot e_1 < \pi_n \cdot e_1$. 

\subsubsection{Orders on measures}
The path orders defined above induce orders on measures via duality. For $k \in \bbZ$, we say that two measures $\Pi^1,\Pi^2 \in \sM_1(\pathsp_{k:\infty},\pathsa_{k:\infty})$ are stochastically ordered, denoted by $\Pi^1 \preceq \Pi^2$, if there exists a coupling $\widetilde{\Pi} \in \sM_1(\pathsp_{k:\infty}\times \pathsp_{k:\infty}, \pathsa_{k:\infty}^{(2)})$ with the property that $\widetilde{\Pi}(X^1 \preceq X^2)=1$
 and for which $\widetilde{\Pi}(X^1\in \aabullet) = \Pi^1(\aabullet)$ and $\widetilde{\Pi}(X^2 \in \aabullet) = \Pi^2(\aabullet)$. This is the ordinary stochastic monotonicity on a Polish space equipped with a closed partial order: by Strassen's theorem, see \cite{Str-65}, this definition is equivalent to the distributional order induced by taking expectations of increasing functions with respect to the path order.  We say that $\Pi^1 \precneq \Pi^2$ if $\Pi^1 \preceq \Pi^2$ and $\Pi^1 \neq \Pi^2$. 

We next introduce a refinement of the usual total order just discussed. We say that measures $\Pi^1,\Pi^2 \in \sM_1(\pathsp_{x:\infty},\pathsa_{k:\infty})$
are strongly ordered, denoted $\Pi^1 \prec \Pi^2$, if there exists a coupling $\wt{\Pi}$ of $\Pi^1$ and $\Pi^2$ as above for which we have $\widetilde{\Pi}(X^1 \prec X^2)=1$. $\Pi^1 \prec \Pi^2$ implies $\Pi^1 \precneq \Pi^2$, but the reverse does not necessarily hold.


Given two path measures with any roots, we say that $\Pi^1$ and $\Pi^2$ admit a coalescing coupling, denoted $\Pi^1 \coal \Pi^2$, if there exists a coupling $\widetilde{\Pi}$ of $\Pi^1$ and $\Pi^2$ under which 
\[\wt{\Pi}(\exists N\in\mathbb{Z}\,:\,X^1_{N:\infty}=X^2_{N:\infty})=1.\]

\subsection{Probability space hypotheses}\label{sec:probsp}
\subsubsection{Setup}\label{sec:setup} We work with a generic quintuple $(\Omhat,\kShat,\Phat,\That,\w)$ consisting of a Polish probability space $(\Omhat,\kShat,\Phat)$, a group $\That=\{\That_x:x\in\Z^2\}$ of continuous bijections of $\Omhat$, and a measurable field of real-valued weights $\w=\w(\what)=\{\w_x(\what):x\in\bbZ^2\}$. In particular, $\That_0$ is the identity map and $\That_x\That_y=\That_{x+y}$ for all $x,y\in\Z^2$. $\Ehat$ denotes expectation under $\Phat$ and $\cIhat$ the $\sigma$-algebra of $\That$-invariant events on $\Omhat$. When there is no risk of confusion, we suppress $\what$ and $\w$ from the notation.

We assume throughout that $\Phat$ is invariant under each $\That_x$ and  that  $\w$ is $\That$-covariant, i.e.,
\begin{align}
\w_x(\That_z\what)=\w_{x+z}(\what) \label{eq:Thatcov}
\end{align}
for all $x,z\in\Z^2$, $\Phat$-almost surely. 

Denote the canonical space of the weights by $(\Omega,\sF) = (\bbR^{\bbZ^2},\mathscr{B}(\bbR^{\bbZ^2}))$. Let $\kS$ denote the $\sigma$-algebra on $(\Omhat,\kShat)$ generated by $\w=\{\w_x:x\in\Z^2\}$. We  use the symbol $\w=(\w_x)_{x\tsp\in\tsp\Z^2}$ to denote  both the $\Omhat\to\Omega$ mapping $\what\mapsto\w(\what)=(\w_x(\what))_{x\tsp\in\tsp\Z^2}$ and a generic element of $\Omega$. We denote by $\P(\acdot)=\Phat\{\what: \w(\what)\in\acdot\}$ the probability measure induced by pushing forward $\Phat$ by the map $\what \mapsto \w(\what)$.

The shifts $T=\{T_x : x \in \Z^2\}$ on $\Omega$ are defined by 
\begin{align}\label{T-cov}
  (T_z\w)_x=\w_{x+z}.
\end{align}
The  shifts on the two spaces are related via the following identity, valid for $x,z\in\bbZ^2$ and $\Phat$-almost all $\what \in \Omhat$:
\begin{align}\label{w-cov}
(T_z[\w(\what)])_x\overset{\eqref{T-cov}}=[\w(\what)]_{z+x}\overset{\rm(def.)}=\w_{z+x}(\what)\overset{\eqref{eq:Thatcov}}=\w_x(\That_z\what).
\end{align}
We need some ergodicity of the environment for our arguments, as recorded in the following condition, assumed throughout the paper.
\begin{condition}\label{cond:ergodic}
We assume that $(\Omega,\sF,\bbP)$ is ergodic under the group of shifts $T=\{T_x : x \in \Z^2\}$.
\end{condition}

For $z\in\bbZ^2$ and $j\in\{1,2\}$, define
\[
S_n^{z,j}=\sum_{i=0}^{n-1}\bigl(\w_{z+ie_j}-\w_{z-e_{3-j}+(i+1)e_j}\bigr),\qquad n\ge 1.
\]
$S_n^{z,j}$ is a one-dimensional random walk with stationary (not necessarily independent) centered increments, constructed by taking differences of weights on adjacent rows $(j=1)$ or columns $(j=2)$. Most of the paper holds on a single event of full probability coming from the following condition, which we assume throughout the entire paper. 

\begin{condition}\label{cond:standing} For each fixed $z\in\bbZ^2$ and $j\in\{1,2\}$,
\be
\Phat\Bigl(\varliminf_{n\to\infty} S_n^{z,j}=-\infty\Bigr)=1.\label{eq:hyp-walk-unbounded}
\ee
\end{condition}
This mild condition holds for example if $S_n^{z,j}$ satisfies a central limit theorem for some $z\in\bbZ^2$ and each $j\in\{1,2\}$. Note also that this condition implies the weights are non-degenerate:
\be
\Phat(\w_0 = \Ehat[\w_0])\neq 1. \label{eq:non-degenerate}
\ee

The pointwise results in the paper either hold deterministically for all $\w\in\Omega$ or on the shift-invariant full-probability event 
\be
\displaystyle\Omega_0 = \bigcap_{z\in\bbZ^2}\bigcap_{j\in\{1,2\}}\{\varliminf_{n\to\infty} S_n^{z,j}=-\infty\}.\label{eq:Omega0-definition}
\ee
Results which involve the limiting free energy or the properties of cocycles will require intersecting with additional full probability events.

\subsubsection{Moment and mixing hypotheses}
When we discuss properties of the limiting free energy (introduced below) and the Busemann process, we will need some additional mixing and moment conditions, which will depend on whether the inverse temperature parameter $\beta\in (0,\infty]$ is finite or infinite. 

\begin{condition}\label{cond:moment-mixing}
For $\beta\in (0,\infty]$, $(\Omhat,\kShat,\Phat,\That,\w)$ satisfies the hypotheses of Section \ref{sec:setup} and
\begin{enumerate}[label={\rm(\alph*)}, ref={\rm\alph*}]
\item\label{cond:moment-mixing-<infty} If $\beta<\infty$, one of the following conditions holds: 
\begin{enumerate}[label={\rm(\roman*)}, ref={\rm\roman*}]
\item \label{cond:moment-mixing-finite-range}
$\{\w_x:x\in\Z^2\}$ have finite range of dependence under $\Phat$ and $\Ehat[|\w_0|^p]<\infty$ for some $p>2$.
\item\label{cond:moment-mixing-ergodic} $\Ehat[|\w_0|]<\infty$, there exists a constant $c>0$ so that $\Phat(\w_0 \leq c)=1$, and $\w$ satisfies for each $j\in\{1,2\}$, $\Phat$-almost surely,
\be
\varlimsup_{\epsilon \searrow 0}\varlimsup_{n\to\infty}\frac{1}{n}\max_{x\in[-n,n]^2} \sum_{0 \leq k \leq \epsilon n} |\w_{x + k e_j}| = 0.\label{eq:class-L}
\ee
\end{enumerate}
\item \label{cond:moment-mixing-iid} If $\beta=\infty$, $\{\w_x:x\in\Z^2\}$ are i.i.d.\ under $\Phat$, and  $\Ehat[|\w_0|^p]<\infty$ for some $p>2$.
\end{enumerate}
\end{condition}

Conditions like \eqref{eq:class-L} are known in the literature as class $\sL$ conditions. Such conditions come from a trade-off between mixing and moment hypotheses on the weights. For example, if $\Phat$ is exponentially mixing under each $\That_z$ for $z\in\bbZ^2\backslash\{0\}$ and if $\Ehat[|\w_0|^p]<\infty$ for some $p>2$, then \eqref{eq:class-L} holds. If $\Phat$ is merely ergodic and $\w_0$ is bounded, then \eqref{eq:class-L} holds. See \cite[Lemma A.4]{Ras-Sep-Yil-13} for more sufficient conditions for membership in class $\sL$. 

\begin{remark}
The reason for the asymmetry of our hypotheses in positive and zero temperature is that finiteness of almost sure Busemann limits coming from an arbitrary fully supported extremal Gibbs measure (defined below) comes as a consequence of backward martingale convergence in positive temperature. The proof of the analogous property for non-degenerate infinite geodesics, recorded as Theorem B.1 in our previous paper \cite{Jan-Ras-Sep-25-strong-}, implicitly used the main result of \cite{Mar-04}, which required an i.i.d. weight assumption.
\end{remark}

\section{Models and main results}\label{sec:results}

\subsection{Directed polymers and last-passage percolation} 
\label{sec:FE_shape}
For sites $x \leq y$ and inverse temperature $\beta\in(0,\infty)$ the point-to-point polymer partition function and free energy are given by
\[
\PF{x}{y}^{\beta,\w} = \sum_{\pi \in \Path{x}{y}} e^{\beta \sum_{v\in\pi \backslash\{y\}}\w_v} \qquad \text{ and }\qquad \FE{x}{y}^{\beta,\w} = \frac{1}{\beta} \log \PF{x}{y}^\beta.
\]
If $x \leq y$ fails, we take the convention that $\PF{x}{y}^{\beta,\w}=0$ and $\FE{x}{y}^{\beta,\w}=-\infty$.  Last-passage percolation can be viewed as the directed polymer at zero temperature. The last-passage time is
\[
\FE{x}{y}^{\infty,\w} = \lim_{\beta\to\infty}\FE{x}{y}^{\beta,\w} = \max_{\pi\in\Path{x}{y}}\biggl\{\sum_{v\in \pi\backslash\{y\}}\w_v\biggr\}.
\]
Maximizers of this variational problem are known as \textit{geodesics}. Semi-infinite or bi-infinite paths are said to be geodesics if each finite subsegment is a geodesic between its endpoints.

For $x \leq y$ and $\beta \in (0,\infty)$, the point-to-point quenched polymer is a measure on $\Path{x}{y}$ defined via
\begin{align}
\poly{x}{y}^{\beta,\w}(\pi) = \frac{e^{\beta \sum_{v \in \pi\backslash\{y\}} \w_v}}{\PF{x}{y}^{\beta,\w}}. \label{eq:polydef}
\end{align}

We denote the closed line segment in $\bbR^2$ connecting the two basis vectors by $\Uset = [e_2,e_1]$. 

\begin{lemma}\label{lem:shape}
{\rm\cite[Theorem 2.7]{Jan-Nur-Ras-22}} 
For $\beta\in(0,\infty]$, assume that Condition \ref{cond:moment-mixing} holds. There exists a deterministic function $\fe^\beta:\bbR_{>0}^2\to\bbR$ called the \textit{limiting free energy} $(\beta<\infty)$ or the \textit{limit shape} $(\beta = \infty)$ such that with $\bbP$-probability one, for each $\delta \in (0,1/2)$
\be\label{sh-th}
\lim_{n\to\infty}\max_{\substack{x\,\in\,n \Uset\tspb\cap\tspb \bbZ_{\geq 0}^2 \\[1pt] x\cdot e_1 \,\in\, [\delta n, (1-\delta) n]}}\frac{\abs{\FE{0}{x}^{\beta,\w}-\fe^\beta(x)}}{n}=0.
\ee
Moreover, $\fe^\beta$ is concave and positively homogeneous of degree one.
\end{lemma}

\begin{remark}\label{rk:extension}
    In general, we always have the inequality 
	\be\fe^{\beta}(e_i) \geq \Ehat[\w_0].\label{eq:fe-ineq}\ee
    The main result of \cite{Mar-04} says if the weights are i.i.d.\ and satisfy $\bbE[|\w_0|^p]<\infty$ for some $p>2$ (in fact, under a slightly weaker moment hypothesis), equality holds in \eqref{eq:fe-ineq} and $\delta$ can be taken to be $0$ in \eqref{sh-th}. \cite[Theorem 3.2]{Ras-Sep-14} gives the same equality under the same condition  if $\beta<\infty$. 
     See also Lemma C.1 in \cite{Jan-Ras-18-arxiv}. 
     Then, Theorem 2.15 in \cite{Gro-Jan-Ras-25-tams} gives again \eqref{sh-th} with $\delta=0$. 
	
    By Theorem 10.3 in \cite{Roc-70}, $\fe^\beta$ can be extended continuously from $\bbR^2_{>0}$ to $\bbR_{\geq 0}^2$. This extension is equal to the upper semi-continuous regularization.  See the second paragraph of Section \ref{app:convex}. We will implicitly work with this extension in what follows. Note that models satisfying Condition \ref{cond:moment-mixing}\eqref{cond:moment-mixing-<infty} for which $\fe^\beta(e_i) \neq \Ehat[\w_0]$ for $i\in\{1,2\}$ exist. 
\end{remark}

 Homogeneity implies that $\fe^\beta$ is determined by its restriction to $\Uset$. The super-differential of $\fe^\beta$ at $\xi\in\R_{\geq 0}^2\backslash\{0\}$ is
\be\label{eq:superdef}
\partial \fe^\beta(\xi)=\bigl\{h\in\R^2:\fe^\beta(\zeta)-\fe^\beta(\xi)\le h\cdot(\zeta-\xi)\text{ for all }\zeta\in\R_{\geq 0}^2\bigr\}.
\ee
By homogeneity, $\partial \fe^\beta(\xi)=\partial \fe^\beta(c\xi)$ for any $c>0$. Thus $\partial \fe^\beta(\bbullet)$ is also determined by points on $\Uset$, and we write
\[
\partial \fe^\beta(\Uset)=\{h\in\R^2:\text{there exists }\xi\in\Uset\text{ with }h\in\partial \fe^\beta(\xi)\}.
\]
Concavity implies the existence of real-valued one-sided derivatives at relative interior points $\xi\in\ri\Uset=]e_2,e_1[$:
\begin{align*}
\nabla \fe^\beta(\xi \pm) \cdot e_1 &= \lim_{\e \searrow 0} \frac{\fe^\beta(\xi \pm \e e_1) - \fe^\beta(\xi)}{\pm \e}\quad\text{and}\\
\nabla \fe^\beta(\xi \pm) \cdot e_2 &= \lim_{\e \searrow 0} \frac{\fe^\beta(\xi \mp \e e_2) - \fe^\beta(\xi)}{\mp \e}.
\end{align*}
Differentiability of $\fe^\beta$ at $\xi\in\ri\Uset$ is equivalent to $\nabla \fe^\beta(\xi+)=\nabla \fe^\beta(\xi-)$, and more generally $\partial \fe^\beta(\xi) = [\nabla \fe^\beta(\xi+),\nabla \fe^\beta(\xi-)]$ and consequently $\ext \partial \fe^{\beta}(\xi)=\{\nabla \fe^\beta(\xi+),\nabla \fe^\beta(\xi-)\}$.

In general, $\partial \fe^\beta(\Uset)$ forms a one dimensional curve. In the case of i.i.d.~last-passage percolation, this comes from combining equation (4.14) in \cite{Geo-Ras-Sep-17-ptrf-1} with Lemma 4.6(c) in \cite{Jan-Ras-20-aop}; the result is also true more generally. The connection to level sets of the free energy is illustrated in Figure \ref{fig:shape}, which previously appeared in \cite{Jan-Ras-Sep-25-strong-}. 

\begin{figure}
\includegraphics[scale=.5]{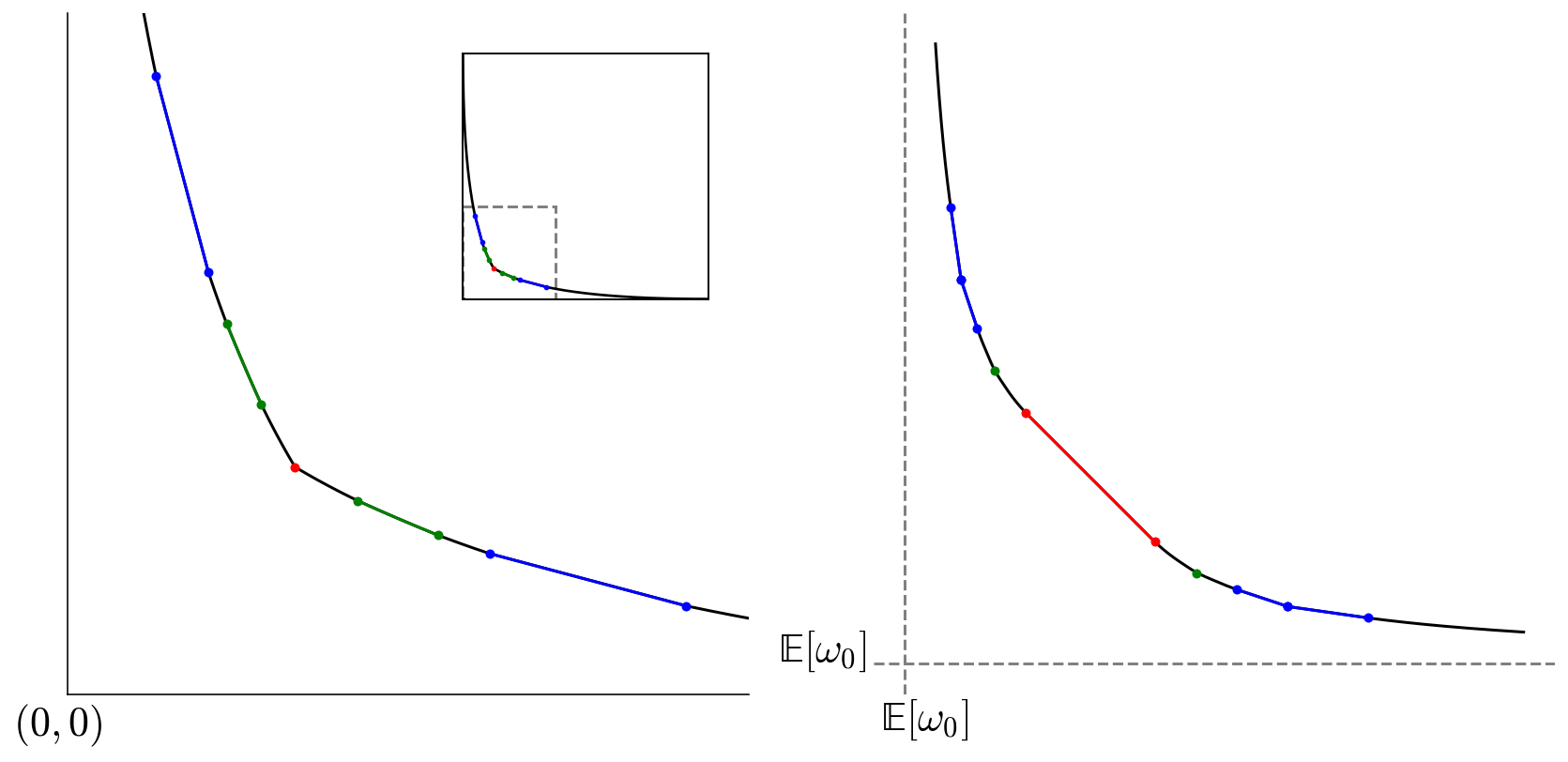}
\caption{\small A level set $\{x\in\R_{\geq 0}^2:\fe^\beta(x)=1\}$ of a possible limiting free energy (left) and its associated 
super-differential curve $\partial \fe^\beta(\Uset)$ (right). The full level set of the free energy is depicted in the inset image on the left and the dashed lines indicate the enlarged portion. 
This free energy has four linear segments and a cusp on the diagonal.  It is differentiable at the endpoints of the two inner segments (green)  and non-differentiable at five points: the four endpoints of the outer segments (blue) and the cusp on the diagonal (red).}\label{fig:shape}
\end{figure}

\subsection{Rooted Gibbs-DLR measures}
\begin{definition}
For $\beta\in(0,\infty)$ and $x \in\bbZ^2$ with $x\cdot(e_1+e_2)=k$, a measure $\Pi$ on $(\pathsp_{x},\pathsa_{k:\infty})$ is a Gibbs-DLR measure rooted at x at inverse temperature $\beta$ if
\be
\Pi(x_{k:n}) = \poly{x}{x_n}^{\beta,\w}(x_{k:n})\Pi(x_n)\label{eq:Gibbs}
\ee
for all $n\ge k$ and admissible paths $x_{k:n}$ with $x_k=x.$ The set of Gibbs-DLR measures rooted at $x$ at inverse temperature $\beta$ is denoted $\DLR_x^{\beta,\w}$.
\end{definition}    
It is immediate that the Gibbs property makes a Gibbs-DLR measure a Markov chain. It also follows immediately from the definitions that the set of Gibbs-DLR measures enjoys the following shift covariance: for all $\w\in\Omega$, $\beta\in(0,\infty)$,
and $x,y\in\bbZ^2$, a measure $\Pi \in \sM_1(\pathsp_x,\pathsa_{k:\infty})$ satisfies
\be
(\theta_y)_{\#}\,\Pi\in \DLR_{x+y}^{\beta,\w} \text{ if and only if }\Pi\in \DLR_x^{\beta,T_y\w}.\label{eq:DLR-eq-covariant}
\ee
Consequently,
\be
\DLR_{x+y}^{\beta,\w}\;=\;(\theta_y)_\#\,\DLR_x^{\beta,T_y\w}\label{eq:DLR-set-covariant}
\ee

\begin{definition}
$\Pi\in\sM_1(\pathsp_x,\pathsa_{k:\infty})$ is said to be \textit{fully supported} if for all admissible paths $x_{k:n}$ with $x_k=x$, $\Pi(x_{k:n})>0$.
\end{definition}

Because of the path structure, $\DLR_x^{\beta,\w}$ always contains two \textit{trivial} rooted Gibbs-DLR measures at $x$. These are the measures that  assign probability one to one of the paths $x + e_1 \bbZ_{\geq 0}$ and $x + e_2 \bbZ_{\geq 0}$. The following result establishes the main property we need on the event $\Omega_0$ from \eqref{eq:Omega0-definition}, namely that the only Gibbs measures which are not fully supported are convex combinations of trivial Gibbs measures. A slightly weaker version of this result for each fixed $\beta$ previously appeared in \cite[Lemma 3.4]{Jan-Ras-20-aop} under an i.i.d.\ hypothesis. The proof in our setting is essentially the same as there and appears at the start of Section \ref{sec:structure}.

\begin{lemma}\label{lem:Omega0}
The following holds for all $\w \in \Omega_0$, $x\in\bbZ^2$, and $\beta\in(0,\infty)$. If $\Pi\in\DLR_x^{\beta,\w}$, $v\ge x$, $j\in\{1,2\}$, and $(v-x)\cdot e_{3-j}>0$, then
\[
\Pi\bigl(X_{v\cdot(e_1+e_2)+m}=v+me_j\text{ for all }m\ge 0\bigr)=0.
\]
Consequently, if $\Pi \in \DLR_x^{\beta,\w}$ is not fully supported, then $\Pi$ is a convex combination of the two trivial Gibbs measures. That is, there exists $\alpha \in [0,1]$ with $\Pi = \alpha\tspb \delta_{x+e_1\bbZ_{\geq 0}} + (1-\alpha)\tspb \delta_{x + e_2\bbZ_{\geq 0}}$.
\end{lemma}
\begin{remark}\label{rem:degen} 
Lemma \ref{lem:Omega0} plays a central role in this paper. In last-passage percolation, it is possible to construct an ergodic environment in which non-trivial semi-infinite geodesics exist that become trapped on a (random) row or column of the lattice. A simple example comes by applying Lemma 5.4 of \cite{Emr-Jan-Sep-25} to the model where, in the notation of that paper, $b_j = 1$ and $a_i$ is an i.i.d.\ sequence taking the values $1$ or $2$ with equal probability. In that model, the critical direction is $e_2$ and the $e_2$ labeled geodesic from any site of the lattice behaves non-trivially until it first encounters a column with $a_i =1$, at which point it proceeds vertically. This environment fails Condition \ref{cond:standing}. We expect that one can prove that the same type of example failing the conclusions of Lemma \ref{lem:Omega0} exists in positive temperature by applying the methods of \cite{Emr-Jan-Sep-25} to the solvable inhomogeneous generalization of the log-gamma polymer. 
\end{remark}

It is immediate to see that for each $\w \in \Omega$, the set of measures $\Pi$ satisfying \eqref{eq:Gibbs} is closed and convex in the compact metric space $\sM_1(\pathsp_x,\pathsa_{k:\infty})$. Choquet's theorem allows us to write each Gibbs measure as an integral convex mixture of extremal Gibbs measures, so it will suffice to focus our attention on the extreme points.

\begin{definition}
$\Pi \in \DLR_x^{\beta,\w}$ is extreme if for all $\Pi^{1},\Pi^{2}\in \DLR_x^{\beta,\w}$ and $\alpha \in (0,1)$, $\Pi = \alpha \Pi^{1} + (1-\alpha) \Pi^{2}$ implies that $\Pi^{1}=\Pi^{2}=\Pi$.
The set of extreme points of $\DLR_x^{\beta,\w}$ is denoted by $\ext \DLR_x^{\beta,\w}$. 
\end{definition}
By Lemma \ref{lem:Omega0}, all non-trivial extremal Gibbs measures are fully supported. Existence of such measures under Condition \ref{cond:moment-mixing} will be discussed later and follows from one of the main results of \cite{Gro-Jan-Ras-25-tams}. If the weights are taken to be i.i.d., this was previously shown in \cite{Jan-Ras-20-aop}.

Recall the partial order given by stochastic monotonicity which comes from the path structure. As described in the notation section, we denote this order by $\preceq$: for $\Pi^1,\Pi^2\in\sM_1(\pathsp_x,\pathsa_{k:\infty})$, 
\[
\Pi^1 \preceq \Pi^2 \iff \Pi^1(A) \leq \Pi^2(A)
\]
for every increasing event $A\in\pathsa_{k:n}$, for all $n>k$. By Strassen's theorem \cite{Str-65}, this is equivalent to the coupling definition of the order given in the notation section. Our first new result is that the set of extremal Gibbs measures is both closed and totally ordered.

\begin{theorem}\label{thm:cltot}
For each $\w \in \Omega_0$, each $x\in\bbZ^2$, and each $\beta\in(0,\infty)$, $\ext \DLR_x^{\beta,\w}$ is closed in $\sM_1(\pathsp_x,\pathsa_{k:\infty})$ and totally ordered under $\preceq$.
\end{theorem}

The proof of this result is broken up into two claims below, recorded as Propositions \ref{prop:DLR-ordered} and \ref{prop:DLR-closed}.

\begin{remark}
In general, the set of extreme points of a metrizable compact convex set is $G_{\delta}$ \cite[Proposition 1.3]{Phe-01}. We do not know of an \textit{a priori} reason to expect $\ext \DLR_x^{\beta,\w}$ to be closed except by way of analogy to the structure of geodesics in first- and last-passage percolation (which is how we \emph{guessed} this property). That same analogy may suggest that $\ext \DLR_x^{\beta,\w}$ continues to be closed in polymers on non-planar graphs (in particular, in higher dimensions), but our proof uses planarity in an essential way. An alternative proof which avoids planarity would be of interest.
\end{remark}

\subsubsection{Coalescing couplings and extremality}
We next turn to a characterization of extremality of fully supported   Gibbs measures, beginning with some definitions.

\begin{definition}\label{def:coalcoup-pair}
Fix finitely many sites $x^1,\dots, x^m\in\bbZ^2$ and abbreviate $k_i = x^i\cdot(e_1+e_2)$. We say that a family $\{\Pi^1,\dots \Pi^m\}$, $\Pi^i\in\sM_1(\pathsp_{x^i},\pathsa_{k_i:\infty})$, $i\in\lzb 1, m \rzb$, admits a coalescing coupling if there exists a measure \[\wt \Pi \in \sM_1\left(\prod_{i=1}^m\pathsp_{x^i}, \prod_{i=1}^m\pathsa_{k_i:\infty}\right)\] 
with the following properties, where we denote by $X^i$ the $i^{\mathrm{th}}$ entry of the coordinate process:
\begin{enumerate}[label={\rm(\alph*)}, ref={\rm\alph*}]
\item {\rm(}Marginals{\rm)} For each $i\in \lzb 1, m \rzb$, $\wt \Pi(X^i \in \aabullet) = \Pi^i(\aabullet).$
\item {\rm(}Coalescence{\rm)} \ For each $i\neq j$, $\wt \Pi(\text{there exists }n \in\bbZ \text{ such that } X^i_{n:\infty}=X^j_{n:\infty})=1$.
\end{enumerate}
The definition for countably infinitely many measures is analogous, with $m$ in the expressions above replaced by $\infty$.
\end{definition}

By a straightforward disintegration-of-measure argument combined with the Kolmogorov extension theorem, existence of a coalescing coupling of any countable  collection of measures is equivalent to existence of coalescing couplings of each pair of those measures.

We next recall a basic fact from the extended arXiv version \cite{Jan-Ras-18-arxiv} of article \cite{Jan-Ras-20-aop}.  In words, this says that a fully supported rooted Gibbs-DLR measure
generates fully supported rooted Gibbs-DLR measures from each site in the quadrant ahead of the root. 

\begin{lemma}{\rm\cite[Lemma 2.9(a)]{Jan-Ras-18-arxiv}} \label{lem:cond-DLR}
Let $\w \in \Omega$, $\beta>0$, and $y \geq x$ in $\bbZ^2$. Suppose $\Pi\in \DLR_{x}^{\beta,\w}$ is fully supported. Define a measure $\Pi_y$ on $(\pathsp_y,\pathsa_{\ell:\infty})$, where $\ell=y\cdot(e_1+e_2)$, via 
\be
\Pi_y(A) = \Pi(A \viiva y) = \frac{\Pi(A,X_\ell=y)}{\Pi(X_\ell=y)},\label{eq:induced-DLR-def}
\ee
for $A \in \pathsa_{\ell:\infty}$. Then $\Pi_x = \Pi$ and $\Pi_y\in \DLR_y^{\beta,\w}$ is fully supported. 
\end{lemma}

\begin{remark}In this section and in the proofs of results in this section, we cite several results from \cite{Jan-Ras-20-aop}, which has an i.i.d.\ assumption on the environment. The results we cite here are deterministic statements which do not depend on that hypothesis. We will not repeat this point again in later citations to that paper.
\end{remark}

With the previous result in mind, we can introduce the notion of two rooted Gibbs-DLR measures being \textit{consistent} if the measures they induce agree on the intersections of the quadrants ahead of their roots.
\begin{definition}\label{def:consistent}
For $\w\in\Omega$ and $\beta\in(0,\infty)$, we say two fully supported rooted Gibbs-DLR measures $\Pi^1\in \DLR_x^{\beta,\w}$ and $\Pi^2\in\DLR_y^{\beta,\w}$, $x,y\in\bbZ^2$, are consistent if $\Pi^1_{x \vee y} = \Pi^2_{x \vee y}$.  

An arbitrary family of fully supported rooted Gibbs-DLR measures is said to be consistent if every pair is consistent.
\end{definition}

As will appear in the sequel, it is immediate from Lemma \ref{lem:cond-DLR} that if $\Pi^1$ and $\Pi^2$ are as above and are consistent, then for all $z \geq x \vee y$, $\Pi^1_z = \Pi^2_z$.

As an immediate consequence of the Gibbs property, the next lemma shows that the family of measures generated by a single fully supported Gibbs measure is consistent in the sense defined above.

\begin{lemma}\label{lem:quad-consist}
For each $\w \in \Omega$, $\beta>0$, and $z \geq y \geq x$ in $\bbZ^2$, if $\Pi\in \DLR_{x}^{\beta,\w}$ is fully supported, then
\[
\Pi_z(A) = \Pi_y(A \viiva z)
\]
for all $A \in \pathsa_{m:\infty},$ where $m=z\cdot(e_1+e_2)$. Consequently, the family $\{\Pi_y : y \geq x\}$ is consistent.  
\end{lemma}
\begin{proof}
Abbreviate $\ell=y\cdot(e_1+e_2)$ and fix $A$ as in the statement. Then first by the definition \eqref{eq:induced-DLR-def} of $\Pi_y$ and then by the Gibbs property,
\[
\Pi_y(A \viiva z) = \frac{\Pi_y(A, X_{m}=z)}{\Pi_y(z)} = \Pi(A \viiva z) = \Pi_z(A). \qedhere
\]
\end{proof}

\begin{definition}\label{def:coalcoup}
For $\beta>0$, we say that a fully supported measure $\Pi\in \DLR_x^{\beta,\w}$ \textit{admits a coalescing self-coupling} if the family $\{\Pi_y : y \geq x\}$ defined via \eqref{eq:induced-DLR-def} admits a coalescing coupling.
\end{definition}
Our characterization in Theorem \ref{thm:coalext} below says that a fully supported rooted Gibbs-DLR measure admits a coalescing self-coupling if and only if it is extremal. As the proof in Section \ref{sec:coal-ext} will show, this is a consequence of a celebrated result of Goldstein \cite{Gol-79} on maximal couplings. That result says that two discrete time stochastic processes taking values in a Polish space agree on the path tail $\sigma$-algebra if and only if it is possible to couple them so that they coalesce; see Theorem 2.1 of \cite{Gol-79}. 

\begin{theorem}\label{thm:coalext}
Suppose that for $\w\in \Omega$ and $\beta>0$, $\Pi \in \DLR_x^{\beta,\w}$ is fully supported. Then $\Pi$ admits a coalescing self-coupling if and only if $\Pi \in \ext \DLR_x^{\beta,\w}.$ 
\end{theorem}

\begin{remark}\label{rem:extreme}
As suggested by the discussion prior to the statement of Theorem \ref{thm:coalext}, this result generalizes considerably. In particular, the proof will apply to rooted Gibbs measures of random walks in random potentials in all dimensions, as studied for example in \cite{Gro-Jan-Ras-25-tams}. 
\end{remark}

Combining the previous result with Lemmas \ref{lem:cond-DLR} and \ref{lem:quad-consist}, we have the following corollary.

\begin{corollary}\label{cor:quadrant-extreme}
Suppose that for $\w\in \Omega$ and $\beta>0$, $\Pi \in \ext \DLR_x^{\beta,\w}$ is fully supported. Then for all $y \geq x$, $\Pi_y$ is fully supported and $\Pi_y \in \ext \DLR_y^{\beta,\w}$.
\end{corollary}

\subsection{Equivalence classes of extremal Gibbs measures and the global total order}
We next extend the construction of consistent measures in a quadrant in Lemma \ref{lem:quad-consist} to a construction of consistent measures indexed by all lattice sites. A main step in this construction is an application of the martingale convergence theorem. This may be viewed as the probabilistic analogue of the zero-temperature geometric fact that passage time differences, computed with respect to a sequence of points tending to infinity along a semi-infinite geodesic, are monotone. There, the resulting limit, known as a Busemann function, generates semi-infinite geodesics from every initial point.

We recall the following characterization of a rooted Gibbs-DLR measure in terms of a cocycle on a quadrant, i.e., a function $\Bus(u,v)$ defined for $u,v \geq x$ which satisfies for $u,v,w \geq x$, $\Bus(u,v) + \Bus(v,w) = \Bus(u,w)$. 

\begin{theorem}{\rm\cite[Theorem 5.2]{Jan-Ras-20-aop}}\label{thm:DLR-cocycle}
For $\beta\in(0,\infty)$, $\w\in\Omega$, and $x\in\bbZ^2$, a fully supported measure
$\Pi \in \sM_1(\pathsp_x,\pathsa_{k:\infty})$ satisfies $\Pi \in \DLR_x^{\beta,\w}$ if and only if the real-valued cocycle
$\{\Bus^{\Pi}(u,v) : u,v\geq x\}$ defined by
\be
e^{-\beta\Bus^{\Pi}(u,v)}
=
\frac{\Pi(v)\,\PF{x}{u}^{\beta,\w}}{\Pi(u)\,\PF{x}{v}^{\beta,\w}}
\label{eq:Bus-cocycle-def}
\ee
satisfies, for all admissible paths $x_{k:n}$ with $x_k=x$,
\be
\Pi(x_{k:n}) = e^{\beta \sum_{m=k}^{n-1} \w_{x_m} - \beta \Bus^{\Pi}(x,x_n)}.\label{eq:Bus-cocycle-measure}
\ee
In particular, the transition probabilities of the Markov chain $\Pi$ are given by 
\be
\pi^{\Pi}_{y,y+e_i} = e^{\beta \w_{y} - \beta \Bus^{\Pi}(y,y+e_i)},\quad y\in\Z^2_{\ge x},\ i\in\{1,2\}.\label{eq:DLR-transition}
\ee
Consequently, if $\{\Bus(u,v):u,v\in\Z^2\}$ is a real-valued cocycle satisfying
\[\sum_{i=1}^2 e^{\beta \w_y - \beta \Bus(y,y+e_i)}=1\quad\forall y\in\Z^2,\]
then for all $x\in\Z^2$, $\Pi^{B,\w}_x$, defined for up-right paths $x_{k:n}$ with $x_k=x$ by 
\[\Pi^{\Bus,\w}(x_{k:n}) = e^{\beta \sum_{m=k}^{n-1} \w_{x_m} - \beta \Bus(x,x_n)},\]
is a fully supported element of $\DLR_x^{\beta,\w}$ and the family is consistent.
\end{theorem}

Using Corollary \ref{cor:quadrant-extreme} along with the total ordering of extremal Gibbs measures 
with a common root (Theorem \ref{thm:cltot})
allows us to define a total order on \textit{all} extremal Gibbs measures, from all roots, simultaneously.  The key tool will be the following fundamental backward martingale associated to a fully supported Gibbs measure. 
\begin{lemma} {\rm\cite[Lemma 5.6]{Jan-Ras-20-aop}}\label{lem:DLR-mg}
For all $\w\in\Omega$, $\beta\in(0,\infty)$, all $x,y\in\bbZ^2$, and all fully supported $\Pi \in \DLR_x^{\beta,\w}$, abbreviating $m = (x \vee y) \cdot(e_1+e_2)$, we have that  $(\PF{y}{X_n}^{\beta,\w}/\PF{x}{X_n}^{\beta,\w} : n\geq m)$ is a $\Pi$-backward martingale with respect to the backward filtration $(\pathsa_{n:\infty} :  n \geq m)$.
\end{lemma}

With this martingale in mind, the following proposition is our key tool for showing that a fully supported extremal rooted Gibbs-DLR measure generates an extremal Gibbs-DLR measure from every vertex on the lattice, and these form  a consistent family. 

\begin{proposition}\label{prop:global-ext}
For $\w\in\Omega_0$, $\beta\in(0,\infty)$, and $x\in\bbZ^2$, if $\Pi\in \ext \DLR_x^{\beta,\w}$ is fully supported, then for all $y,z\in\bbZ^2$,  the Busemann limit below exists $\Pi$-almost surely and is finite: 
\be
\Busgeo_\Pi(\w,y,z) = \lim_{n\to\infty} \frac{1}{\beta}\bigl(\log \PF{y}{X_n}^{\beta,\w} - \log \PF{z}{X_n}^{\beta,\w}\bigr) =\lim_{n\to\infty} \frac{1}{\beta}\Bigl(\log \bfE^\Pi\Bigl[\frac{\PF{y}{X_n}^{\beta,\w}}{\PF{x}{X_n}^{\beta,\w}}\Bigr]-\log \bfE^\Pi\Bigl[\frac{\PF{z}{X_n}^{\beta,\w}}{\PF{x}{X_n}^{\beta,\w}}\Bigr]\Bigr).\label{eq:global-ext-Busemann}
\ee

\begin{enumerate}[label={\rm(\roman*)}, ref={\rm\roman*}]
\item \label{prop:global-ext:cocrev} The field $\{\Busgeo_{\Pi}(y,z) : y,z\in\bbZ^2\}$ satisfies:
\begin{enumerate}[label={\rm(\alph*)}, ref={\rm\alph*}]
\item {\rm(}Cocycle{\rm)} For all $w,y,z\in\bbZ^2$ $\Busgeo_\Pi(w,y) + \Busgeo_{\Pi}(y,z) = \Busgeo_{\Pi}(w,z)$.
\item {\rm(}Recovery{\rm)} For all $y \in \bbZ^2$,
\be
\sum_{i=1}^2e^{\beta \w_{y}- \beta \Busgeo_{\Pi}(y,y+e_i)}=1.\label{eq:Busgeo-recovery}
\ee
\end{enumerate}
\item \label{prop:global-ext:consistent-Gibbs} {\rm(}Consistent extremal Gibbs measures{\rm)} For each $z \in \bbZ^2$, abbreviate $\ell=z\cdot(e_1+e_2)$ and define $\cDLR_z^{\Pi} \in \sM_1(\pathsp_z, \pathsa_{\ell:\infty})$ by setting, for an admissible path $x_{\ell:n}$ with $n \geq \ell$ and $x_\ell=z$,
\be
\cDLR_z^{\Pi}(x_{\ell:n}) = e^{\beta\sum_{k=\ell}^{n-1}\w_{x_k} - \beta \Busgeo_\Pi(z,x_n)}.\label{eq:cDLR-def}
\ee
Then $\cDLR_z^{\Pi}\in\ext \DLR_z^{\beta,\w}$, $\cDLR_z^{\Pi}$ is fully supported, and the family $\{\cDLR_z^{\Pi} : z\in\bbZ^2\}$ is consistent.
\item \label{prop:global-ext:agree} {\rm(}Agreement with $\Pi${\rm)} For $z \geq x$, $\cDLR_z^{\Pi} = \Pi_z$.
\item \label{prop:global-ext:Busagree} {\rm(}Agreement of Busemann limits{\rm)} For all $z \in \bbZ^2$, $\Busgeo_{\Pi} = \Busgeo_{\cDLR_z^{\Pi}}.$
\end{enumerate}
\end{proposition}
\begin{remark}
By Proposition \ref{prop:DLR-transition}, the expression in the last limit in \eqref{eq:global-ext-Busemann} is eventually constant as a function of $n$.
\end{remark}
\begin{remark}
As with Theorem \ref{thm:coalext}, the proof of Proposition \ref{prop:global-ext} is dimension independent and applies to the general positive-temperature RWRP model studied in \cite{Gro-Jan-Ras-25-tams}; see Remark \ref{rem:extreme}.
\end{remark}

Fully supported rooted extremal Gibbs measures are partitioned into equivalence classes by declaring $\Pi\in\ext \DLR_x$ and $\Pi'\in\ext \DLR_y$ to be equivalent whenever $\Busgeo_{\Pi}=\Busgeo_{\Pi'}$. Proposition \ref{prop:global-ext} then shows that, for every $x\in\Z^2$, the set of fully supported extremal Gibbs measures rooted at $x$ contains exactly one representative from each equivalence class. Our next few results describe the structure of these equivalence classes. 

\begin{theorem}\label{thm:global-weakorder}
For $\w\in\Omega_0$, $\beta\in (0,\infty)$, and $x,y\in\bbZ^2$ suppose that $\Pi^1\in \ext \DLR_x^{\beta,\w}$ and $\Pi^2\in \ext\DLR_y^{\beta,\w}$ are both fully supported. Then the following are equivalent:
\begin{enumerate}[label={\rm(\roman*)}, ref={\rm\roman*}]
\item\label{thm:global-weakorder:Pi} $\Pi^1_{x \vee y} \preceq \Pi^2_{x \vee y}$.
\item \label{thm:global-weakorder:cDLR-exist} There exists $z \in \bbZ^2$ for which $\cDLR_z^{\Pi^1} \preceq \cDLR_z^{\Pi^2}$.
\item\label{thm:global-weakorder:Busgeo} We have
\begin{align}\label{A1<A2}
\forall z \in \bbZ^2\,:\ \ 
\Busgeo_{\Pi^1}(z,z+e_1) \geq \Busgeo_{\Pi^2}(z,z+e_1) \quad \text{and} \quad \Busgeo_{\Pi^1}(z,z+e_2) \leq \Busgeo_{\Pi^2}(z,z+e_2).
\end{align}
\item \label{thm:global-weakorder:cDLR-all} For all $z \in \bbZ^2$, $\cDLR_z^{\Pi^1} \preceq \cDLR_z^{\Pi^2}$.
\end{enumerate}
\end{theorem}

The next theorem records a more refined characterization of measures being strictly ordered than what follows immediately from the previous result. A particular consequence is that two fully supported extremal Gibbs-DLR measures are strictly ordered if and only if all of their transition probabilities differ. A version of this consequence for the direction-indexed Busemann process previously appeared as Theorem 3.2 in \cite{Bat-Fan-Sep-25}. With Theorem \ref{thm:global-weakorder} in mind, the non-trivial implication that needs to be proven in  the next theorem is that \eqref{thm:global-strictorder:cDLR-exist}
 implies \eqref{thm:global-strictorder:Busgeo}. We include the full equivalence for ease of reference.
\begin{theorem}\label{thm:global-strictorder}
For $\w\in\Omega_0$, $\beta\in (0,\infty)$, and $x,y\in\bbZ^2$ suppose that $\Pi^1\in \ext \DLR_x^{\beta,\w}$ and $\Pi^2\in \ext\DLR_y^{\beta,\w}$ are both fully supported. Then the following are equivalent:
\begin{enumerate}[label={\rm(\roman*)}, ref={\rm\roman*}]
\item \label{thm:global-strictorder:Pi} $\Pi^1_{x \vee y} \precneq \Pi^2_{x \vee y}$.
\item \label{thm:global-strictorder:cDLR-exist} There exists $z \in \bbZ^2$ for which $\cDLR_z^{\Pi^1} \precneq \cDLR_z^{\Pi^2}$.
\item \label{thm:global-strictorder:Busgeo} For all $z \in \bbZ^2$,
\[
\Busgeo_{\Pi^1}(z,z+e_1) > \Busgeo_{\Pi^2}(z,z+e_1) \quad \text{and} \quad \Busgeo_{\Pi^1}(z,z+e_2) < \Busgeo_{\Pi^2}(z,z+e_2).
\]
\item \label{thm:global-strictorder:cDLR-all} For all $z \in \bbZ^2$, $\cDLR_z^{\Pi^1} \precneq \cDLR_z^{\Pi^2}$.
\end{enumerate}
\end{theorem}

Relation \eqref{A1<A2} defines a partial order on functions $\Z^2\times\Z^2\to\R$ and, in particular, on cocycles.

\begin{definition}\label{def:cocycle-order}
Given two cocycles $A^1$ and $A^2$, we say that
\begin{enumerate}[label={\rm(\roman*)}, ref={\rm\roman*}]
\item $A^1 \preceq A^2$ if for all $z\in\bbZ^2$, 
\[
A^1(z,z+e_1) \geq A^2(z,z+e_1)\quad\text{and}\quad A^1(z,z+e_2) \leq A^2(z,z+e_2)\]
\item $A^1 \prec A^2$ if for all $z\in\bbZ^2$, 
\[A^1(z,z+e_1) > A^2(z,z+e_1)\quad\text{and}\quad A^1(z,z+e_2) < A^2(z,z+e_2).\]
\end{enumerate}
$A^1 = A^2$ is equivalent to  $A^1 \preceq A^2$ and $A^2 \preceq A^1$. We write $A^1 \precneq A^2$ if $A^1 \preceq A^2$ and $A^1 \neq A^2$.
\end{definition}

A consequence of Theorems \ref{thm:global-weakorder} and \ref{thm:global-strictorder} is that if $\Pi^1$ and $\Pi^2$ are fully supported rooted extremal Gibbs measures, $\Busgeo_{\Pi^1} \precneq \Busgeo_{\Pi^2}$ is equivalent to $\Busgeo_{\Pi^1} \prec \Busgeo_{\Pi^2}$. This is not the case in zero temperature models. Combining this observation with Proposition \ref{prop:global-ext}, Theorem \ref{thm:global-weakorder}, and Theorem \ref{thm:global-strictorder}, we have the following trichotomy, which one can view as analogous to the asymptotic total order on geodesics in zero temperature, studied in \cite{Jan-Ras-Sep-25-strong-}. 

\begin{corollary}\label{cor:as-order-trichotomy}
For $\w \in \Omega_0$, $\beta\in(0,\infty)$, $x,y\in\bbZ^2$, and fully supported $\Pi^x\in\ext\DLR_x^{\beta,\w}$ and $\Pi^y\in\ext\DLR_y^{\beta,\w}$, exactly one of the following occurs:
\[
\Busgeo_{\Pi^x}\prec \Busgeo_{\Pi^y}\qquad \text{ or }\qquad \Busgeo_{\Pi^x}= \Busgeo_{\Pi^y}\qquad \text{ or }\qquad\Busgeo_{\Pi^y}\prec \Busgeo_{\Pi^x}.
\]
\end{corollary}

The next proposition records the fact that equality in the trichotomy is equivalent to coalescence and consistency.
\begin{proposition}\label{prop:coal-cond}
For $\w\in\Omega_0$, $\beta\in(0,\infty)$, and $V \subset \bbZ^2$, let $\Pi^x\in\ext\DLR_x^{\beta,\w}$ be fully supported for each $x\in V$. 
The following three conditions are equivalent:
\begin{enumerate}[label={\rm(\roman*)}, ref={\rm\roman*}, series = coal-cond]
\item \label{prop:coal-cond:cons} The family $\{\Pi^x : x\in V\}$ is consistent.
\item \label{prop:coal-cond:coal} The family $\{\Pi^x : x\in V\}$ admits a coalescing coupling.
\item \label{prop:coal-cond:Busgeo} For all $x,y\in V$, $\Busgeo_{\Pi^x}= \Busgeo_{\Pi^y}$.
\end{enumerate}
\end{proposition}

\subsection{Cocycle strong existence and uniqueness} \label{sec:coc-existunique}
We next discuss existence and uniqueness of certain kinds of fully supported rooted extremal Gibbs measures. This will require working under Condition \ref{cond:moment-mixing}.

We begin by discussing the primary tools used to construct such measures, known in the literature under various names, including generalized Busemann functions, stationary correctors, and shift-covariant recovering $L^1$ cocycles.

\begin{definition}\label{def:recovcoc}
Fix $\beta\in(0,\infty]$.
A measurable function $\Bhat:\Omhat\times\bbZ^2\times\bbZ^2\to\bbR$ is called a shift-covariant $\beta$-recovering $L^1(\Omhat,\kShat,\Phat)$ cocycle if it satisfies the following properties:
\begin{enumerate}[label={\rm(\alph*)}, ref={\rm\alph*}]
\item \textup{(Shift covariance)} $\Phat$-almost surely, for each $x,y,z\in\bbZ^2$,
\[
\Bhat(x+z,y+z) = \Bhat(x,y)\circ \That_z
\]
\item \textup{($\beta$-recovery)} $\Phat$-almost surely, for all $x \in \bbZ^2$,
\begin{align}
1&=\textstyle\sum_{i=1}^2 e^{\beta \w_x(\what)-\beta \Bhat(\what,x,x+e_i)}
\qquad\text{if }\beta\in(0,\infty), \label{eq:pos-temp-recovery} 
\\[2pt] 
\text{and}\quad \w_x(\what)&=\Bhat(\what,x,x+e_1)\wedge \Bhat(\what,x,x+e_2)
\quad\text{if }\beta=\infty. \label{eq:infty-recovery}
\end{align}
\item \textup{($L^1(\Omhat,\kShat,\Phat)$)} For all $x,y\in\bbZ^2$, $\Ehat[|\Bhat(x,y)|]<\infty$.
\item \textup{(Cocycle)} $\Phat$-almost surely, for all $x,y,z\in\bbZ^2$ $\Bhat(x,y) + \Bhat(y,z) = \Bhat(x,z)$.
\end{enumerate} 

The collection of all such cocycles is denoted by $\cKhat^\beta$. $\cK^\beta$ denotes the space of shift-covariant $\beta$-recovering $L^1(\Omega,\sF,\bbP)$ cocycles, i.e., those on the canonical space.
\end{definition}

For $\beta\in(0,\infty]$ and $\Bhat\in\cKhat^\beta$ 
define the random 2-vector $\hhB(\Bhat)=\hhB(\Bhat,\what)\in\R^2$ via
\begin{align}\label{h-def}
	\hhB(\Bhat)\cdot e_i=-\Ehat[\Bhat(0,e_i)\,|\,\cIhat],\quad i\in\{1,2\}.
\end{align}

By \cite[Theorem 4.4]{Jan-Ras-20-aop} (see \cite[Theorem B.3]{Jan-Ras-18-arxiv} for the details), for $\Phat$-a.e.\ $\what$, 
\begin{align}\label{B-shape}
\lim_{n\to\infty}n^{-1}\max_{\abs{x}_1\le n}\abs{\Bhat(\what, 0,x)+\hhB(\Bhat, \what)\cdot x}=0.
\end{align}
Because we assume Condition \ref{cond:ergodic}, any $B\in\cK^\beta$ has $\bbP(\hhB(B)=\E[\hhB(B)])=1.$

We have the following, which is the same as Lemma 4.5 of \cite{Jan-Ras-20-aop}. The proof is almost identical, but adapted slightly to avoid an unnecessary application of the main result of \cite{Mar-04}, which requires an i.i.d.~weight hypothesis. The proof is in Appendix \ref{app:convex}.

\begin{lemma}\label{lem:h-superdiff}
Fix $\beta\in(0,\infty]$ and suppose that Condition \ref{cond:moment-mixing} holds. Let $\Bhat\in \cKhat^\beta$ be given. Then
\begin{enumerate}[label={\rm(\alph*)}, ref={\rm\alph*}]
\item\label{lem:h-superdiff-general} 
$\Phat$-almost surely, $-\hhB(\Bhat) \in \partial \fe^{\beta}(\Uset)$.
\item \label{lem:h-superdiff-mean} If there exists $\xi\in \Uset$ for which $-\Ehat[\hhB(\Bhat)]\in \partial \fe^{\beta}(\xi)$, then $\Phat$-almost surely, $-\hhB(\Bhat)\in \partial \fe^{\beta}(\xi)$.
\item \label{lem:h-superdiff-extreme} If there exists $\xi\in \Uset$ for which $-\Ehat[\hhB(\Bhat)]\in \ext \partial \fe^{\beta}(\xi)$, then $\Phat$-almost surely, $\hhB(\Bhat)=\Ehat[\hhB(\Bhat)]$.
\end{enumerate}
\end{lemma}

The above lemma says that any shift-covariant recovering $L^1(\Phat)$ cocycle must have a conditional mean vector which lies in the negative of the super-differential almost surely. It also identifies when such a vector is confined to the super-differential of a fixed direction and gives a sufficient condition for the random vector $\hhB(\Bhat)$ to be almost surely constant.  

Our next two results say that any $\Bhat\in\cKhat^\beta$ which is defined on some extended space and which has an almost surely constant conditional mean vector can be realized uniquely as an element on the canonical space (i.e.\ as an element of $\cK^\beta$). In the language of stochastic analysis, these are statements of strong existence  and strong uniqueness of shift-covariant recovering $L^1(\Phat)$ cocycles with deterministic conditional mean vectors. This is optimal, because on extended probability spaces, one can take convex mixtures of cocycle distributions to produce new cocycles with random conditional mean vectors. See Example 3.7 in \cite{Jan-Ras-Sep-25-strong-}.  Theorem \ref{thm:decomp} gives the reverse direction: if the tilt of a cocycle is genuinely random, then the cocycle is a mixture of ones in $\cK^\beta$. 

\begin{theorem}\label{thm:Bhat-B}
Fix $\beta\in(0,\infty)$ and  suppose that Condition \ref{cond:moment-mixing}\eqref{cond:moment-mixing-<infty} holds. Let $\Bhat\in\cKhat^\beta$ and  $h=\Ehat[\hhB(\Bhat)]$. Assume that $\Phat\{\what:\hhB(\Bhat,\what)=h\}=1$. Then there exists $\Bus\in\cK^\beta$ such that $\Bhat(\what)=\Bus(\w(\what))$, $\Phat$-almost surely. 
\end{theorem}

In the next result, for $h,h'\in-\partial\fe^\beta(\Uset)$, $h \precneq h'$ means $h\cdot e_1 < h'\cdot e_1$ and $h \cdot e_2 > h'\cdot e_2$. 

\begin{theorem}\label{thm:uniqueness}
Fix $\beta\in(0,\infty)$ and  suppose that Condition \ref{cond:moment-mixing}\eqref{cond:moment-mixing-<infty} holds. Let $\Bus^1,\Bus^2\in\cK^{\beta}$. Then exactly one of the following three happens: 
\begin{enumerate}  [label={\rm(\alph*)}, ref={\rm\alph*}] 
\item $\hhB(\Bus^1)\precneq\hhB(\Bus^2)$ and $\P(\Bus^1\prec\Bus^2)=1$, 
\item $\hhB(\Bus^1)\succneq\hhB(\Bus^2)$ and $\P(\Bus^1\succ\Bus^2)=1$, or 
\item $\hhB(\Bus^1)=\hhB(\Bus^2)$ and $\P(\Bus^1=\Bus^2)=1$.
\end{enumerate}
\end{theorem}
The analogues of these theorems when $\beta=\infty$ under Condition \ref{cond:moment-mixing}\eqref{cond:moment-mixing-iid} are Theorems 3.1 and 3.4 in \cite{Jan-Ras-Sep-25-strong-}, where an additional finite energy/coalescence condition is required on the cocycle. This issue is discussed in further detail in Section \ref{sec:L1}. 

For fixed $\beta\in(0,\infty)$, define the set of mean vectors of cocycles in $\cK^\beta$, i.e.\ shift-covariant $\beta$-recovering $L^1(\P)$ cocycles
\[
\sH^\beta=\{\hhB(B):B\in\cK^\beta\}.
\]
We call elements of $\sH^\beta$ \textit{tilts}. The following generalizes \cite[Proposition 3.5]{Jan-Ras-Sep-25-strong-} to the case $\beta<\infty$.

\begin{proposition}\label{prop:sH-extreme}
Fix $\beta\in(0,\infty)$ and suppose Condition \ref{cond:moment-mixing}\eqref{cond:moment-mixing-<infty} holds. Then
\be\label{sH_sub}
\{-\nabla \fe^{\beta}(\xi\sigg): \xi\in\ri\Uset,\ \sigg\in\{-,+\}\}\subset \sH^\beta \subset -\partial \fe^\beta(\Uset).
\ee
\end{proposition}

\begin{remark}\label{rk:H=grads}
When $\fe^\beta$ is differentiable, the inclusions in \eqref{sH_sub} are equalities. This holds in the exactly solvable log-gamma polymer; whether this is the case more generally is an open problem.
\end{remark}

\begin{proof}[Proof of Proposition \ref{prop:sH-extreme}]
The second inclusion comes from Lemma \ref{lem:h-superdiff}\eqref{lem:h-superdiff-general}. For the first inclusion, fix $\xi\in\ri\Uset$ and $\sigg\in\{-,+\}$, and let $h=-\nabla \fe^{\beta}(\xi\sigg)$. By Theorem 4.5 of \cite{Gro-Jan-Ras-25-tams}, there exists a probability space equipped with shifts and weights, $(\Omhat,\kShat,\Phat,\That,\w)$, satisfying the hypotheses of Section \ref{sec:probsp} on which there exists a cocycle 
$\Bhat\in\cKhat^\beta$ which satisfies $\Ehat[\hhB(\Bhat)]=h$. Since $-h\in\ext\partial \fe^{\beta}(\xi)$, Lemma \ref{lem:h-superdiff}\eqref{lem:h-superdiff-extreme} gives $\hhB(\Bhat, \what)=h$, $\Phat$-almost surely. Theorem \ref{thm:Bhat-B} now yields a cocycle $B\in\cK^\beta$ with $\hhB(B)=h$. Thus $h\in\sH^\beta$.
\end{proof}

By Theorems \ref{thm:Bhat-B} and \ref{thm:uniqueness}, for each $h\in\sH^\beta$ there exists a unique $\Bus\in\cK^\beta$ with $\hhB(\Bus)=h$. We denote this cocycle by $\Bus^{\beta,h}$. 

\subsection{Strong existence and uniqueness of the Busemann process}
The connection between the shift-covariant recovering cocycles discussed in Section \ref{sec:coc-existunique} and the extremal Gibbs measures which are the main focus of the paper comes from the following lemma. 

\begin{lemma}\label{lem:Bhat-ext}
Fix $\beta\in(0,\infty)$ and suppose that $\Bhat\in\cKhat^{\beta}$. There exists a shift-invariant Borel event $\Omhatext$ with $\Phat(\Omhatext)=1$ with the property that for all $\what \in \Omhatext$, the shift covariance, $\beta$-recovery, and cocycle properties in Definition \ref{def:recovcoc} hold for $\Bhat$ and
the consistent family
$\{\Pi_x^{\Bhat,\what} : x\in\bbZ^2\}$ of fully supported Gibbs-DLR measures defined for $x\in\bbZ^2$, $k=x\cdot(e_1+e_2)$, $m \geq k$ and admissible paths $x_{k:m}$ with $x_k=x$ by
\[
\Pi_x^{\Bhat,\what}(x_{k:m}) =\exp\biggl\{\beta\sum_{r=k}^{m-1}\w_{x_r}(\what)-\beta\Bhat(\what,x,x_m)\biggr\}, 
\]
admits a coalescing coupling. Consequently, for all $x\in\bbZ^2$, $\Pi_x^{\Bhat,\what} \in \ext \DLR_x^{\beta,\w(\what)}$.
\end{lemma}

The non-trivial portion of the claim of this lemma is the coalescence, which appeared previously in \cite{Jan-Ras-20-aop}. That paper adapts the Licea-Newman \cite{Lic-New-96} coalescence argument (which is itself a type of Burton-Keane \cite{Bur-Kea-89} lack of space argument) to the setting of weakly elliptic random walks in stationary random environments with $\{e_1,e_2\}$ steps. In particular, we do not need to assume Condition \ref{cond:moment-mixing}\eqref{cond:moment-mixing-<infty} here. An expert may also observe that there is no finite energy condition assumed in the present paper, as is usually required for such arguments. This assumption is unnecessary in our setting because the extra randomness coming from positive temperature allows one to perform the entire modification argument with respect to an auxiliary family of i.i.d.\ Uniform[0,1] random variables.

Theorem \ref{thm:uniqueness} implies that the collection of cocycles which are indexed by tilts is totally ordered.

\begin{corollary}\label{cor:Bus-process-order}
Fix $\beta\in(0,\infty)$ and suppose that Condition \ref{cond:moment-mixing}\eqref{cond:moment-mixing-<infty} holds. For any $h,h'\in\sH^\beta$, either we have $h\preceq h'$ and $\bbP(\Bus^{\beta,h}\preceq\Bus^{\beta,h'})=1$ or we have $h'\preceq h$ and $\bbP(\Bus^{\beta,h'}\preceq\Bus^{\beta,h})=1$. In particular, $\preceq$ is a total order on $\sH^\beta$ and $\sH^\beta\ni h\mapsto\Bus^{\beta,h}$ is nondecreasing. 
\end{corollary}

Let $\cHdense^\beta$ be a countable dense subset of $\sH^\beta$ containing every left- and right-isolated point of $\sH^\beta$, where a point is called left- or right-isolated if it cannot be approximated from the corresponding side in the order of Corollary \ref{cor:Bus-process-order}. Using the monotonicity in Corollary \ref{cor:Bus-process-order} and the cocycle property satisfied by $\Bus^{\beta,h}$, $h\in\cHdense^\beta$, define the process
\begin{align}\label{Busproc-def}
&\Bus^{\beta,h-}(x,y)=\lim_{\cHdense^\beta\tsp\ni\tsp h'\nearrow h}\Bus^{\beta,h'}(x,y)\quad\text{and}\quad
\Bus^{\beta,h+}(x,y)=\lim_{\cHdense^\beta\tsp\ni\tsp h'\searrow h}\Bus^{\beta,h'}(x,y),
\end{align}
for $x,y\in\bbZ^2$ and $h\in\sH^\beta$. Then for $\bbP$-almost every $\w$, for any $h\in\sH^\beta$ and $\sigg\in\{-,+\}$, $\Bus^{\beta,h\sig}$ is a $\beta$-recovering cocycle. The following lemma says that for a fixed $h\in\sH^\beta$, the above definitions recover $\Bus^{\beta,h}$.

\begin{lemma}\label{lem:Bus-process-extend}
Fix $\beta\in(0,\infty)$ and suppose that Condition \ref{cond:moment-mixing}\eqref{cond:moment-mixing-<infty} holds. Let $h\in\sH^\beta$. Then $\bbP$-almost surely, for all $x,y\in\bbZ^2$, $\Bus^{\beta,h-}(x,y)=\Bus^{\beta,h+}(x,y)=\Bus^{\beta,h}(x,y)$. In particular, this holds for $\bbP$-almost every $\w$, simultaneously for all $h\in\cHdense^\beta$.
\end{lemma}

\begin{proof}
By Corollary \ref{cor:Bus-process-order}, $\Bus^{\beta,h'}\preceq\Bus^{\beta,h}\preceq\Bus^{\beta,h''}$, $\P$-almost surely, for any $h',h''\in\cHdense^\beta$ such that $h'\preceq h\preceq h''$.  On the one hand, this gives 
$\P\{\Bus^{\beta,h-}\preceq\Bus^{\beta,h}\preceq\Bus^{\beta,h+}\}=1$. 
On the other hand, monotone convergence gives $\hhB(\Bus^{\beta,h\sig})=h=\hhB(\Bus^{\beta,h})$, for both $\sigg\in\{-,+\}$. This implies that $\Bus^{\beta,h\sig}(x,x+e_i)=\Bus^{\beta,h}(x,x+e_i)$, $\bbP$-almost surely and for all $x\in\Z^2$ and $i\in\{1,2\}$. The claim follows from this and the cocycle property.  
\end{proof}

Our next result records the strong existence and uniqueness of the Busemann process on the canonical space $(\Omega,\sF,\bbP)$. The analogous theorem when $\beta=\infty$ under Condition \ref{cond:moment-mixing}\eqref{cond:moment-mixing-iid} is Proposition 3.10 in \cite{Jan-Ras-Sep-25-strong-}.

Recall the random variables $\Bus^{\beta,h}$, $\beta\in(0,\infty)$, $h\in\sH^\beta$, defined after the proof of Proposition \ref{prop:sH-extreme}.

\begin{theorem}\label{thm:Bus-process}
Fix $\beta\in(0,\infty)$ and suppose that Condition \ref{cond:moment-mixing}\eqref{cond:moment-mixing-<infty} holds. Then the process
\begin{align}\label{Bproc}
\bigl(\Bus^{\beta,h\sig}(x,y):x,y\in\bbZ^2,\ h\in\sH^\beta,\ \sigg\in\{-,+\}\bigr),
\end{align}
defined in \eqref{Busproc-def} on $(\Omega,\sF,\bbP)$, has the following properties:
\begin{enumerate}[label={\rm(\alph*)}, ref={\rm\alph*}]
\item\label{thm:Bus-process:same} For each $h\in\sH^\beta$, \begin{align}\label{Bbar=B}
\bbP\{\Bus^{\beta,h-}=\Bus^{\beta,h+}=\Bus^{\beta,h}\}=1.
\end{align}
In general, whenever $\Bus^{\beta,h-}(\w)=\Bus^{\beta,h+}(\w)$, we omit the sign and write $\Bus^{\beta,h}(\w)$ for the common cocycle. By \eqref{Bbar=B}, this agrees almost surely with the previously defined random variable $\Bus^{\beta,h}$.
\item\label{thm:Bus-process:Bus} For each $h\in\sH^\beta$, $\Bus^{\beta,h}\in\cK^\beta$ and $\hhB(\Bus^{\beta,h})=h$.
\item\label{thm:Bus-process:mean} For each $h\in\sH^\beta$ and $i\in\{1,2\}$, $\bbE[\Bus^{\beta,h}(0,e_i)]=-h\cdot e_i$.
\item\label{thm:Bus-process:monotone} For $\bbP$-almost every $\w$, all $x\in\bbZ^2$, and every $h\precneq h'$ in $\sH^\beta$,  
\[
\Bus^{\beta,h-}(x,x+e_1)\ge \Bus^{\beta,h+}(x,x+e_1)> \Bus^{\beta,h'-}(x,x+e_1)\ge \Bus^{\beta,h'+}(x,x+e_1)
\]
and
\[
\Bus^{\beta,h-}(x,x+e_2)\le \Bus^{\beta,h+}(x,x+e_2)< \Bus^{\beta,h'-}(x,x+e_2)\le \Bus^{\beta,h'+}(x,x+e_2).
\]
\item\label{thm:Bus-process:cont} For $\bbP$-almost every $\w$, all $h\in\sH^\beta$, and all $x,y\in\bbZ^2$,
\[
\Bus^{\beta,h-}(x,y)=\lim_{\substack{\sH^\beta\tsp\ni\tsp h'\\ h'\nearrow h}}\Bus^{\beta,h'\pm}(x,y),
\qquad
\Bus^{\beta,h+}(x,y)=\lim_{\substack{\sH^\beta\tsp\ni\tsp h'\\ h'\searrow h}}\Bus^{\beta,h'\pm}(x,y).
\]
\end{enumerate}

Moreover, this process is unique: if $(\widetilde B^{h\sig}(x,y):x,y\in\Z^2,h\in\sH^\beta,\sigg\in\{-,+\})$ is another process on $(\Omega,\sF,\P)$ that satisfies the continuity in \eqref{thm:Bus-process:cont} and such that for a dense set of $h\in\sH^\beta$, $\widetilde B^{h-},\widetilde B^{h+}\in\cK^\beta$ and $\E[\hhB(B^{h-})]=\E[\hhB(B^{h+})]=h$, then, almost surely, this process is equal to the process \eqref{Bproc}. 
\end{theorem}

\begin{remark}\label{rmk:dir_Bus}
If the weights are i.i.d.\ then the Busemann process can be connected to directional limits of free energy differences as follows. Under this condition, Theorem 4.14 in \cite{Jan-Ras-20-aop} implies that the following holds $\bbP$-almost surely: let $x\in\bbZ^2$ and let $x_n$ be any sequence of sites in $\bbZ^2$ with $x_n/n\to\xi\in\,]e_2,e_1[$; call $\underline{\xi}$ and $\overline{\xi}$  the left-most and right-most endpoints of maximal linear segments of $\fe^\beta$ containing $\xi$, ordered so that $\underline{\xi}\cdot e_1\le \overline{\xi}\cdot e_1$. Then
\begin{align*}
-\Bus^{\beta,(-\nabla\fe^\beta(\underline{\xi}-))-}(x,x+e_1)
&\le \varliminf_{n\to\infty}\frac1\beta\log\frac{\PF{x+e_1}{x_n}^{\beta,\w}}{\PF{x}{x_n}^{\beta,\w}} \le \varlimsup_{n\to\infty}\frac1\beta\log\frac{\PF{x+e_1}{x_n}^{\beta,\w}}{\PF{x}{x_n}^{\beta,\w}} \\
&\le -\Bus^{\beta,(-\nabla\fe^\beta(\overline{\xi}+))+}(x,x+e_1),
\end{align*}
and
\begin{align*}
-\Bus^{\beta,(-\nabla\fe^\beta(\overline{\xi}+))+}(x,x+e_2)
&\le \varliminf_{n\to\infty}\frac1\beta\log\frac{\PF{x+e_2}{x_n}^{\beta,\w}}{\PF{x}{x_n}^{\beta,\w}} \le \varlimsup_{n\to\infty}\frac1\beta\log\frac{\PF{x+e_2}{x_n}^{\beta,\w}}{\PF{x}{x_n}^{\beta,\w}} \\
&\le -\Bus^{\beta,(-\nabla\fe^\beta(\underline{\xi}-))-}(x,x+e_2).  
\end{align*}
The direction-indexed Busemann process   
$\bigl(\Bus^{\beta,(-\nabla\fe^\beta(\xi\sig))\sig} (x,y):x,y\in\bbZ^2,\, \xi\in\,]e_2, e_1[\tspb,\,  \sigg\in\{-,+\} \bigr)$, appearing in the bounds, does not equal the full tilt-indexed process of Theorem \ref{thm:Bus-process} if the first inclusion in \eqref{sH_sub} is strict. As implied by Remark \ref{rk:H=grads}, the two processes are the same in the exactly solvable log-gamma polymer.
\end{remark}

We next record the fact that the Busemann process generates a family of extremal Gibbs-DLR measures.

\begin{corollary}\label{cor:Bus-process-DLR}
Fix $\beta\in(0,\infty)$ and suppose that Condition \ref{cond:moment-mixing}\eqref{cond:moment-mixing-<infty} holds. There exists a shift-invariant event $\Omega'\subset\Omega_0$ with $\bbP(\Omega')=1$ such that for every $\w\in\Omega'$, the following hold.
\begin{enumerate}[label={\rm(\alph*)}, ref={\rm\alph*}] \itemsep=2pt
\item\label{cor:Bus-process-DLR:extreme} For each $h\in\sH^\beta$, $\sigg\in\{+,-\}$, and $x\in\bbZ^2$, the measure $\Pi_x^{\beta,h\sig,\w}$ defined for $k=x\cdot(e_1+e_2)$, $m\ge k$, and admissible $x_{k:m}$ with $x_k=x$ by
\be
\Pi_x^{\beta,h\sig,\w}(x_{k:m})
=
\exp\biggl\{\beta\sum_{r=k}^{m-1}\w_{x_r}-\beta\Bus^{\beta,h\sig}(\w, x,x_m)\biggr\}\label{eq:Bus-process-DLR-definition}
\ee
is a fully supported element of $\ext\DLR_x^{\beta,\w}$.

\item\label{cor:Bus-process-DLR:consistent} For each $h\in\sH^\beta$ and $\sigg\in\{+,-\}$, the family $\{\Pi_x^{\beta,h\sig,\w}:x\in\bbZ^2\}$ is consistent and satisfies
\be
\Busgeo_{\Pi_x^{\beta,h\sig,\w}}=\Bus^{\beta,h\sig}(\w)
\qquad\text{for all }x\in\bbZ^2.\label{APi=B}
\ee

\item\label{cor:Bus-process-DLR:monotone} If $h\precneq h'$ in $\sH^\beta$, then for all $x\in\bbZ^2$, 
\be
\Pi_x^{\beta,h-,\w}\preceq \Pi_x^{\beta,h+,\w}\prec \Pi_x^{\beta,h'-,\w}\preceq \Pi_x^{\beta,h'+,\w}.
\ee
\end{enumerate}
\end{corollary}
\begin{remark}
    In the case of i.i.d.~exponential weights, it is shown in Theorem 3.11 of \cite{Jan-Ras-Sep-23} that the Busemann process generates all non-trivial semi-infinite geodesics. The same result was subsequently shown for the directed landscape (the renormalization fixed point of the KPZ universality class) in \cite{Bus-Sep-Sor-24,Bus-25-}. Based on these exactly solvable examples, it is natural to conjecture that for models which lie within the KPZ class (and perhaps more generally), the collection of Gibbs-DLR measures in Corollary \ref{cor:Bus-process-DLR} exhausts the non-trivial elements of $\ext \DLR_x^\w$. Partial progress toward proving this for the exactly solvable log-gamma polymer was made in \cite{Jan-Ras-20-aop, Bat-Fan-Sep-25}, with the missing input being the analogue of a ``No 3 geodesics'' result of the type proven in \cite{Cou-11, Bus-25-}. A precise statement of the missing input is that, almost surely, for any $\xi\in\ri\Uset$, there exist at most two distinct extremal Gibbs-DLR measures, rooted at the origin and supported on paths which satisfy a law of large numbers with asymptotic direction $\xi$.
\end{remark}

Theorem \ref{thm:Bus-process}\eqref{thm:Bus-process:cont} implies in particular that the Busemann process is jointly measurable in $(\w,h)$, which allows us to compose with random variables. Combining Theorems \ref{thm:Bhat-B}, \ref{thm:uniqueness}, and \ref{thm:Bus-process} with the ergodic decomposition theorem yields the following decomposition of the elements of $\cKhat^\beta$ on any extended probability space satisfying Condition \ref{cond:moment-mixing}\eqref{cond:moment-mixing-<infty}: any shift-covariant $\beta$-recovering $L^1$ cocycle is a mixture of the tilt-indexed Busemann process, determined by the law of the tilt vector of the cocycle, with no mass assigned to discontinuity tilts.  The analogous statement, with an additional finite energy/coalescence hypothesis, for $\beta=\infty$ under Condition \ref{cond:moment-mixing}\eqref{cond:moment-mixing-iid} is Theorem 3.14 in \cite{Jan-Ras-Sep-25-strong-}. 

\begin{theorem}\label{thm:decomp}
Fix $\beta\in(0,\infty)$ and suppose Condition \ref{cond:moment-mixing}\eqref{cond:moment-mixing-<infty} holds. Let $\Bhat\in\cKhat^\beta$. Then for $\Phat$-almost every $\what$,
\[
\hhB(\Bhat,\what)\in\sH^\beta
\qquad\text{and}\qquad
\Bhat(\what)=\Bus^{\beta,\hhB(\Bhat,\what)-}(\w(\what))=\Bus^{\beta,\hhB(\Bhat,\what)+}(\w(\what)).
\]
\end{theorem}

\smallskip

\subsection{$L^1$ continuity}\label{sec:L1}
We next turn to  $L^1$ continuity of the Busemann cocycles, viewed as a function of the weight field, the inverse temperature, and the tilt. For the following results, we will need to keep track of the dependence of various objects on the law of the weight field. Let $\nu$  be  a Borel probability measure on $\R^{\bbZ^2}$ and $\beta\in(0,\infty]$ such that the canonical space equipped with $\nu$ (i.e., $(\Omega,\sF,\nu,T,\w,\beta)$) satisfies Condition \ref{cond:moment-mixing} at inverse temperature $\beta$. We write $\fe^{\beta,\nu}$ for the corresponding limiting free energy, and let $\cK^{\beta,\nu}$ be the space of shift-covariant $\beta$-recovering $L^1(\Omega,\sF,\nu)$ cocycles on the canonical space. For $\beta\in(0,\infty]$, set
\be\label{eq:sH-beta-nu}
\sH^{\beta,\nu}=\{\hhB(B):B\in\cK^{\beta,\nu}\}.
\ee
When $\beta<\infty$, Theorems \ref{thm:Bhat-B} and \ref{thm:uniqueness} imply that for each $h\in\sH^{\beta,\nu}$, up to sets of measure zero, there is exactly one $B\in\cK^{\beta,\nu}$ with $\hhB(B)=h$. When $\beta=\infty$, Theorems 3.1 and 3.4 in \cite{Jan-Ras-Sep-25-strong-} combine to give nearly the same result, except that the corresponding uniqueness only holds within the class of cocycles satisfying a finite energy condition, which can be taken without loss of generality in that setting to be forward measurability. Here forward measurable means that $B(x,y)$ is measurable with respect to the weights in $\Z_{\geq u}^2$ whenever $\{x,y\}\subset \Z_{\geq u}^2$. In either case, whenever this canonical representative is defined, we denote an arbitrary Borel version on $(\Omega,\sF,\nu)$ by $\Bus^{\beta,h,\nu}$. 

Call
\be\label{eq:sH-all-nu}
\sH_{(0,\infty]}^\nu = \{(\beta,h): \beta\in(0,\infty], h\in\sH^{\beta,\nu}\}
\ee
and
\be\label{eq:sH-all-ext-nu}
\sH_{(0,\infty]}^{\nu,\textup{ext}}
=
\bigl\{(\beta,h)\in \sH_{(0,\infty]}^\nu:-h\in \{\nabla \fe^{\beta,\nu}(\xi\sigg):\xi\in\ri\Uset,\ \sigg\in\{-,+\}\}\bigr\}.
\ee
Recall also that if for some $\beta\in(0,\infty]$,  $\fe^{\beta,\nu}$ is differentiable at $\xi \in \ri \Uset,$ then $\partial \fe^{\beta,\nu}(\xi) = \{\nabla \fe^{\beta,\nu}(\xi)\}$. Consequently, if $\fe^{\beta,\nu}$ is differentiable on $\ri \Uset$ for all $\beta$, then $\sH_{(0,\infty]}^{\nu,\textup{ext}}=\sH_{(0,\infty]}^\nu$.

In the next theorem we work on a single extended probability space equipped with countably many different weight fields which have been coupled together. In some circumstances, we will require the following condition if the limit is at $\beta=\infty$.

\begin{condition}\label{cond:zero-temp-strong}
Suppose that $(\Omega,\sF,\nu,T,\w,\infty)$ satisfies Condition \ref{cond:moment-mixing}\eqref{cond:moment-mixing-iid}. 
Furthermore, assume that for each $h\in\sH^{\infty,\nu}$ there is a Borel representative $\Bus^{\infty,h,\nu}$ of the unique element of $\cK^{\infty,\nu}$ with $\hhB(\Bus^{\infty,h,\nu})=h$, $\nu$-almost surely. Precisely: whenever $(\Omhat,\kShat,\Phat,\That,W)$ is an extended probability space satisfying the hypotheses of Section \ref{sec:probsp}, with $W$ having law $\nu$, and $\Bhat$ is a shift-covariant $\infty$-recovering $L^1(\Phat)$ cocycle for $W$ such that $\Phat\{\what:\hhB(\Bhat,\what)=h\}=1$, then
\[
\Bhat(\what,x,y)=\Bus^{\infty,h,\nu}(W(\what),x,y)
\]
for all $x,y\in\bbZ^2$, $\Phat$-almost surely.
\end{condition}

This condition differs from the strong uniqueness proven in Theorem 3.4 of \cite{Jan-Ras-Sep-25-strong-} in that we do not have a finite energy or coalescence hypothesis on the cocycle in this new condition. 

\begin{theorem}\label{thm:L1-continuity-coupled}
Let $(\Omhat,\kShat,\Phat,\That)$ be a Polish probability space with a group of continuous automorphisms $\That = \{\That_z : z\in\bbZ^2\}$. For each $n\in\bbN\cup\{\infty\}$, let $\w^n=\{\w_x^n:x\in\bbZ^2\}$ be a field on $\Omhat$ and $\beta_n\in(0,\infty]$ an inverse temperature such that $(\Omhat,\kShat,\Phat,\That,\w^n,\beta_n)$ satisfies Condition \ref{cond:moment-mixing}, and let $\nu_n$ be the law of $\{\w_x^n : x\in\bbZ^2\}$ on $(\Omega,\sF)$ under $\Phat$. Suppose also that for each $n\in\bbN\cup\{\infty\}$ with $\beta_n=\infty$, the law $\nu_n$ satisfies Condition \ref{cond:zero-temp-strong}.
 Assume the following:
\begin{enumerate}[label={\rm(\alph*)}, ref={\rm\alph*}]\itemsep=2pt
\item\label{thm:L1-continuity-coupled:weights-L1}
$\Ehat[|\w_0^n-\w_0^\infty|]\to0$.
\item\label{thm:L1-continuity-coupled.params}
$(\beta_n,h_n)\in\sH_{(0,\infty]}^{\nu_n}$ for all $n$, $\ddd\lim_{n\to\infty}(\beta_n,h_n)=(\beta_\infty,h_\infty)$, and $(\beta_\infty,h_\infty)\in\sH_{(0,\infty]}^{\nu_\infty,\textup{ext}}$.
\end{enumerate}
Then for all $x,y\in\bbZ^2$,
\[
\lim_{n\to\infty}\Ehat\bigl[\bigl|\Bus^{\beta_n,h_n,\nu_n}(\w^n,x,y)-\Bus^{\beta_\infty,h_\infty,\nu_\infty}(\w^\infty,x,y)\bigr|\bigr]=0.
\]
\end{theorem}

\begin{remark}\label{rem:unnecessary-condition}
Condition \ref{cond:zero-temp-strong} can be removed from Theorem \ref{thm:L1-continuity-coupled} and the rest of the paper if, whenever $\beta_n<\infty$, the field $\{\w^n_x:x\in\Z^2\}$ is i.i.d., in addition to satisfying the rest of Condition \ref{cond:moment-mixing}\eqref{cond:moment-mixing-<infty}. Under this stronger hypothesis, the spot in the proof 
of the theorem where Condition \ref{cond:zero-temp-strong} is invoked can be replaced as follows. Lemma \ref{lem:forward-measurable} shows that under this hypothesis, every $\Bus \in \cK^\beta$ is forward-measurable, up to sets of measure zero. This property is inherited by limits in probability, and so the limiting cocycle in the proof must be forward-measurable. This implies the coalescence hypothesis of Theorem 3.4 in \cite{Jan-Ras-Sep-25-strong-}, by Corollary 4.12 of \cite{Jan-Ras-Sep-25-strong-}. This allows us to invoke Theorem 3.4 in \cite{Jan-Ras-Sep-25-strong-} in place of Condition \ref{cond:zero-temp-strong} in the proof.

We do not restrict to Condition \ref{cond:moment-mixing}\eqref{cond:moment-mixing-iid} in the statement because we believe it may be the case that an improved zero-temperature strong uniqueness input which does not require assuming a finite energy (or, equivalently, coalescence)  condition on the cocycles, as required by Condition \ref{cond:zero-temp-strong}, can be proven. This will require different methods from those of \cite{Jan-Ras-Sep-25-strong-} and we plan to address this improvement in future work.  
\end{remark}

\begin{remark}
The restriction to limit points in $\sH_{(0,\infty]}^{\nu_\infty,\textup{ext}}$ corresponds to restricting to the direction-indexed part of the tilt-indexed Busemann process. This is because for $\xi \in \ri \Uset$, the one-sided gradients $\nabla\fe^{\beta,\nu}(\xi\pm)$ are the two extreme points of $\partial\fe^{\beta,\nu}(\xi)$. These points play a role in the argument via Lemma \ref{lem:h-superdiff}\eqref{lem:h-superdiff-extreme}, which says that if $\Ehat[\hhB(\Bhat)]=h$ is an extreme point of $-\partial\fe^{\beta,\nu}(\xi)$, then $\Phat(\hhB(\Bhat)=h)=1$.  The direction-indexed process has been the primary focus of study in past works like \cite{Geo-Ras-Sep-17-ptrf-1,Geo-Ras-Sep-17-ptrf-2, Jan-Ras-Sep-23}. 
\end{remark}

\begin{remark} \label{rmk:Dspace} Theorem \ref{thm:L1-continuity-coupled}   implies weak convergence of finite-dimensional distributions of the Busemann process  
 in a zero-temperature limit $\beta_n\to\infty$ when the parameters $(\beta_n, h_n, \nu_n)$ are appropriately chosen. Tightness at a process level  has been obtained only for the solvable log-gamma polymer whose zero-temperature limit is the corner growth model with exponential weights. Theorem 4.6 in \cite{Bat-Fan-Sep-25} 
 gives the weak convergence of the direction-indexed process  
 $(\Bus^{\beta,(-\nabla\fe^{\beta}(\xi))+, \tspb \nu^\beta}(x,x+e_1): \xi\in\,]e_2, e_1]\tspa)$
 as $\beta\to\infty$ on the space $D(\tspa]e_2, e_1],\R)$ of c\`adl\`ag functions indexed by $\tspa]e_2, e_1]$. 
In that theorem, the i.i.d.\ weights $\{\w^\beta_x:x\in\Z^2\}\sim\nu^\beta$  for $\beta\in(0,\infty)$ have marginal  distributions $\w^\beta_x\sim-\beta^{-1}\log G^{\beta^{-1}}_x$ where  $G^\alpha_x\sim$ Gamma$(\alpha)$, and $\w^\infty_x\sim$ Exp(1), i.e.\ the distribution of $\w_x^{\beta}$ changes with $\beta$ so that the model, as described in Section \ref{sec:FE_shape}, stays in the solvable log-gamma setting for all $\beta\in(0,\infty)$ and the weak limit $\{\w^\beta_x:x\in\Z^2\}\overset{d}\to\{\w^\infty_x:x\in\Z^2\}$ holds as $\beta\to\infty$.  In this case the shape functions are known to be differentiable and strictly concave, so the direction-indexed Busemann processes are equal to the tilt-indexed ones (Proposition \ref{prop:sH-extreme} and Remark \ref{rk:H=grads}). 
 
 The proof  in \cite{Bat-Fan-Sep-25} relies on a coupling of the exact distributions of these Busemann processes in terms of an inhomogeneous Poisson process, and hence does not readily generalize.  It is restricted to the nearest-neighbor case $(x,y)=(x,x+e_1)$ because for general pairs $(x,y)$  the joint Busemann distribution across directions does not yet have a useful explicit description. 
\end{remark}

\begin{remark}
The requirement $h_n\to h_\infty$ in Condition \eqref{thm:L1-continuity-coupled.params} of  Theorem \ref{thm:L1-continuity-coupled}   says that the super-derivative elements indexing the Busemann functions converge. Sufficient conditions for this to hold are standard in convex analysis. Lemma \ref{lem:supergrad-converge} in the appendix provides one such condition which factors into one of our applications, Corollary \ref{cor:L1-continuity-perturb} below. The requirement that we only consider extreme points in the limit is a technical artifact of our proof, which does leave a potential gap in convergence of the Busemann process if cocycles with non-extremal tilts exist. The negative of any such tilt must lie on the relative interior of a superdifferential line segment $]\nabla \fe^{\beta,\nu_{\infty}}(\xi+),\nabla \fe^{\beta,\nu_{\infty}}(\xi-)[$ for some $\xi\in \ri \Uset$. This set is empty if the limit shape is differentiable, as is expected to hold very generally.
\end{remark}

The next result is an immediate application of Theorem \ref{thm:L1-continuity-coupled}. It proves continuity in $L^1$ of Busemann cocycles as a function of the inverse temperature and the tilt in a fixed environment. In the statement below, the $\pm$ sign distinction is not needed due to part \eqref{thm:Bus-process:same} of Theorem \ref{thm:Bus-process}.  This is the first time any regularity of the Busemann process as a function of the inverse temperature in a fixed environment has been shown.

\begin{corollary}\label{cor:L1-continuity}
Assume the hypotheses of Theorem \ref{thm:L1-continuity-coupled} with $(\Omhat,\kShat,\Phat,\That,\w^n)=(\Omega,\sF,\bbP,T,\w)$ for all $n\in\bbN\cup\{\infty\}$, and with $\nu_n=\nu$, the law of $\w$ under $\bbP$. Then for all $x,y\in\bbZ^2$ and all $(\beta,h)\in\sH_{(0,\infty]}^{\nu,\textup{ext}}$,
\[
\lim_{n\to\infty}\bbE\bigl[\bigl|\Bus^{\beta_n,h_n,\nu}(x,y)-\Bus^{\beta,h,\nu}(x,y)\bigr|\bigr]=0
\]
whenever $(\beta_n,h_n)\to(\beta,h)$ in $\sH_{(0,\infty]}^\nu$.\end{corollary}

The next corollary proves continuity of the Busemann cocycles under certain bounded perturbations of the weights, at least in directions of differentiability. As we will see momentarily, this implies continuity in probability of the generated infinite volume Gibbs-DLR measures. 
\begin{corollary}\label{cor:L1-continuity-perturb}
Fix $\beta\in(0,\infty]$. Let $(\Omhat,\kShat,\Phat,\That)$ be a Polish probability space with a group of continuous automorphisms, and let it carry shift-covariant weight fields $\w$ and $\eta$, with $|\eta_0|\le M$ $\Phat$-almost surely for some $M>0$. Let $\epsilon_n\to0$, set $\w^n=\w+\epsilon_n\eta$ and $\w^\infty=\w$, and, for $n\in\bbN\cup\{\infty\}$, let $\nu_n$ be the law of $\w^n$ on $(\Omega,\sF)$ under $\Phat$. Assume that $(\Omhat,\kShat,\Phat,\That,\w^n,\beta)$ satisfies Condition \ref{cond:moment-mixing} for every $n\in\bbN\cup\{\infty\}$ and, if $\beta=\infty$, that each $\nu_n$ also satisfies Condition \ref{cond:zero-temp-strong}.

Fix $\xi\in\ri\Uset$ at which $\fe^{\beta,\nu_\infty}$ is differentiable, set $h_\infty=-\nabla \fe^{\beta,\nu_\infty}(\xi)$, and choose any $h_n$ such that
\[
h_n\in\sH^{\beta,\nu_n}\cap\bigl(-\partial\fe^{\beta,\nu_n}(\xi)\bigr)
\qquad\text{for each }n\in\bbN.
\]
Then for all $x,y\in\bbZ^2$,
\[
\lim_{n\to\infty}\Ehat\bigl[\bigl|\Bus^{\beta,h_n,\nu_n}(\w^n,x,y)-\Bus^{\beta,h_\infty,\nu_\infty}(\w^\infty,x,y)\bigr|\bigr]=0.
\]
\end{corollary}

\subsection{Convergence of Gibbs measures and large deviations}
An immediate consequence of Theorem \ref{thm:L1-continuity-coupled} is that the extremal Gibbs measures generated by the Busemann cocycles are continuous in probability, provided that we restrict attention to positive temperature.

\begin{corollary}\label{cor:Gibbs-measure-continuity}
Assume the hypotheses of Theorem \ref{thm:L1-continuity-coupled} and assume that $\beta_n\in(0,\infty)$ for all $n\in\bbN\cup\{\infty\}$. For $n\in\bbN\cup\{\infty\}$ and $x\in\bbZ^2$, let $\Pi_x^{\beta_n,h_n,\nu_n,\w^n}$ be the Gibbs-DLR measure generated by $\Bus^{\beta_n,h_n,\nu_n}(\w^n)$, defined for $k=x\cdot(e_1+e_2)$, $m\ge k$, and admissible $x_{k:m}$ with $x_k=x$ by
\be
\Pi_x^{\beta_n,h_n,\nu_n,\w^n}(x_{k:m})
=
\exp\biggl\{\beta_n\sum_{r=k}^{m-1}\w^n_{x_r}-\beta_n\Bus^{\beta_n,h_n,\nu_n}(\w^n,x,x_m)\biggr\}.\label{Pindef}
\ee
Then $\Pi_x^{\beta_n,h_n,\nu_n,\w^n}\to\Pi_x^{\beta_\infty,h_\infty,\nu_\infty,\w^\infty}$
in $\sM_1(\pathsp_x,\pathsa_{k:\infty})$, in $\Phat$-probability.
\end{corollary}

In general, geodesics with given roots and directions can fail to be unique, so we cannot claim that Gibbs measures converge to a particular Busemann geodesic in the zero temperature limit. Applying Corollary \ref{cor:L1-continuity} to the zero temperature limit implies large deviation estimates under which paths at which the rate function vanishes are exactly the limiting infinite geodesics.

We say that a path is a $\Bus^{\infty,h,\nu}(\w)$-geodesic if each of its steps attains equality in the $\beta=\infty$ recovery relation \eqref{eq:infty-recovery} for $\Bus^{\infty,h,\nu}(\w)$.  That such paths are geodesics is a straightforward computation; see e.g.\ \cite[Lemma 4.1(a)]{Geo-Ras-Sep-17-ptrf-2}. 

\begin{corollary}\label{cor:cylinder-rate}
Assume the hypotheses of Theorem \ref{thm:L1-continuity-coupled}, with $\beta_n<\infty$ for all $n\in\bbN$, $\beta_\infty=\infty$, and $(\infty,h_\infty)\in\sH_{(0,\infty]}^{\nu_\infty,\textup{ext}}$. For $x\in\bbZ^2$ and an admissible path $x_{k:m}$ with $x_k=x$, define
\be
I_x^{h_\infty, \nu_\infty,\w^\infty}(x_{k:m})
=
\sum_{j=k}^{m-1}\bigl(\Bus^{\infty,h_\infty,\nu_\infty}(\w^\infty,x_j,x_{j+1})-\w^\infty_{x_j}\bigr),\label{eq:cylinder-rate-sum}
\ee
Then
\[
-\frac{1}{\beta_n}\log \Pi_x^{\beta_n,h_n,\nu_n,\w^n}(x_{k:m})
\to
I_x^{h_\infty, \nu_{\infty},\w^\infty}(x_{k:m})
\]
in $L^1(\Phat)$ as $n
\to\infty$. Moreover, $I_x^{h_\infty,\nu_\infty,\w^\infty}(x_{k:m})\ge0$, and equality holds if and only if $x_{k:m}$ is a $\Bus^{\infty,h_\infty,\nu_\infty}(\w^\infty)$-geodesic.
\end{corollary}

Because the previous corollary holds in $L^1$, we can diagonalize over the countable family of finite path segments and pass to subsequences to obtain a deterministic subsequence on which a full quenched large deviation principle holds.

\begin{corollary}\label{cor:subseq-quenched-ldp}
Assume the hypotheses of Theorem \ref{thm:L1-continuity-coupled}, with $\beta_n<\infty$ for all $n\in\bbN$, $\beta_\infty=\infty$, and $(\infty,h_\infty)\in\sH_{(0,\infty]}^{\nu_\infty,\textup{ext}}$. Then there exist a deterministic subsequence $\{n_\ell\}_{\ell\ge1}$ and an event $\Omhat_1\in\kShat$ with $\Phat(\Omhat_1)=1$ such that for every $\what\in\Omhat_1$ and every $x\in\bbZ^2$, the measures $\Pi_x^{\beta_{n_\ell},h_{n_\ell},\nu_{n_\ell},\w^{n_\ell}(\what)}$ on $\pathsp_x$ satisfy a large deviation principle with normalization $\beta_{n_\ell}$ and rate function
\be\begin{aligned}
I_x^{h_\infty,\nu_\infty,\w^\infty(\what)}(x_{k:\infty})
&=
\sum_{j=k}^\infty\bigl(\Bus^{\infty,h_\infty,\nu_\infty}(\w^\infty(\what),x_j,x_{j+1})-\w^\infty_{x_j}(\what)\bigr)\\
&=
\sup_{m\ge k}\Bigl\{\Bus^{\infty,h_\infty,\nu_\infty}(\w^\infty(\what),x,x_m)-\sum_{j=k}^{m-1}\w^\infty_{x_j}(\what)\Bigr\},
\label{eq:quenched-rate-fn}
\end{aligned}\ee
where $k=x\cdot(e_1+e_2)$. Moreover,
$I_x^{h_\infty,\nu_\infty,\w^\infty(\what)}(x_{k:\infty})=0$
if and only if $x_{k:\infty}$ is a $\Bus^{\infty,h_\infty,\nu_\infty}(\w^\infty(\what))$-geodesic.
\end{corollary}

With our main results now stated, we turn to the proofs.

\section{Properties of \texorpdfstring{$\DLR_x^\beta$}{DLRxbeta} and \texorpdfstring{$\ext \DLR_x^\beta$}{ext DLRxbeta}}\label{sec:structure} 
We start by proving the new structural properties of $\DLR_x^\beta$ and $\ext \DLR_x^\beta$ which play a role in our analysis. Throughout this section, we will work on the canonical space $\Omega$. Many of our results below are restricted to the set $\Omega_0$ from \eqref{eq:Omega0-definition}.

We begin by proving that on $\Omega_0$, all Gibbs-DLR measures which are not convex combinations of trivial measures are fully supported.
\begin{proof}[Proof of Lemma \ref{lem:Omega0}]
Fix $\w\in\Omega_0$, $x\in\bbZ^2$, $\beta\in(0,\infty)$, and $\Pi\in \DLR_x^{\beta,\w}$. We first rule out non-trivial coordinate rays. By symmetry it is enough to treat the case $j=2$. Fix $v\ge x$ with $(v-x)\cdot e_1>0$, set $r=v\cdot(e_1+e_2)$, and suppose
\[
\delta=\Pi\bigl(X_{r+m}=v+me_2\text{ for all }m\ge 0\bigr)>0.
\]
For $N\ge 1$, let
\[
E_N=\{X_{r+m}=v+me_2\text{ for }0\le m\le N\}.
\]
Then $\delta\le \Pi(E_N)$ for every $N\ge 1$. Summing \eqref{eq:Gibbs} over all admissible paths from $x$ to $v$, we get
\[
\Pi(E_N)=\Pi(v+Ne_2)\,\frac{\PF{x}{v}^{\beta,\w}\exp\{\beta\sum_{m=0}^{N-1}\w_{v+me_2}\}}{\PF{x}{v+Ne_2}^{\beta,\w}}
\le
\frac{\PF{x}{v}^{\beta,\w}\exp\{\beta\sum_{m=0}^{N-1}\w_{v+me_2}\}}{\PF{x}{v+Ne_2}^{\beta,\w}}.
\]
Choose one fixed admissible path from $x$ to $v-e_1$, and let $c(\w)$ be the sum of the weights along this path, excluding the endpoint $v-e_1$. We have
\[
\PF{x}{v+Ne_2}^{\beta,\w}\ge \exp\Bigl\{\beta c(\w)+\beta\sum_{m=0}^{N}\w_{v-e_1+me_2}\Bigr\}.
\]
Therefore
\[
\delta\le \PF{x}{v}^{\beta,\w} e^{-\beta c(\w)-\beta\w_{v-e_1}}\exp\Bigl\{\beta\sum_{m=0}^{N-1}\bigl(\w_{v+me_2}-\w_{v-e_1+(m+1)e_2}\bigr)\Bigr\}
= \PF{x}{v}^{\beta,\w} e^{-\beta c(\w)-\beta\w_{v-e_1}} e^{\beta S_N^{v,2}}.
\]
Since $\w\in\Omega_0$, $\varliminf_{N\to\infty}S_N^{v,2}(\w)=-\infty$, 
so the right-hand side has a subsequence converging to $0$. This contradicts $\delta>0$. 

It remains to prove the final claim. Assume now that $\Pi$ is not fully supported.  Let $k=x\cdot(e_1+e_2)$. If $\Pi(x_{k:n})=0$ for some admissible path $x_{k:n}$ with endpoint $x_n=y$, then \eqref{eq:Gibbs} and positivity of $\poly{x}{y}^{\beta,\w}(x_{k:n})$ give $\Pi(y)=0$.

Fix any $z\ge y$. We have $0=\Pi(X_n=y,X_m=z)=\Pi(z)\,\poly{x}{z}^{\beta,\w}(X_n=y)$ and so $\Pi(z)=0$. Therefore $\Pi\bigl(\exists m\ge k:\ X_m\in y+\bbZ_{\ge 0}^2\bigr)=0.$ Consequently, under $\Pi$ the path cannot take infinitely many $e_1$-steps and infinitely many $e_2$-steps, since then both coordinates would diverge to $+\infty$ and the path would eventually enter $y+\bbZ_{\ge 0}^2$. Hence $\Pi$-almost surely the path eventually follows a coordinate ray. The first part of the proof rules out all coordinate rays except the two rooted rays $x+\bbZ_{\geq0}e_1$ and $x+\bbZ_{\geq0}e_2$. Thus $\Pi$ is supported on the two trivial paths, and the result follows.
\end{proof}

\subsection{Tree coupling of finite measures}
Many of our arguments will involve a tree coupling of finite volume polymer measures. This coupling was used to study Gibbs measures in the polymer model previously in \cite{Geo-etal-15, Jan-Ras-20-aop}, but may have appeared earlier as well. 

For $y \in x + \bbZ_{\geq 0}^2$ with $y\neq x$, define the \emph{backward polymer transition probabilities}
\be\label{eq:backward-prob}
\overleftarrow{\pi}^{\beta,\w}_{y,y-e_i}
\;=\;
\frac{\PF{x}{y-e_i}^{\beta,\w}\,e^{\beta \w_{y-e_i}}}{\PF{x}{y}^{\beta,\w}},
\qquad i\in\{1,2\},
\ee
with the convention that the numerator is $0$ if $y-e_i\notin x+\bbZ_{\geq 0}^2$. By the one-step decomposition of the partition function,
$\overleftarrow{\pi}^{\beta,\w}_{y,y-e_1}+\overleftarrow{\pi}^{\beta,\w}_{y,y-e_2}=1$.

Denote by $\zeta=(\zeta_z : z\in\bbZ^2)$ an auxiliary field of i.i.d.\ Uniform$[0,1]$ random variables on the product space. 
For each $\w\in\Omega$, we use the $\zeta$ field to construct a random directed spanning tree $\cT_x$ on the vertex set $x+\bbZ_{\geq 0}^2$ by choosing, for each $y\in(x+\bbZ_{\geq 0}^2)\setminus\{x\}$,
a unique predecessor $p(y)\in\{y-e_1,y-e_2\}$ as follows:
\be\label{eq:parent-map}
p(y)=\begin{cases}
y-e_1, & \zeta_y \le \overleftarrow{\pi}^{\beta,\w}_{y,y-e_1},\\
y-e_2, & \zeta_y > \overleftarrow{\pi}^{\beta,\w}_{y,y-e_1}.
\end{cases}
\ee
Every non-root vertex has exactly one incoming edge, so this defines a tree oriented away from the root $x$.

Let $Q_x^{\beta,\w}$ denote the law of the tree constructed above. A key point that will be used repeatedly in the sequel is that under $Q_x^{\beta,\w}$, the predecessor random variables $(p(y) : y\in x+\bbZ_{\geq 0}^2\setminus\{x\})$ are independent. 

For $y\in x+\bbZ_{\geq 0}^2$, let $\sigma^{x,y}$ denote the unique up-right path in $\cT_x$ from $x$ to $y$. Our next result records the fact that the measure $Q_x^{\beta, \w}$ couples together all point-to-point polymer measures $\poly{x}{y}^{\beta,\w}$ through the single tree $\cT_x$.

\begin{lemma}\label{lem:tree-marginals}
For each $y\in x+\bbZ_{\geq 0}^2$, the law of $\sigma^{x,y}$ under $Q_x^{\beta,\w}$ is $\poly{x}{y}^{\beta,\w}$.
\end{lemma}

\begin{proof}
Fix $y$ with $y\cdot(e_1+e_2)=m$ and a path $\pi_{k:m}=(\pi_k,\dots,\pi_m)\in\Path{x}{y}$. The event $\{\sigma^{x,y}=\pi_{k:m}\}$ occurs if and only if $p(\pi_j)=\pi_{j-1}$ for each $j=k+1,\dots,m$.
By independence of the i.i.d.~field $\zeta$, the predecessor variables are independent under $Q_x^{\beta,\w}$, and hence
\[
Q_x^{\beta,\w}(\sigma^{x,y}=\pi_{k:m})
=\prod_{j=k+1}^m Q_x^{\beta,\w}\bigl(p(\pi_j)=\pi_{j-1}\bigr)
=\prod_{j=k+1}^m \overleftarrow{\pi}^{\beta,\w}_{\pi_j,\pi_{j-1}}.
\]
Substituting \eqref{eq:backward-prob}, we see that
\[
\prod_{j=k+1}^m \overleftarrow{\pi}^{\beta,\w}_{\pi_j,\pi_{j-1}}
=\prod_{j=k+1}^m \frac{\PF{x}{\pi_{j-1}}^{\beta,\w}e^{\beta\w_{\pi_{j-1}}}}{\PF{x}{\pi_j}^{\beta,\w}}
=\frac{e^{\beta\sum_{j=k}^{m-1}\w_{\pi_j}}}{\PF{x}{y}^{\beta,\w}}
=\poly{x}{y}^{\beta,\w}(\pi_{k:m}). \qedhere
\]
\end{proof}

\subsection{Total ordering of rooted extremal Gibbs measures}\label{sec:total-order}
We next work toward proving that the set of extremal Gibbs measures rooted at a site $x$ is totally ordered. As we will see, this essentially follows from the existence of the tree coupling of the finite volume polymer measures, recorded above as Lemma \ref{lem:tree-marginals}.

\begin{lemma}\label{nice-coupling}
    Fix any $\w\in\Omega$ and $x\in\Z^2$ and let $\Pi \in \DLR_x^{\beta,\w}$ be fully supported. Abbreviate $k=x\cdot(e_1+e_2)$ and, as throughout, let $X_{k:\infty}$ denote the path coordinate variables under $\Pi$. Then for each integer $m>k$, the distribution of the concatenation $\sigma^{x,X_m}\concat X_{m:\infty}$, under $Q_x^{\beta,\w}\otimes\Pi$, is equal to $\Pi$.
\end{lemma}

\begin{proof}
By Theorem \ref{thm:DLR-cocycle} and equation \eqref{eq:DLR-transition},  $\Pi$ is a Markov chain. Fix integers $n>m>k$ and an admissible path $x_{k:n}$ with $x_k=x$. Then
\begin{align*}
&Q_x^{\beta,\w}\otimes\Pi(\sigma^{x,X_m}=x_{k:m},X_{m:n}=x_{m:n}) =
Q_x^{\beta,\w}(\sigma^{x,x_m}=x_{k:m})\,\Pi(x_{m:n}) =
\poly{x}{x_m}^{\beta,\w}(x_{k:m})\,\Pi(x_{m:n}) \\
&\qquad= \Pi(x_{k:m}\viiva x_m)\,\Pi(x_{m:n}) =
\Pi(x_{k:m}\viiva x_{m:n})\,\Pi(x_{m:n}) =
\Pi(x_{k:n}).
\end{align*}
The first equality is independence, the second is Lemma \ref{lem:tree-marginals}, and the third is
\eqref{eq:Gibbs} with endpoint $x_m$, the fourth is the Markov property, and the last holds by definition.
\end{proof}

The basic structure of our argument will be to use the previous lemma to construct a monotone coupling of two extremal Gibbs measures through a weak limit procedure. Our next result records the fact that any coupling of the type we consider preserves the Gibbs property, even after conditioning on the $\sigma$-algebra of the joint future of the pair of paths, up to the level of the concatenation.

\begin{lemma}
\label{lem:concat-cond}
Fix $\w\in\Omega_0$, $\beta>0$, and $x\in\Z^2$, and abbreviate $k=x\cdot(e_1+e_2)$.
Let $\Pi^1,\Pi^2\in\DLR_x^{\beta,\w}$ be fully supported and let $\Pi^{1,2}$ be any coupling of $\Pi^1$ and $\Pi^2$. Denote by $Q_x^{\beta,\w}\otimes \Pi^{1,2}$ the product measure. For each integer $n>k$, define the concatenated paths
\[
\widehat X^{i,n} \;=\; \sigma^{x,X_n^i}\concat X^i_{n:\infty}, \qquad i\in\{1,2\},
\]
and let $\widetilde\Pi^{1,2,n}$ denote the coupling of $\Pi^1$ and $\Pi^2$ given by the law of $(\widehat X^{1,n},\widehat X^{2,n})$ on $\pathsp_x\times\pathsp_x$ under $Q_x^{\beta,\w}\otimes \Pi^{1,2}$.
Then for any integers $n>m>k$, any finite up-right path segment $x_{k:m}$ with $x_k=x$, and any event
$B\in \pathsa^{(2)}_{m:\infty}$,
\begin{equation}\label{eq:concat-cond}
\widetilde\Pi^{1,2,n}\bigl(X^{1}_{k:m}=x_{k:m},\, B\bigr)
\;=\;
\bfE^{\widetilde\Pi^{1,2,n}}\Bigl[\one_B\;\poly{x}{X^{1}_m}^{\beta,\w}(x_{k:m})\Bigr].
\end{equation}
Equivalently, under $\widetilde\Pi^{1,2,n}$, the conditional distribution of $X^{1}_{k:m}$ given $\pathsa^{(2)}_{m:\infty}$ is $\poly{x}{X^{1}_m}^{\beta,\w}$.
\end{lemma}
\begin{proof}
Let $k<m<n$ be as in the statement. Fix any finite path segment $x_{k:m}$ with $x_k=x$ and any $B \in \pathsa_{m:\infty}^{(2)}$. Abbreviate by $\overline{\bfE}$ the expectation under $Q_x^{\beta,\w}\otimes\Pi^{1,2}$.  Observe that $(\wh X^{1,n}_{m:\infty}, \wh X^{2,n}_{m:\infty})$ is measurable with respect to $\sigma((p(y) : y \geq x, y\cdot(e_1+e_2)>m), (X^{1}_{n:\infty}, X^{2}_{n:\infty})) = \sA$. Using the independence of $\sigma(p(y) : y \geq x, y \cdot (e_1+e_2) \leq m)$ and $\sA$ under $\overline{\bfE}$ and Lemma \ref{lem:tree-marginals},
\begin{align*}
\widetilde\Pi^{1,2,n}\bigl(X^{1}_{k:m}=x_{k:m},\, B\bigr) &= \overline{\bfE}[\one_{x_{k:m}}(\wh X^{1,n}_{k:m}) \one_B(\wh X^{1,n}_{m:\infty}, \wh X^{2,n}_{m:\infty}) ] \\
&= \overline{\bfE}[\overline{\bfE}[\one_{x_{k:m}}(\wh X^{1,n}_{k:m})\viiva \sA] \one_B(\wh X^{1,n}_{m:\infty}, \wh X^{2,n}_{m:\infty})] \\
&= \overline{\bfE}[\poly{x}{\widehat X^{1,n}_m}^{\beta,\w}(x_{k:m})\one_B(\wh X^{1,n}_{m:\infty},\wh X^{2,n}_{m:\infty})] = \bfE^{\widetilde\Pi^{1,2,n}}\Bigl[\one_B\;\poly{x}{X^{1}_m}^{\beta,\w}(x_{k:m})\Bigr].\qedhere
\end{align*}
\end{proof}

As an immediate consequence of the previous result,  we see that conditioning any limit point on a joint tail event results in Gibbs-DLR marginals.

\begin{corollary}\label{cor:tail-dlr}
In the setting of Lemma~\ref{lem:concat-cond}, let $\widetilde\Pi^{1,2}$ be any weak limit point of
$\{\widetilde\Pi^{1,2,n} : n\in \bbZ_{\geq k}\}$ on $\pathsp_x\times\pathsp_x$.
Let $A\in\ptail^{(2)}$ satisfy $\widetilde\Pi^{1,2}(A)>0$, and for $i\in\{1,2\}$, let $\nu_i$ be the probability measure on $(\pathsp_x,\pathsa_{k:\infty})$ determined by $\nu_i(\aabullet) = \widetilde\Pi^{1,2}(X^i\in \aabullet\,\viiva A)$. Then $\nu_i\in\DLR_x^{\beta,\w}$.
\end{corollary}

\begin{proof}
Fix $m>k$. By Lemma~\ref{lem:concat-cond}, the identity \eqref{eq:concat-cond} holds for each $n>m$ under
$\widetilde\Pi^{1,2,n}$. Considering cylinder sets first and then applying the monotone class theorem \cite[Appendix Theorem 4.3]{Eth-Kur-86}, we see that the identity in \eqref{eq:concat-cond} also holds for any weak limit point $\widetilde\Pi^{1,2}$.

Now take any $A\in\ptail^{(2)}$, $i\in\{1,2\}$, and a finite path segment $x_{k:m}$. Set
\[
C = A\cap\{X^i_m=x_m\}.
\]
We have $C\in \pathsa^{(2)}_{m:\infty}$ and therefore
\[
\widetilde\Pi^{1,2}\bigl(X^i_{k:m}=x_{k:m},\,A\bigr)
=
\bfE^{\widetilde\Pi^{1,2}}\Bigl[\one_{A\cap\{X^i_m=x_m\}}\;\poly{x}{x_m}^{\beta,\w}(x_{k:m})\Bigr]
=
\poly{x}{x_m}^{\beta,\w}(x_{k:m})\;\widetilde\Pi^{1,2}(A, X^i_m=x_m).
\]
Dividing by $\widetilde\Pi^{1,2}(A)$ verifies the DLR equations \eqref{eq:Gibbs} for $\nu_i$. 
\end{proof}

We next record that any weak limit point as in Corollary \ref{cor:tail-dlr} couples together the two measures in such a way that the paths are either equal or strictly ordered, almost surely. This essentially follows from the tree structure of the finite paths under $Q_x^{\beta,\w}$. In the statement below, recall that $X^1 \prec X^2$ means that $X^1 \preceq X^2$ and there exists an index $N$ so that $X^1_{k:N}=X^2_{k:N}$ and for all $n \geq N+1$, $X^1_n \cdot e_1 < X^2_n \cdot e_1$. Similarly, $X^1_{k:m} \prec X^2_{k:m}$ means that there exists $\ell\in \{k,\dots, m-1\}$ so that $X^1_{k:\ell}=X^2_{k:\ell}$ and for $n \in \{\ell+1,\dots,m\}$, $X^1_n \cdot e_1 < X^2_n \cdot e_1$.

\begin{lemma}\label{lem:limit-ordered}
Fix $\w\in\Omega_0$ and $x\in\Z^2$, and let $\wt\Pi^{1,2,n}$ be as in the statement of Lemma \ref{lem:concat-cond}.
Let $\wt\Pi^{1,2}$ be any weak limit point of $\{\wt\Pi^{1,2,n} : n \in \bbZ_{\geq k}\}$ in $\sM_1(\pathsp_x\times\pathsp_x)$.
Then
\[
\wt\Pi^{1,2}\bigl(X^1\prec X^2 \textup{ or }X^1 = X^2 \textup{ or } X^1\succ X^2\bigr)=1.
\]
\end{lemma}

\begin{proof}
For each $m>k$ define the closed cylinder event
\[
C_m=\bigl\{X^1_{k:m}\prec X^2_{k:m}\bigr\}\ \cup\ \bigl\{X^1_{k:m} = X^2_{k:m}\bigr\}\ \cup\ \bigl\{X^1_{k:m}\succ X^2_{k:m}\bigr\}.
\]
Fix $m>k$ and take $n\ge m$. Under $\wt\Pi^{1,2,n}$, the first $m$ steps of each coordinate path coincide with the first $m$ steps of the corresponding tree
path $\sigma^{x,X^i_n}$. Because these are paths in a tree,
\[
\wt\Pi^{1,2,n}(C_m)=1\qquad \text{for all }n\ge m.
\]
Passing to any sub-sequential limit, we have for all $m$, $\wt \Pi^{1,2}(C_m)=1$. The result follows.
\end{proof}

We now turn toward the main result of this section. Recall that $\preceq$ denotes the stochastic ordering on path measures, which is induced by the partial order $\preceq$ on paths. 
\begin{proposition}\label{prop:DLR-ordered}
For all $\w\in\Omega_0$, all $\beta\in(0,\infty)$, all $x\in\bbZ^2,$ and all $\Pi^1,\Pi^2 \in \ext \DLR_x^{\beta,\w}$, we have 
\[
\Pi^1 \precneq \Pi^2, \qquad \Pi^2 \precneq \Pi^1, \qquad \text{ or }\qquad\Pi^1=\Pi^2.
\]
\end{proposition}

\begin{proof}
We have $\delta_{x+\bbZ_{\geq 0}e_2}\precneq\Pi\precneq\delta_{x+\bbZ_{\geq 0}e_1}$ for all $\Pi\in\DLR_x^{\beta,\w}\;\setminus\;\{\delta_{x+\bbZ_{\geq 0} e_2}, \delta_{x+\bbZ_{\geq 0} e_1}\}$. These excluded trivial measures are the only extreme non-fully supported DLR solutions rooted at $x$ for $\w \in \Omega_0$.  It is enough to prove the claim for fully supported $\Pi^1,\Pi^2\in\ext\DLR_x^{\beta,\w}$.

Let $\wt\Pi^{1,2,n}$ be as in the statement of Lemma \ref{lem:concat-cond}.
Let $\wt\Pi^{1,2}$ be any weak limit point of $\{\wt\Pi^{1,2,n} : n \in \bbZ_{\geq k}\}$. Define the tail-measurable asymptotic ordering and equality events
\begin{align*}
A_{\prec} &= \{X^1_m \cdot e_1 < X^2_m \cdot e_1 \text{ for all sufficiently large }m\} \in \ptail^{(2)},\\
A_{\succ} &= \{X^1_m \cdot e_1 > X^2_m \cdot e_1 \text{ for all sufficiently large }m\} \in \ptail^{(2)}, \\
A_{=} &= \{X^1_m = X^2_m \text{ for all sufficiently large }m\} \in \ptail^{(2)}.
\end{align*}
By Lemma \ref{lem:limit-ordered}, we have
\[
    \wt \Pi^{1,2}(\{X^1 \prec X^2\}\; \Delta \; A_{\prec}) = \wt \Pi^{1,2}(\{X^1 \succ X^2\}\; \Delta \; A_{\succ}) = \wt \Pi^{1,2}(\{X^1 = X^2\}\; \Delta \; A_{=})=0
\]
and 
\[
\wt \Pi^{1,2}(A_{\prec}) + \wt \Pi^{1,2}(A_{\succ}) + \wt \Pi^{1,2}(A_=)=1.
\]
By the law of total probability (with the understanding that $0$ times an undefined quantity is zero), we have for any $B\in\pathsa_{k:\infty}$ and $i\in\{1,2\}$
\begin{align*}
\Pi^i(B) = \wt \Pi^{1,2}(X^i\in B) &= \wt \Pi^{1,2}(X^i\in B\viiva A_{\prec})\tspb\wt \Pi^{1,2}(A_{\prec}) + \wt \Pi^{1,2}(X^i\in B\viiva A_=)\tspb\wt \Pi^{1,2}(A_=) \\
&\qquad+ \wt \Pi^{1,2}(X^i\in B\viiva A_{\succ})\tspb\wt \Pi^{1,2}(A_{\succ}).
\end{align*}
We can apply Corollary \ref{cor:tail-dlr}, the hypothesis that $\Pi^1,\Pi^2 \in \ext \DLR_x^{\beta,\w}$, and the previous expression to see that if $A\in\{A_{\prec}, A_=, A_{\succ}\}$ has $\wt \Pi^{1,2}(A)>0$, then
\[
\Pi^i(\aabullet) = \wt \Pi^{1,2}(X^i \in\aabullet\viiva A).
\]
Thus, for any $\sigg\in\{\prec,=,\succ\}$, if $\wt \Pi^{1,2}(A_\sig)>0$, then $\wt \Pi^{1,2}(\aabullet \viiva A_\sig)$ is a coupling of $\Pi^1$ and $\Pi^2$ under which $\{X^1 \sigg X^2\}$ holds almost surely. The result follows, since the three events are mutually exclusive. 
\end{proof}

The proof of Proposition \ref{prop:DLR-ordered} gives the following corollary that will be used to prove closedness.

\begin{corollary}\label{cor:ordered-tree-limit}
Fix $\w\in\Omega_0$, $\beta>0$, and $x\in\Z^2$. Let $\Pi^1,\Pi^2\in\ext\DLR_x^{\beta,\w}$ be fully supported with $\Pi^1\precneq \Pi^2$. In the setting of Lemma \ref{lem:concat-cond}, any weak limit point
$\widetilde\Pi^{1,2}$ of $\{\widetilde\Pi^{1,2,n}:n\in\bbZ_{\geq k}\}$ satisfies
\[
\widetilde\Pi^{1,2}(X^1\prec X^2)=1.
\]
\end{corollary}

\begin{remark} \label{rmk:BFS5}
A similar result showing almost sure eventual separation of paths generated by the Gibbs-DLR measures coming from the direction-indexed Busemann process was previously proved in Theorem 3.10 of \cite{Bat-Fan-Sep-25}. This was in the usual monotone coupling obtained by using common i.i.d.~Uniform[0,1] random variables to sample from the transition probabilities. To guarantee that the Gibbs-DLR measures considered there are extreme, the shape function was assumed to satisfy the property that endpoints of linear segments are differentiability points (Lemma 5.12 of \cite{Jan-Ras-18-arxiv}).
\end{remark}

\subsection{$\ext \DLR_x^{\beta,\w}$ is closed}
We next turn to showing that the set of extremal Gibbs measures is closed. The next few results are deterministic statements coming from the Gibbs property. 
\begin{proposition} {\rm\cite[Theorem 5.2]{Jan-Ras-20-aop}} \label{prop:DLR-transition}
For all $\w\in\Omega$, all $x\in\bbZ^2$, and all $\beta\in(0,\infty)$, if $\Pi\in\DLR_x^{\beta,\w}$ is fully supported, then for all $y \geq x$ and all $m > y\cdot(e_1+e_2)$, the transition probabilities $\pi^{\Pi}$ satisfy
\[
\pi^{\Pi}(y,y+e_i)=\Pi(y+e_i \viiva y) = e^{\beta \w_y}\frac{\bfE^{\Pi}\left[\frac{\PF{y+e_i}{X_m}^{\beta,\w}}{\PF{x}{X_m}^{\beta,\w}}\right]}{\bfE^{\Pi}\left[\frac{\PF{y}{X_m}^{\beta,\w}}{\PF{x}{X_m}^{\beta,\w}}\right]}
\]
\end{proposition}
We next record the fact that a Gibbs measure is extreme if and only if it is tail-trivial. Once again, this is a deterministic statement. The proof is standard and can be found in any textbook covering Gibbs measures in statistical mechanical models. See, for example, Corollary 7.4 in \cite{Geo-88} or Theorem 7.24 in \cite{Ras-Sep-15-ldp}.

\begin{lemma}\label{lem:DLR-triv} For all $\w\in\Omega$, $x\in\bbZ^2$, and $\beta\in(0,\infty),$ $\Pi \in \ext \DLR_x^{\beta,\w}$ if and only if $\Pi \in \DLR_x^{\beta,\w}$ and for all $A \in \ptail$, $\Pi(A)\in\{0,1\}.$
\end{lemma}
Our next lemma reduces this condition to one which is easier to check, involving finite paths.
\begin{lemma}\label{lem:DLR-triv-finite}
For all $\w\in \Omega$,  $x\in\bbZ^2$, and $\beta\in(0,\infty)$, $\Pi \in \ext \DLR_x^{\beta,\w}$ if and only if $\Pi \in \DLR_x^{\beta,\w}$ and for all finite paths $x_{k:m}$ with $x_k = x$, 
\begin{align}
\Pi(x_{k:m}\viiva \ptail) = \Pi(x_{k:m}) \qquad \Pi\text{-almost surely}.\label{eq:finite-trivial}
\end{align}
\end{lemma}
\begin{proof}
If $\Pi\in \ext \DLR_{x}^{\beta,\w}$, then the claim follows from tail triviality, so it suffices to show the reverse implication. 

Consider the class of bounded $\pathsa_{k:\infty}$-measurable random variables $Y$ for which
\begin{align}
\bfE^{\Pi}[\one_A Y] = \bfE^{\Pi}[Y]\Pi(A)\label{eq:Y-trivial}
\end{align}
holds for all $A\in \ptail$.  The hypothesis \eqref{eq:finite-trivial} says that this holds for cylinder events $\{X_{k:m}=x_{k:m}\}$. Because $\pathsa_{k:\infty}$ is generated by these cylinder events, the monotone class theorem \cite[Appendix Theorem 4.3]{Eth-Kur-86} implies that if \eqref{eq:finite-trivial} holds for all cylinder events, then \eqref{eq:Y-trivial} holds for all bounded $\pathsa_{k:\infty}$-measurable random variables $Y$. Tail triviality now follows by considering $Y=\one_A$ where $A\in \ptail.$
\end{proof}

We next record a quick consequence of Lemma \ref{lem:DLR-mg}, Proposition \ref{prop:DLR-transition}, Lemma \ref{lem:DLR-triv}, and martingale convergence.

\begin{corollary}\label{cor:DLR-mglim}
For all $\w\in\Omega$, $x\in\bbZ^2$, and $\beta \in (0,\infty)$ if $\Pi \in \ext \DLR_x^{\beta,\w}$ is fully supported, then for all $y \geq x$ and $i\in\{1,2\}$, the transition probabilities satisfy
\[
\pi^{\Pi}_{y,y+e_i} = e^{\beta \w_y}\lim_{n\to\infty} \frac{\PF{y+e_i}{X_n}}{\PF{y}{X_n}} \qquad \text{ }\Pi\text{-almost surely}.
\]
\end{corollary}
We will use
the following characterization of extremality several times.

\begin{lemma}\label{lem:conditional-transition-sufficient}
For $\w\in\Omega$, $\beta\in(0,\infty)$ and fully supported $\Pi\in\DLR_x^{\beta,\w}$, $\Pi\in \ext \DLR_x^{\beta,\w}$ if and only if
for all $y \geq x$, all $j \in \{1,2\}$, and all $m>y\cdot(e_1+e_2),$
\be\begin{aligned}
\pi^{\Pi}_{y,y+e_j}
\bfE^\Pi\biggl[\frac{\PF{y}{X_m}^{\beta,\w}}{\PF{x}{X_m}^{\beta,\w}}\,\bigg\vert\,\ptail\biggr]
=
e^{\beta \w_y}
\bfE^\Pi\biggl[\frac{\PF{y+e_j}{X_m}^{\beta,\w}}{\PF{x}{X_m}^{\beta,\w}}\,\bigg\vert\,\ptail\biggr].
\end{aligned}\label{eq:conditional-transition-sufficient}\ee
\end{lemma}
\begin{proof}
If $\Pi\in \ext \DLR_x^{\beta,\w}$ is fully supported, then \eqref{eq:conditional-transition-sufficient} holds by Proposition \ref{prop:DLR-transition} and the tail triviality coming from Lemma \ref{lem:DLR-triv}.

Conversely, suppose \eqref{eq:conditional-transition-sufficient} holds. We will first show by induction on the length of the path that for all  up-right paths $x_{k:\ell}$ from $x=x_k$ to some $y=x_\ell$, we have 
\be\Pi(x_{k:\ell})= \prod_{r=k}^{\ell-1} \pi^\Pi_{x_r, x_{r+1}} = e^{\beta \sum_{r=k}^{\ell-1}\w_{x_r}} \bfE^\Pi\biggl[\frac{\PF{y}{X_m}^{\beta,\w}}{\PF{x}{X_m}^{\beta,\w}}\,\bigg\vert\,\ptail\biggr].
\label{eq:tailcomp}
\ee
Note that we do not know that the conditional expectations given the tail $\sigma$-algebra are almost surely strictly positive, so this is not immediate from a telescoping product. The claim is trivial when the terminal point is $x_k=x$, since $\bfE^\Pi[\PF{x}{X_m}^{\beta,\w}/\PF{x}{X_m}^{\beta,\w}\viiva\ptail]=1$. Suppose the claim holds up to length $r$. Choose $j$ so that $x_{r+1}=x_r+e_j$. Multiplying the induction hypothesis by $\pi^\Pi_{x_r,x_{r+1}}$ and applying \eqref{eq:conditional-transition-sufficient} with $y=x_r$ gives
\[
\prod_{q=k}^{r}\pi^\Pi_{x_q,x_{q+1}} =e^{\beta \sum_{q=k}^{r}\w_{x_q}} \bfE^\Pi\biggl[\frac{\PF{x_{r+1}}{X_m}^{\beta,\w}}{\PF{x}{X_m}^{\beta,\w}}\,\bigg\vert\,\ptail\biggr].
\]

On the other hand, the Gibbs condition says that for $m>\ell$,
\[
\Pi(x_{k:\ell}\viiva\pathsa_{m:\infty}) = \poly{x}{X_m}^{\beta,\w}(x_{k:\ell}) = e^{\beta \sum_{r=k}^{\ell-1}\w_{x_r}}\frac{\PF{y}{X_m}^{\beta,\w}}{\PF{x}{X_m}^{\beta,\w}} \qquad \Pi\text{-almost surely.}
\]
Backward martingale convergence then gives that $\Pi$-almost surely, \begin{align*}
\Pi(x_{k:\ell}\viiva \ptail)&=\lim_{m\to\infty}\Pi(x_{k:\ell}\viiva\pathsa_{m:\infty}) = \lim_{m\to\infty} e^{\beta \sum_{r=k}^{\ell-1}\w_{x_r}}\frac{\PF{y}{X_m}^{\beta,\w}}{\PF{x}{X_m}^{\beta,\w}} \\
&=\lim_{m\to\infty} e^{\beta \sum_{r=k}^{\ell-1}\w_{x_r}}\bfE^\Pi\biggl[\frac{\PF{y}{X_m}^{\beta,\w}}{\PF{x}{X_m}^{\beta,\w}}\,\bigg\vert\,\ptail\biggr] \stackrel{\eqref{eq:tailcomp}}{=} \Pi(x_{k:\ell}).
\end{align*}
In the third equality, we used that the partition function ratios converge both almost surely and in $L^1(\Pi)$. 
The result now follows from Lemma \ref{lem:DLR-triv-finite}.
\end{proof}

The proof that the set of extreme Gibbs-DLR measures is closed uses a countable path generalization of the coupling of two Gibbs-DLR measures in the previous section. The proofs of the next two results are essentially verbatim identical to those in the two-path case.

\begin{lemma}\label{lem:countable-concat-cond}
Fix $\w\in\Omega_0$, $\beta>0$, and $x\in\Z^2$, and abbreviate
$k=x\cdot(e_1+e_2)$.
Let $(\Pi^i)_{i\ge1}\subset\DLR_x^{\beta,\w}$ be fully supported and let $\Pi^{1:\infty}$ be any coupling of $(\Pi^i)_{i\ge1}$. Denote by $Q_x^{\beta,\w}\otimes \Pi^{1:\infty}$ the product measure. For each integer $n>k$, define the concatenated paths
\[
\widehat X^{i,n} \;=\; \sigma^{x,X_n^i}\concat X^i_{n:\infty},
\qquad i\ge1,
\]
and let $\widetilde\Pi^{1:\infty,n}$ denote the law of $(\widehat X^{1,n},\widehat X^{2,n},\dots)$ on $\pathsp_x^{\bbN}$ under $Q_x^{\beta,\w}\otimes \Pi^{1:\infty}$.

Then for any integers $n>m>k$, any finite up-right path segment $x_{k:m}$ with $x_k=x$, any integer $i\ge1$, and any event $B\in\pathsa^{(\infty)}_{m:\infty}$,
\[
\widetilde\Pi^{1:\infty,n}\bigl(X^i_{k:m}=x_{k:m},\,B\bigr)
=
\bfE^{\widetilde\Pi^{1:\infty,n}}\Bigl[\one_B\;\poly{x}{X^i_m}^{\beta,\w}(x_{k:m})\Bigr].
\]
Equivalently, under $\widetilde\Pi^{1:\infty,n}$, the conditional distribution of $X^i_{k:m}$ given $\pathsa^{(\infty)}_{m:\infty}$ is $\poly{x}{X^i_m}^{\beta,\w}$.
\end{lemma}

As in the two-path case, conditioning any limit point on a joint tail event results in Gibbs-DLR marginals.

\begin{corollary}\label{cor:countable-tail-dlr}
In the setting of Lemma~\ref{lem:countable-concat-cond}, let $\widetilde\Pi^{1:\infty}$ be any weak limit point of $\{\widetilde\Pi^{1:\infty,n} : n\in \bbZ_{\geq k}\}$. Let $A\in \ptail^{(\infty)}$ satisfy $\widetilde\Pi^{1:\infty}(A)>0$, and for $i\geq1$, let $\nu_i$ be the probability measure on $(\pathsp_x,\pathsa_{k:\infty})$ determined by $\nu_i(\aabullet) = \widetilde\Pi^{1:\infty}(X^i\in \aabullet\,\viiva A)$. Then $\nu_i\in\DLR_x^{\beta,\w}$.
\end{corollary}

%

Projecting to two dimensional marginals, Corollary \ref{cor:ordered-tree-limit} implies the following result.
\begin{corollary}\label{cor:countable-ordered}
In the setting of Lemma \ref{lem:countable-concat-cond}, suppose additionally that $(\Pi^i)_{i\geq1}\subset \ext \DLR_x^{\beta,\w}$ and that $\Pi^j\precneq \Pi^i$ for all $i<j$. Let $\widetilde\Pi^{1:\infty}$ be any weak limit point of $\{\widetilde\Pi^{1:\infty,n}:n\in\bbZ_{\geq k}\}$ on $\pathsp_x^{\bbN}$. Then
\[
\widetilde\Pi^{1:\infty}(X^j\prec X^i\text{ for all }i<j)=1.
\]
\end{corollary}

The next lemma rules out the possibility that a strictly monotone limit has an atom on the trivial path that is consistent with the monotonicity without being the Dirac mass at that path.

\begin{lemma}\label{lem:monotone-trivial-atom}
Fix $\w\in\Omega_0$, $\beta>0$, and $x\in\Z^2$, and abbreviate $k=x\cdot(e_1+e_2)$. Let $(\Pi^i)_{i\ge1}\subset\ext\DLR_x^{\beta,\w}$ be fully supported and suppose that $\Pi^i\to\Pi$ weakly.
\begin{enumerate}[label={\rm(\roman*)}, ref={\rm\roman*}]
\item\label{lem:monotone-trivial-atom:decreasing} If $\Pi^{i+1}\precneq\Pi^i$ for all $i\ge1$ and 
\[
\Pi(X_{k:\infty}=x+e_2\bbZ_{\geq0})>0,
\]
then $\Pi=\delta_{x+e_2\bbZ_{\geq0}}$.
\item \label{lem:monotone-trivial-atom:increasing} If $\Pi^i\precneq\Pi^{i+1}$ for all $i\ge1$ and
\[
\Pi(X_{k:\infty}=x+e_1\bbZ_{\geq0})>0,
\]
then $\Pi=\delta_{x+e_1\bbZ_{\geq0}}$.
\end{enumerate}
\end{lemma}
\begin{proof}
We prove \eqref{lem:monotone-trivial-atom:decreasing}, with the proof of \eqref{lem:monotone-trivial-atom:increasing} being similar.
Let $\Pi^{1:\infty}$ be any coupling of $(\Pi^i)_{i\ge1}$, and let $\widetilde\Pi^{1:\infty}$ be any weak limit point of $\{\widetilde\Pi^{1:\infty,n}:n\in\bbZ_{\geq k}\}$ from Lemma \ref{lem:countable-concat-cond}. By Corollary \ref{cor:countable-ordered}, we have $\widetilde\Pi^{1:\infty}(X^j\prec X^i\text{ for all }i<j)=1$.

For each $n\ge k$, define $X_n=\lim_{i\to\infty}X^i_n$. The limit exists by monotonicity. Then $X=(X_n)_{n\ge k}\in\pathsp_x$ and $X$ has law $\Pi$. Let $A$ be the event that for every $N\ge k$, there exists $n\ge N$ such that $X^i_n=x+(n-k)e_2$ for all sufficiently large $i$; we have $A\in\ptail^{(\infty)}$. Since $X^j\prec X^i$ for all $i<j$ with probability one, we have 
\[\widetilde\Pi^{1:\infty}\bigl(A\,\Delta\,\{X_{k:\infty}=x+e_2\bbZ_{\geq0}\}\bigr)=0.\]

Assume $\Pi(X_{k:\infty}=x+e_2\bbZ_{\geq0})>0$. This and the previous display give $\widetilde\Pi^{1:\infty}(A)>0$.
By the law of total probability, with the understanding that $0$ times an undefined quantity is zero, for any $B\in\pathsa_{k:\infty}$ and $i\geq1$,
\begin{align}\label{decomp}
\Pi^i(B)=\widetilde\Pi^{1:\infty}(X^i\in B\viiva A)\widetilde\Pi^{1:\infty}(A)+\widetilde\Pi^{1:\infty}(X^i\in B\viiva A^c)\widetilde\Pi^{1:\infty}(A^c).
\end{align}
By Corollary \ref{cor:countable-tail-dlr}, each positive-probability conditional marginal appearing on the right-hand side is in $\DLR_x^{\beta,\w}$. Since $\Pi^i\in\ext\DLR_x^{\beta,\w}$ and  $\widetilde\Pi^{1:\infty}(A)>0$, it follows that $\Pi^i(\aabullet)=\widetilde\Pi^{1:\infty}(X^i\in\aabullet\viiva A)$ for all $i\geq1$. On the event $A$, $X^i\to x+e_2\bbZ_{\geq0}$ coordinatewise, so since $\Pi^i \to \Pi$, we conclude that $\Pi=\delta_{x+e_2\bbZ_{\geq0}}$.
\end{proof}

We now prove the main result of this section, namely that the set of extremal Gibbs measures is closed.

\begin{proposition}\label{prop:DLR-closed}
For all $\w \in \Omega_0$, $x\in\bbZ^2$, and $\beta\in(0,\infty)$, $\ext \DLR_x^{\beta,\w}$ is closed in $\sM_1(\pathsp_x,\pathsa_{k:\infty})$.
\end{proposition}
\begin{proof}
Take any convergent sequence $(\Pi^n:n\ge1)\subset\ext\DLR_x^{\beta,\w}$ and call the limit $\Pi$. Since $\DLR_x^{\beta,\w}$ is closed, $\Pi\in\DLR_x^{\beta,\w}$. By Proposition \ref{prop:DLR-ordered}, it suffices to pass to a strictly monotone subsequence. We treat the strictly decreasing case with the strictly increasing case being similar. Abbreviate $k=x\cdot(e_1+e_2)$. After discarding at most the first term (and reindexing), we may assume that each $\Pi^i$ is fully supported.

Since $\Pi\preceq\Pi^1$, and since $\Pi^1$ is extremal and fully supported, $\Pi(X_{k:\infty}=x+e_1\bbZ_{\geq0})=0$. If $\Pi$ is not fully supported, Lemma \ref{lem:Omega0} implies that $\Pi$ is a convex combination of the two trivial measures, and hence $\Pi=\delta_{x+e_2\bbZ_{\geq0}}\in\ext\DLR_x^{\beta,\w}$. We may therefore assume that $\Pi$ is fully supported. Lemma \ref{lem:monotone-trivial-atom} then gives $\Pi(X_{k:\infty}=x+e_2\bbZ_{\geq0})=0$.

Let $\Pi^{1:\infty}$ be any coupling of $(\Pi^i)_{i\ge1}$, and let $\wt\Pi=\widetilde\Pi^{1:\infty}$ be any weak limit point of $\{\widetilde\Pi^{1:\infty,n}:n\in\bbZ_{\geq k}\}$ from Lemma \ref{lem:countable-concat-cond}. By Corollary \ref{cor:countable-ordered}, $\wt\Pi(X^j\prec X^i\text{ for all }i<j)=1$. Define $X_n=\lim_{i\to\infty}X^i_n$. Then $X=(X_n)_{n\ge k}\in\pathsp_x$ and $X$ has law $\Pi$.

By Lemma \ref{lem:DLR-mg} and backward martingale convergence, for all $y\in x+\Z_{\geq 0}^2$ and all $m>y\cdot(e_1+e_2)$,
\[
\lim_{n\to\infty}\frac{\PF{y}{X_n}^{\beta,\w}}{\PF{x}{X_n}^{\beta,\w}}
=
\bfE^\Pi\biggl[\frac{\PF{y}{X_m}^{\beta,\w}}{\PF{x}{X_m}^{\beta,\w}}\,\bigg\vert\,\ptail\biggr],
\qquad \Pi\text{-almost surely and in }L^1(\Pi).
\]
The same identity holds $\wt\Pi$-almost surely.

Since $\Pi$ and $\Pi^i$, $i\in\Z_{>0}$, are all fully supported, Lemma \ref{lem:Omega0} and the coupling imply that, $\wt\Pi$-almost surely, for all $y\in x+\Z_{\geq 0}^2$ and all $i\ge1$, $X_n\ge y+e_1+e_2$ and $X^i_n\ge y+e_1+e_2$ for all sufficiently large $n$.

Since $X^i\succeq X$, Lemma B.2 in \cite{Jan-Ras-20-aop} implies that for each $i\ge1$ and all sufficiently large $n$,
\[
\frac{\PF{y+e_1}{X^i_n}^{\beta,\w}}{\PF{y}{X^i_n}^{\beta,\w}}\ge\frac{\PF{y+e_1}{X_n}^{\beta,\w}}{\PF{y}{X_n}^{\beta,\w}}
\qquad\text{and}\qquad
\frac{\PF{y+e_2}{X^i_n}^{\beta,\w}}{\PF{y}{X^i_n}^{\beta,\w}}\le\frac{\PF{y+e_2}{X_n}^{\beta,\w}}{\PF{y}{X_n}^{\beta,\w}}.
\]
Recall that $\Pi^i$ is assumed to be extremal. Cross-multiply, divide by $\PF{x}{X_n^i}^{\beta,\w}\PF{x}{X_n}^{\beta,\w}$, and apply Lemma \ref{lem:DLR-mg} and backward martingale convergence to get
\[
\bfE^{\Pi^i}\biggl[\frac{\PF{y+e_1}{X_m}^{\beta,\w}}{\PF{x}{X_m}^{\beta,\w}}\biggr]\bfE^\Pi\biggl[\frac{\PF{y}{X_m}^{\beta,\w}}{\PF{x}{X_m}^{\beta,\w}}\,\bigg\vert\,\ptail\biggr]
\ge
\bfE^\Pi\biggl[\frac{\PF{y+e_1}{X_m}^{\beta,\w}}{\PF{x}{X_m}^{\beta,\w}}\,\bigg\vert\,\ptail\biggr]\bfE^{\Pi^i}\biggl[\frac{\PF{y}{X_m}^{\beta,\w}}{\PF{x}{X_m}^{\beta,\w}}\biggr]
\]
and
\[
\bfE^{\Pi^i}\biggl[\frac{\PF{y+e_2}{X_m}^{\beta,\w}}{\PF{x}{X_m}^{\beta,\w}}\biggr]\bfE^\Pi\biggl[\frac{\PF{y}{X_m}^{\beta,\w}}{\PF{x}{X_m}^{\beta,\w}}\,\bigg\vert\,\ptail\biggr]
\le
\bfE^\Pi\biggl[\frac{\PF{y+e_2}{X_m}^{\beta,\w}}{\PF{x}{X_m}^{\beta,\w}}\,\bigg\vert\,\ptail\biggr]\bfE^{\Pi^i}\biggl[\frac{\PF{y}{X_m}^{\beta,\w}}{\PF{x}{X_m}^{\beta,\w}}\biggr].
\]
Here, we used that the marginals of $X^i$ and $X$ under $\wt \Pi$ are $\Pi^i$ and $\Pi$ as well as the tail triviality of $\Pi^i$. By Proposition \ref{prop:DLR-transition} applied to $\Pi^i$, and since $\Pi^i$ is fully supported, this gives $\Pi$-almost surely,
\[
e^{-\beta\w_y}\pi^{\Pi^i}_{y,y+e_1}\bfE^\Pi\biggl[\frac{\PF{y}{X_m}^{\beta,\w}}{\PF{x}{X_m}^{\beta,\w}}\,\bigg\vert\,\ptail\biggr]\ge
\bfE^\Pi\biggl[\frac{\PF{y+e_1}{X_m}^{\beta,\w}}{\PF{x}{X_m}^{\beta,\w}}\,\bigg\vert\,\ptail\biggr]
\]
and
\[
e^{-\beta\w_y}\pi^{\Pi^i}_{y,y+e_2}\bfE^\Pi\biggl[\frac{\PF{y}{X_m}^{\beta,\w}}{\PF{x}{X_m}^{\beta,\w}}\,\bigg\vert\,\ptail\biggr]\le
\bfE^\Pi\biggl[\frac{\PF{y+e_2}{X_m}^{\beta,\w}}{\PF{x}{X_m}^{\beta,\w}}\,\bigg\vert\,\ptail\biggr].
\]
Multiplying by $e^{\beta\w_y}$ and then letting $i\to\infty$, convergence of transition probabilities gives 
\[
\pi^{\Pi}_{y,y+e_1}\bfE^\Pi\biggl[\frac{\PF{y}{X_m}^{\beta,\w}}{\PF{x}{X_m}^{\beta,\w}}\,\bigg\vert\,\ptail\biggr]-e^{\beta\w_y}\bfE^\Pi\biggl[\frac{\PF{y+e_1}{X_m}^{\beta,\w}}{\PF{x}{X_m}^{\beta,\w}}\,\bigg\vert\,\ptail\biggr]\ge 0
\]
and
\[
\pi^{\Pi}_{y,y+e_2}\bfE^\Pi\biggl[\frac{\PF{y}{X_m}^{\beta,\w}}{\PF{x}{X_m}^{\beta,\w}}\,\bigg\vert\,\ptail\biggr]-e^{\beta\w_y}\bfE^\Pi\biggl[\frac{\PF{y+e_2}{X_m}^{\beta,\w}}{\PF{x}{X_m}^{\beta,\w}}\,\bigg\vert\,\ptail\biggr]\le 0.
\]
Taking $\Pi$ expectations and applying Proposition \ref{prop:DLR-transition} shows that both expressions have expectation zero. Hence the inequalities are equalities $\Pi$-almost surely. The result now follows from Lemma \ref{lem:conditional-transition-sufficient}.
\end{proof}

\subsection{Coalescence and extremality}\label{sec:coal-ext}

\begin{proof}[Proof of Theorem \ref{thm:coalext}]
Suppose that $\Pi\in \DLR_x^{\beta,\w}$ is fully supported and admits a coalescing self-coupling $\wt \Pi$ in environment $\w$. Let $A \in \ptail$ and observe that the assumption that $X^x \coal X^y$ $\wt \Pi$-almost surely implies that
\[
\one_A(X^x) = \one_A(X^y) \qquad \wt \Pi\text{-almost surely.}
\]
Taking $\wt \Pi$ expectations of this equality, we see that $\wt \Pi(X^x \in A) = \wt \Pi(X^y \in A)$ for all $y \geq x$. It then follows that $\Pi(A) = \Pi_y(A)$ for all $y \geq x$. Call $k=x \cdot (e_1+e_2)$. Next, we apply in sequence (i) the fact that $A \in \ptail \subset \pathsa_{k:\infty}$, (ii) martingale convergence, (iii) the Gibbs property and $A\in\ptail\subset\pathsa_{n:\infty}$, and (iv) the just-proven fact that $\Pi(A) = \Pi(A \viiva y)$ for all $y\geq x$ to conclude: 
\[
\one_A = \Pi(A\viiva \pathsa_{k:\infty}) = \lim_{n\to\infty}\Pi(A \viiva \pathsa_{k:n}) = \lim_{n\to\infty} \Pi(A \viiva X_n) = \Pi(A) \qquad \Pi\text{-almost surely.}
\]
Thus $\Pi(A) \in \{0,1\}.$ Then $\Pi$ is extremal by Lemma \ref{lem:DLR-triv}.

Conversely, suppose that $\Pi \in \ext \DLR_x^{\beta,\w}$ is fully supported and fix $y \geq x$. Take $n > y \cdot(e_1+e_2)$ and let $A\in \pathsa_{n:\infty}$. 
We compute 
\[
\Pi_y(A) = \frac{\bfE^{\Pi}[\one_A \one_{\{y\}\in X}]}{\Pi(y)} = \bfE^{\Pi}\left[\one_A \frac{\Pi(y\viiva \pathsa_{n:\infty})}{\Pi(y)}\right].
\]
It follows from this identity that the restriction of $\Pi_y$ to $\pathsa_{n:\infty}$ is absolutely continuous with respect to the restriction of $\Pi$ to $\pathsa_{n:\infty}$ with Radon-Nikodym derivative
\[
\frac{d\Pi_y|_{\pathsa_{n:\infty}}}{d\Pi|_{\pathsa_{n:\infty}}} = \frac{\Pi(y\viiva \pathsa_{n:\infty})}{\Pi(y)}.
\]
By backward martingale convergence, this Radon-Nikodym derivative converges $\Pi$-almost surely and in $L^1(\Pi)$ as $n\to\infty$ to 
\[
\frac{\Pi(y\viiva \ptail)}{\Pi(y)} = 1,
\]
where in the last step, we used the tail triviality of $\Pi$ coming from Lemma \ref{lem:DLR-triv}. Because the Radon-Nikodym derivative converges to $1$ almost surely and in $L^1(\Pi)$, we see that
\[
\lim_{n\to\infty} \|\Pi_y|_{\pathsa_{n:\infty}} - \Pi|_{\pathsa_{n:\infty}}\|_{\textup{TV}} = 0.
\]
It now follows from Theorem 2.1 in \cite{Gol-79} that, for each $y\ge x$, there exists a coupling $\wt\Pi^y$ of $\Pi_y$ and $\Pi$ under which the two paths coalesce. Disintegrate this coupling with respect to its $\Pi$-coordinate:
\[
\wt\Pi^y(d\gamma^y,d\gamma)=K_y(\gamma,d\gamma^y)\Pi(d\gamma),
\]
where for $y=x$, we set $K_x(\gamma,\aabullet)=\delta_\gamma(\aabullet)$. If $F\subset x+\bbZ_{\geq0}^2$ is finite, define a probability measure on $\prod_{y\in F}\pathsp_y$ by
\[
\wt\Pi^F\bigl((d\gamma^y)_{y\in F}\bigr)=\int_{\pathsp_x}\prod_{y\in F}K_y(\gamma,d\gamma^y)\Pi(d\gamma).
\]
These finite-dimensional distributions are consistent as $F$ varies, so Kolmogorov's extension theorem gives a coupling of the entire family $\{\Pi_y:y\ge x\}$. For each fixed $y\ge x$, the pair $(\gamma^y,\gamma^x)$ has law $\wt\Pi^y$, and hence coalesces almost surely. Taking the countable intersection over $y\ge x$, all paths coalesce with $\gamma^x$ almost surely, and therefore the coupling is coalescing.\end{proof}

\subsection{Extension to the plane and the asymptotic total order}
We next start to work toward the proof of Proposition \ref{prop:global-ext}. We begin with some intermediate lemmas. Note that in the next lemma, we do not restrict to $\Omega_0$. 

\begin{lemma}\label{lem:forward-extreme}
For $\w\in\Omega$, $\beta\in(0,\infty)$, let $\Pi\in \DLR_x^{\beta,\w}$ be fully supported and let $v\in x+\bbZ_{\geq0}^2$. Assume that $\Pi_v\in\ext\DLR_v^{\beta,\w}$ and that 
\[
\Pi(X_n\in v+\bbZ_{\geq0}^2\text{ for all sufficiently large }n)=1.
\]
Then $\Pi\in\ext\DLR_x^{\beta,\w}$.
\end{lemma}
\begin{proof}
We show that $\Pi$ is tail trivial. Let $\ell=v\cdot(e_1+e_2)$ and take $A\in\ptail$. Since $A\in\pathsa_{\ell:\infty}$, $\Pi(A)=0$ implies $\Pi_v(A)=0$.

Conversely, assume $\Pi(A)>0$. By the hypothesis, there exists $n> \ell$ such that
\[
\Pi(A,\ X_n\in v+\bbZ_{\geq0}^2)>0.
\]
Since $A\in\pathsa_{n:\infty}$, the Gibbs property gives
\[
\Pi(A,\ X_\ell=v)
=
\bfE^\Pi\bigl[\one_A\,\Pi(X_\ell=v\viiva\pathsa_{n:\infty})\bigr]
=
\bfE^\Pi\bigl[\one_A\,\poly{x}{X_n}^{\beta,\w}(X_\ell=v)\bigr].
\]
On the event $\{X_n\in v+\bbZ_{\geq0}^2\}$, the point-to-point polymer measure $\poly{x}{X_n}^{\beta,\w}$ gives strictly positive probability to paths that pass through $v$ at time $\ell$. Hence $\Pi(A,\ X_\ell=v)>0$, and therefore $\Pi_v(A)>0$.

Thus $\Pi(A)=0$ if and only if $\Pi_v(A)=0$ for all $A\in\ptail$. Since $\Pi_v$ is extremal, its tail is trivial. Therefore $\Pi$ is tail trivial. Lemma \ref{lem:DLR-triv} gives $\Pi\in\ext\DLR_x^{\beta,\w}$.
\end{proof}

\begin{lemma}\label{lem:Busgeo}
For $\w\in\Omega_0$, $\beta\in(0,\infty)$, $x,y,z\in\bbZ^2$, and fully supported $\Pi\in\ext\DLR_x^{\beta,\w}$, the following limit exists $\Pi$-almost surely
\be\Busgeo_\Pi(y,z)=\lim_{n\to\infty}\frac{1}{\beta}\Bigl(\log \PF{y}{X_n}^{\beta,\w}-\log\PF{z}{X_n}^{\beta,\w}\Bigr)=\lim_{n\to\infty}\frac{1}{\beta}\Bigl(\log \bfE^\Pi\Bigl[\frac{\PF{y}{X_n}^{\beta,\w}}{\PF{x}{X_n}^{\beta,\w}}\Bigr]-\log \bfE^\Pi\Bigl[\frac{\PF{z}{X_n}^{\beta,\w}}{\PF{x}{X_n}^{\beta,\w}}\Bigr]\Bigr).\label{eq:Busgeo}\ee
and satisfies for all $y\in\bbZ^2$
\be
\sum_{i=1}^2 e^{\beta \w_y-\beta \Busgeo_\Pi(y,y+e_i)}=1.\label{eq:Busgeo-lemma-recovery}
\ee
If $y \geq x$, then for $i\in\{1,2\}$
\be
\pi^{\Pi}_{y,y+e_i} = \Pi(y+e_i\viiva y) = e^{\beta \w_y - \beta \Busgeo_{\Pi}(y,y+e_i)}\label{eq:Geobus-transition-equal}.
\ee
\end{lemma}

\begin{proof}
By Lemma \ref{lem:DLR-mg} and backward martingale convergence, for all $m\ge (x \vee y)\cdot(e_1+e_2)$,
\be \lim_{n\to\infty}   \frac{\PF{y}{ X_n}^{\beta,\w}}{\PF{x}{X_n}^{\beta,\w}}  =  \bfE^{\Pi}\biggl[ \,  \frac{\PF{y}{X_m}^{\beta,\w}}{\PF{x}{X_m}^{\beta,\w}} \,\bigg\vert\, \ptail\,\biggr]  =\bfE^{\Pi}\biggl[ \,  \frac{\PF{y}{X_m}^{\beta,\w}}{\PF{x}{X_m}^{\beta,\w}}\,\biggr],\qquad\text{$\Pi$-almost surely and in $L^1(\Pi)$}. \label{eq:Blimtmp}  
\ee 
The second equality above uses the triviality of $\ptail$ under $\Pi \in \ext \DLR_x^{\beta,\w}$ coming from Lemma \ref{lem:DLR-triv}. Moreover, because $\Pi$ is fully supported, as soon as $m \geq\max(x\cdot e_1,y\cdot e_1) + \max(x\cdot e_2, y\cdot e_2),$ $X_m \in x \vee y + \bbZ_{\geq 0}^2$ with positive $\Pi$-probability. Therefore, the expectation on the right-hand side is strictly positive. Taking $\beta^{-1}\log$ gives \eqref{eq:Busgeo}.
For all $n$ such that $\PF{y+e_1}{X_n}^{\beta,\w}>0$ and $\PF{y+e_2}{X_n}^{\beta,\w}>0$ (which, by Lemma \ref{lem:Omega0}, hold for all sufficiently large $n$), we have that
\[
\PF{y}{X_n}^{\beta,\w} = e^{\beta\w_{y}}(\PF{y+e_1}{X_n}^{\beta,\w} + \PF{y+e_2}{X_n}^{\beta,\w}),
\]
and hence
\[
1 = e^{\beta\w_y}\left(\frac{\PF{y+e_1}{X_n}^{\beta,\w}}{\PF{y}{X_n}^{\beta,\w}}+\frac{\PF{y+e_2}{X_n}^{\beta,\w}}{\PF{y}{X_n}^{\beta,\w}}\right).
\]
Letting $n\to\infty$ and using \eqref{eq:Busgeo} gives \eqref{eq:Busgeo-lemma-recovery}. \eqref{eq:Geobus-transition-equal} follows from \eqref{eq:Busgeo} and Proposition \ref{prop:DLR-transition}. Note also that by Proposition \ref{prop:DLR-transition}, the right-hand side of \eqref{eq:Busgeo} does not change with $n$ once $n>(x\vee y)\cdot(e_1+e_2)$.
\end{proof}

\begin{proof}[Proof of Proposition \ref{prop:global-ext}]
The existence of the limit in \eqref{eq:global-ext-Busemann} is in  Lemma \ref{lem:Busgeo}. The cocycle property is immediate from the definition and recovery is in \eqref{eq:Busgeo-lemma-recovery}. This verifies the claims in \eqref{prop:global-ext:cocrev}.

That the expression in \eqref{eq:cDLR-def} defines an element of $\DLR_z^{\beta,\w}$ whenever $\Busgeo_{\Pi}$ is a $\beta$-recovering cocycle is a consequence of a straightforward computation recorded as part of Theorem 5.2 in \cite{Jan-Ras-20-aop}. The positivity of the right-hand side ensures the measure is fully supported. To check consistency, let $y \leq z$ be given and abbreviate $m = y \cdot(e_1+e_2)$ and $\ell=z\cdot(e_1+e_2)$. Fix any admissible path $x_{\ell:n}$ with $x_\ell=z$. Then
\[
\cDLR_y^{\Pi}(x_{\ell:n} \viiva z) = \frac{\cDLR_y^{\Pi}(x_{\ell:n})}{\cDLR_y^{\Pi}(z)} = \frac{\PF{y}{z}^{\beta,\w}e^{\beta \sum_{k=\ell}^{n-1}\w_{x_k} - \beta \Busgeo_{\Pi}(y,x_n)}}{\PF{y}{z}^{\beta,\w}e^{-\beta \Busgeo_{\Pi}(y,z)}} = e^{\beta \sum_{k=\ell}^{n-1}\w_{x_k} - \beta \Busgeo_{\Pi}(z,x_n)}=\cDLR_z^{\Pi}(x_{\ell:n}).
\]

To check the extremality claim in \eqref{prop:global-ext:consistent-Gibbs}, we first verify \eqref{prop:global-ext:agree}. This is straightforward: together, \eqref{eq:cDLR-def} and \eqref{eq:Geobus-transition-equal} imply that for any $z\ge x$, the transition probabilities of $\cDLR_{z}^{\Pi}$ match those of $\Pi_z$.

Returning to extremality, Corollary \ref{cor:quadrant-extreme} gives $\cDLR_y^{\Pi}\in\ext \DLR_y^{\beta,\w}$ for all $y\ge x$. We will extend this to arbitrary roots using Lemma \ref{lem:forward-extreme}. First observe that no measure $\cDLR_z^\Pi$ assigns positive probability to a path of the form $u+e_j \bbZ_{\geq 0}$, $j\in\{1,2\}$, $u\in z+\Z^2_+$. Lemma \ref{lem:Omega0} rules out all such coordinate rays that are not rooted at $z$. To rule out the two rays rooted at $z$, fix $j\in\{1,2\}$. If $\cDLR_z^\Pi$ assigned positive probability to $z+e_j\bbZ_{\geq0}$, then consistency and full support would imply that $\cDLR_{z-e_{3-j}}^\Pi$ assigns positive probability to the same ray. This contradicts Lemma \ref{lem:Omega0}. It follows that, under $\cDLR_z^\Pi$, the path almost surely eventually enters $v+\bbZ_{\geq0}^2$ for every vertex $v\ge z$ as any path that fails to do so must eventually follow a coordinate ray. 

Now take any $y\in\bbZ^2$ and set $v=x\vee y$. The measure $\cDLR_v^\Pi$ is extremal by the case $v\ge x$, and the preceding paragraph gives
\[
\cDLR_y^\Pi(X_n\in v+\bbZ_{\geq0}^2\text{ for all sufficiently large }n)=1.
\]
Lemma \ref{lem:forward-extreme} implies that $\cDLR_y^\Pi\in\ext\DLR_y^{\beta,\w}$. This completes the verification of \eqref{prop:global-ext:consistent-Gibbs}.

Finally, we check \eqref{prop:global-ext:Busagree}. Denote by $C\in\ptail$ the event
\[
C = \biggl\{X' : \lim_{n\to\infty}\frac{1}{\beta}\bigl(\log \PF{y}{X_n'}^{\beta,\w}-\log\PF{z}{X_n'}^{\beta,\w}\bigr)=\lim_{n\to\infty}\frac{1}{\beta}\biggl(\log \bfE^\Pi\biggl[\frac{\PF{y}{X_n}^{\beta,\w}}{\PF{x}{X_n}^{\beta,\w}}\biggr]-\log \bfE^\Pi\biggl[\frac{\PF{z}{X_n}^{\beta,\w}}{\PF{x}{X_n}^{\beta,\w}}\biggr]\biggr)\biggr\}.
\]
Note that the limit of expectations on the right-hand side exists (by backward martingale convergence) for any $\w\in \Omega$ and $\beta\in(0,\infty)$ provided only that $\Pi \in \ext \DLR_x^{\beta,\w}$ is fully supported, which is our hypothesis. In particular, the existence and value of that limit are not part of the event. 

We have $\Pi(C) = 1$ by Lemma \ref{lem:Busgeo}. For any vertex $u\ge x$, full support gives $\Pi(u)>0$, and \eqref{prop:global-ext:agree} gives $\cDLR_u^\Pi(\aabullet)=\Pi(\aabullet\viiva u)$. Since $\Pi(C)=1$, $\cDLR_u^\Pi(C)=1$.
Now fix any $w\in\bbZ^2$ and set $v=w\vee x$. Since $v\ge x$, we have $\cDLR_v^\Pi(C)=1$. By consistency, $\cDLR_w^\Pi(C\viiva v)=1$, and full support gives $\cDLR_w^\Pi(v)>0$. Thus $\cDLR_w^\Pi(C)>0$. Since $C\in\ptail$ and $\cDLR_w^\Pi$ is extremal, $\cDLR_w^\Pi(C)=1$.
\end{proof}
We next turn to the proofs of the results concerning a global total ordering of extremal Gibbs measures, starting with Theorem \ref{thm:global-weakorder}.
\begin{proof}[Proof of Theorem \ref{thm:global-weakorder}]
Let $\Pi^1$ and $\Pi^2$ be as in the statement. We check the equivalence in the listed order. First, \eqref{thm:global-weakorder:Pi} implies \eqref{thm:global-weakorder:cDLR-exist} by Proposition \ref{prop:global-ext}\eqref{prop:global-ext:agree} with $z= x \vee y$. 

Now suppose that \eqref{thm:global-weakorder:cDLR-exist} holds and let $y\in\bbZ^2$ be such that $\cDLR_y^{\Pi^1} \preceq \cDLR_y^{\Pi^2}$. Let $\wt \Pi$ be a coupling of $\cDLR_y^{\Pi^1}$ and $\cDLR_y^{\Pi^2}$ under which $\wt{\Pi}(X^1 \preceq X^2) = 1.$ Since $X^1_n\preceq X^2_n$ almost surely under $\wt{\Pi}$, Lemma B.2 in \cite{Jan-Ras-20-aop} gives us that for all $z\in\bbZ^2$, with $z+e_1+e_2 \leq X_n^j$ for both $j\in\{1,2\}$,
\[
\frac{\PF{z+e_1}{X^1_n}^{\beta,\w}}{\PF{z}{X^1_n}^{\beta,\w}}
\le
\frac{\PF{z+e_1}{X^2_n}^{\beta,\w}}{\PF{z}{X^2_n}^{\beta,\w}} \qquad \text{ and }\qquad
\frac{\PF{z+e_2}{X^1_n}^{\beta,\w}}{\PF{z}{X^1_n}^{\beta,\w}}
\ge
\frac{\PF{z+e_2}{X^2_n}^{\beta,\w}}{\PF{z}{X^2_n}^{\beta,\w}}.
\]
Since $\w\in\Omega_0$ and the measures $\cDLR_y^{\Pi^1}$ and $\cDLR_y^{\Pi^2}$ are extremal and fully supported, Lemma \ref{lem:Omega0} implies that, $\wt\Pi$-almost surely, for each $z$, $z+e_1+e_2\le X_n^j$ for both $j\in\{1,2\}$ and all sufficiently large $n$. Taking $n\to\infty$ and appealing to Proposition \ref{prop:global-ext}\eqref{prop:global-ext:Busagree} implies \eqref{thm:global-weakorder:Busgeo}.

Suppose \eqref{thm:global-weakorder:Busgeo} holds and fix $z\in\bbZ^2$. By definition, for $u \geq z$ and $i\in\{1,2\}$, the transition probability of $\cDLR_z^{\Pi^i}$ to go from $u$ to $u+e_1$ is
\begin{align}\label{1666}
\pi^{\cDLR_z^{\Pi^i}}(u,u+e_1) = e^{\beta \w_u - \beta \Busgeo_{\Pi^i}(u,u+e_1)}.
\end{align}
It follows that $\cDLR_z^{\Pi^1} \preceq \cDLR_z^{\Pi^2}$, which is \eqref{thm:global-weakorder:cDLR-all}.

\eqref{thm:global-weakorder:cDLR-all} implies \eqref{thm:global-weakorder:Pi} by Proposition \ref{prop:global-ext}\eqref{prop:global-ext:agree}.
\end{proof}
Next we prove the characterization of the strict total order in Theorem \ref{thm:global-strictorder}.
\begin{proof}[Proof of Theorem \ref{thm:global-strictorder}]
As above, \eqref{thm:global-strictorder:Pi} immediately implies \eqref{thm:global-strictorder:cDLR-exist}, \eqref{thm:global-strictorder:Busgeo} immediately implies \eqref{thm:global-strictorder:cDLR-all}, and \eqref{thm:global-strictorder:cDLR-all} immediately implies \eqref{thm:global-strictorder:Pi}. The only non-trivial implication is \eqref{thm:global-strictorder:cDLR-exist} implies \eqref{thm:global-strictorder:Busgeo}, so that is what we devote the rest of the proof to.

Assume that \eqref{thm:global-strictorder:cDLR-exist} holds, i.e., there exists $z \in \bbZ^2$ for which $\cDLR_z^{\Pi^1}\precneq \cDLR_z^{\Pi^2}$.  Then, by Theorem \ref{thm:global-weakorder},
\begin{align}\label{1678}
\Busgeo_{\Pi^1}(u,u+e_1) \geq \Busgeo_{\Pi^2}(u,u+e_1)\quad\text{and}\quad
\Busgeo_{\Pi^1}(u,u+e_2) \leq \Busgeo_{\Pi^2}(u,u+e_2)\qquad\forall u\in\bbZ^2.
\end{align} 

Next, observe that, by the recovery property \eqref{eq:Busgeo-recovery}, for any $u\in\bbZ^2$, 
\[\Busgeo_{\Pi^1}(u,u+e_1)>\Busgeo_{\Pi^2}(u,u+e_1) \text{ if and only if }\Busgeo_{\Pi^1}(u,u+e_2)<\Busgeo_{\Pi^2}(u,u+e_2).\] 
We will call a point $u$ for which both these strict inequalities hold a point of increase. 

If $u$ is not a point of increase, then
\[\Busgeo_{\Pi^1}(u,u+e_1) - \Busgeo_{\Pi^1}(u,u+e_2) = \Busgeo_{\Pi^2}(u,u+e_1) - \Busgeo_{\Pi^2}(u,u+e_2)
\label{eq:plaquette-1}.\]
Combining this with the cocycle property, we have
\begin{align*}
&\Busgeo_{\Pi^1}(u+e_2,u+e_1+e_2) - \Busgeo_{\Pi^1}(u+e_1,u+e_1+e_2) \\
&\qquad= \Busgeo_{\Pi^1}(u,u+e_1)-\Busgeo_{\Pi^1}(u,u+e_2)= \Busgeo_{\Pi^2}(u,u+e_1) - \Busgeo_{\Pi^2}(u,u+e_2)\\ 
&\qquad= \Busgeo_{\Pi^2}(u+e_2,u+e_1+e_2) - \Busgeo_{\Pi^2}(u+e_1,u+e_1+e_2).
\end{align*}
This and \eqref{1678} force both $\Busgeo_{\Pi^1}(u+e_2,u+e_1+e_2)=\Busgeo_{\Pi^2}(u+e_2,u+e_1+e_2)$ and $\Busgeo_{\Pi^1}(u+e_1,u+e_1+e_2)=\Busgeo_{\Pi^2}(u+e_1,u+e_1+e_2)$. Therefore if $u$ is not a point of increase, neither $u+e_1$ nor $u+e_2$ is a point of increase. An induction argument then shows that if $u$ is not a point of increase, then no vertex in the quadrant $u + \bbZ_{\geq 0}^2$ is a point of increase. But then the transition probabilities of $\cDLR_u^{\Pi^1}$ and $\cDLR_u^{\Pi^2}$ (given in \eqref{1666}) all agree and, consequently, $\cDLR_u^{\Pi^1}=\cDLR_u^{\Pi^2}$.

Having $\cDLR_u^{\Pi^1}=\cDLR_u^{\Pi^2}$, apply the equivalence \eqref{thm:global-weakorder:cDLR-exist}$\Leftrightarrow$\eqref{thm:global-weakorder:cDLR-all} in Theorem \ref{thm:global-weakorder} twice, to get that $\cDLR_y^{\Pi^1}=\cDLR_y^{\Pi^2}$ for all $y\in \bbZ^2$. This contradicts the assumption $\cDLR_z^{\Pi^1}\precneq\cDLR_z^{\Pi^2}$.
This contradiction proves that every vertex in $\bbZ^2$ must be a point of increase, which is \eqref{thm:global-strictorder:Busgeo}.
\end{proof}

We now turn to the characterization of coalescence in terms of cocycle equality, recorded as Proposition \ref{prop:coal-cond}.
\begin{proof}[Proof of Proposition \ref{prop:coal-cond}]
First, assume \eqref{prop:coal-cond:cons}, i.e., that the family $\{\Pi^x : x \in V\}$ is consistent. It suffices to show that pairwise coalescing couplings exist, so fix vertices $x\neq y$ in $V$. Full support implies that $\Pi^x(x \vee y)>0$ and $\Pi^y(x \vee y)>0$, so that we can condition. Consistency implies that $\Pi^x_{x \vee y} = \Pi^y_{x \vee y}$. For any event $A\in\ptail$, we must have $\Pi^x(A) = \Pi^x_{x\vee y}(A) = \Pi^y_{x \vee y}(A) = \Pi^y(A)$. Existence of a coalescing coupling now follows from Goldstein's theorem \cite[Theorem 2.1]{Gol-79}.

Now, suppose that \eqref{prop:coal-cond:coal} holds. Fix $x \neq y$ in $V$. Because a coalescing coupling exists, $\Pi^x$ and $\Pi^y$ agree on $\ptail$. This is part of Goldstein's theorem and the proof appeared previously in the proof of Theorem \ref{thm:coalext}. The limit in \eqref{eq:global-ext-Busemann} is tail-measurable and hence \eqref{prop:coal-cond:Busgeo} follows.

Finally, suppose \eqref{prop:coal-cond:Busgeo} holds. Then consistency follows from Proposition \ref{prop:global-ext}. This is because $\Busgeo_{\Pi^x}=\Busgeo_{\Pi^y}$ implies $\cDLR_z^{\Pi^x} = \cDLR_z^{\Pi^y}$ for all $z\in\bbZ^2$ while for $z \geq x$, $\cDLR_z^{\Pi^x} = \Pi^x_z$ and for $z \geq y$, $\cDLR_z^{\Pi^y} = \Pi_z^y$. Therefore \eqref{prop:coal-cond:cons} holds.
\end{proof}

\section{Strong existence and strong uniqueness}\label{sec:strong}
We next prove strong existence of the Busemann process in the directed polymer model. The argument can be viewed as a positive temperature analogue of that in \cite{Jan-Ras-Sep-25-strong-}, itself inspired by the construction of \cite{Ahl-Hof-16-}. At a high level, the strategy parallels that in Section 4.2 of \cite{Jan-Ras-Sep-25-strong-}, but several key steps require substantially different arguments. In particular, the almost sure total order and closedness of $\ext\DLR_x^{\beta,\w}$ (Theorem \ref{thm:cltot}) are straightforward in the zero-temperature setting, while requiring a separate analysis in positive temperature (Section \ref{sec:structure}). Furthermore, we need to identify an appropriate dense subset of $\ext\DLR_x^{\beta,\w}$ on which to carry out the construction, leading to a rather different implementation of the general strategy. With these ingredients in place, the main idea of the argument is to use the total order on the almost surely compact space $\ext\DLR_x^{\beta,\w}$ to define a shift-covariant cumulative distribution function. Given a random element of this set on an extended probability space, we can then use the associated conditional quantile function given the weights $\w(\what)$ to shift covariantly sample from the conditional distribution using the inverse CDF method. We then show that the resulting random variable is degenerate by taking advantage of certain consistency properties of the measures coming from coalescence. This implies measurability with respect to the weights.

\subsection{Approximating set}
Our first goal in this section is to construct a shift covariantly defined, countable, and dense measurable subset of $\ext \DLR_x^{\beta, \w}$ which contains all left- or right-isolated elements of $\ext \DLR_x^{\beta,\w}$, if any exist.

\subsubsection{Left- and right-isolated measures}\label{sec:lriso}
We start by giving the definition of being right-isolated relative to a totally ordered subset of the space of measures, with the corresponding definition of left-isolated being similar. Since the stochastic order is closed under weak convergence, the weak closure of any totally ordered subset of $\sM_1(\pathsp_x,\pathsa_{k:\infty})$ is compact and totally ordered. The topology of weak convergence is finer than the order topology in this setting, so every nonempty totally ordered subset $K$ which is compact in the topology of weak convergence has a unique minimum and a unique maximum, denoted by $\inf K$ and $\sup K$, respectively. By convention, these extreme elements are regarded as left- and right-isolated relative to $K$, respectively.

\begin{definition}\label{def:LIRI}
If $K$ is a closed totally ordered subset of $(\sM_1(\pathsp_x,\pathsa_{k:\infty}),d_{\sM_1},\preceq)$, then $\mu \in$ $K\,\setminus$ $\,\{\sup K\}$ is right-isolated relative to $K$ if there exists $\nu \in K$ with $\mu \precneq \nu$ in the total order for which $\{\gamma \in K : \mu \precneq \gamma \precneq \nu \} = \varnothing$.
\end{definition}

Before moving on to our construction, we prove a characterization of being right-isolated which will be useful in the sequel.  The corresponding left-isolated statement is similar. 

\begin{lemma}\label{lem:RI}
Fix $x\in\bbZ^2$ with $x\cdot(e_1+e_2)=k$.  Let $K$ be a closed totally ordered subset of
$(\sM_1(\pathsp_x,\pathsa_{k:\infty}),d_{\sM_1},\preceq)$.  Then $\mu\in K\,\setminus\,\{\sup K\}$ is right-isolated relative to $K$ if and only if there exist
$\nu\in K$ with $\mu\precneq \nu$, a finite up-right path $x_{k:n}$, $x_k=x$, and $\e>0$ such that
\begin{equation}\label{eq:right-gap}
\{\gamma\in K:\ \mu\precneq\gamma\text{ and }\gamma(x_{k:n}\preceq X_{k:n})<\mu(x_{k:n}\preceq X_{k:n})+\e\}=\varnothing.
\end{equation}
\end{lemma}

\begin{proof}
Let $\mu\in K\ \setminus\ \{\sup K\}$ be right-isolated and let $\nu\in K$ satisfy $\mu\precneq\nu$ and
$\{\gamma\in K:\mu\precneq\gamma\precneq\nu\}=\varnothing$.
Since order intervals separate probability measures, $\mu\precneq\nu$ implies that there exist $n$ and $x_{k:n}\in\pathsp_{x:n}$ with
$\mu(x_{k:n}\preceq X_{k:n})<\nu(x_{k:n}\preceq X_{k:n})$.
Set $\e=\nu(x_{k:n}\preceq X_{k:n})-\mu(x_{k:n}\preceq X_{k:n})>0$.
If $\gamma\in K$ satisfies $\mu\precneq\gamma$, then necessarily $\nu\preceq\gamma$ (else $\mu\precneq\gamma\precneq\nu$),
hence $\gamma(x_{k:n}\preceq X_{k:n})\ge \nu(x_{k:n}\preceq X_{k:n})=\mu(x_{k:n}\preceq X_{k:n})+\e$.
This proves \eqref{eq:right-gap}.

Conversely, assume \eqref{eq:right-gap} holds for some $x_{k:n}$ and $\e>0$, and also fix $\nu_0\in K$ with $\mu\precneq\nu_0$. $\nu_0$ exists because $\mu\ne\sup K$. Then $\nu_0(x_{k:n}\preceq X_{k:n})\ge \mu(x_{k:n}\preceq X_{k:n})+\e$. Let
\[
S=\bigl\{\gamma\in K:\ \gamma(x_{k:n}\preceq X_{k:n})\ge \mu(x_{k:n}\preceq X_{k:n})+\e\bigr\}.
\]
Then $S$ is nonempty and closed in $K$, hence compact.
Since $S$ is totally ordered, it has a minimum; call it $\nu$. We have $\mu\precneq \nu$ because $\nu(x_{k:n}\preceq X_{k:n})\ge \mu(x_{k:n}\preceq X_{k:n})+\e$.
If there were $\gamma\in K$ with $\mu\precneq\gamma\precneq\nu$, then \eqref{eq:right-gap} forces
$\gamma(x_{k:n}\preceq X_{k:n})\ge \mu(x_{k:n}\preceq X_{k:n})+\e$, i.e.\ $\gamma\in S$, contradicting the minimality of $\nu$.
Thus $\mu$ is right-isolated. 
\end{proof}

\subsubsection{Dyadic tuple sup and inf construction}\label{sec:dyadic}
We now turn to the construction of the dense set of measures.

For each $m \in \bbN$ and $k\in\{1,\dots,2^m\}$ we denote the $k^{\mathrm{th}}$ dyadic sub-interval of $[0,1]$ at dyadic scale $m$ and the collection of such intervals by
\[
\dya_{k,m}=\Bigl[\frac{k-1}{2^m},\frac{k}{2^m}\Bigr],\qquad
\dy_m=\{\dya_{k,m}:k=1,\dots,2^m\},\qquad
\dy=\bigcup_{m\ge1}\dy_m .
\]
For $n\in\bbN$, denote the set of $2^n$-tuples of dyadic intervals via 
\[
\dyatup_n=\{(I^1,\dots,I^{2^n}):\ I^j\in\dy\ \text{for }j=1,\dots,2^n\}.
\]
Fix a deterministic enumeration (lexicographic, for example) of the $2^n$ paths of length $n$  rooted at the origin and call these paths $\{\pi^{0,n,1},\dots,\pi^{0,n,2^n}\}$. Recall the shift $\theta_y$ on path spaces, defined in \eqref{eq:pathshift}. For $x\in\mathbb{Z}^2$, define $\pi^{x,n,j}=\theta_x\pi^{0,n,j}$.

Fix $x \in \bbZ^2$ with $x\cdot(e_1+e_2) = k$. For each $I=(I^1,\dots,I^{2^n})\in\dyatup_n$ call
\[
A_{x,I} = \bigl\{\mu \in \sM_1(\pathsp_x,\pathsa_{k:\infty}) : \mu(\pi^{x,n,j})\in I^j\ \ \forall\, j=1,\dots, 2^n\bigr\}.
\]
Because the dyadic intervals are closed, each such $A_{x,I}$ is closed. For each $\omega\in\Omega_0$, $A_{x,I} \cap \ext \DLR_x^{\beta,\w}$ is compact and totally ordered, so each such non-empty set admits unique maximal and minimal elements. As is typical, we take the convention that the supremum and infimum of the empty set are $\delta_{x + e_2\bbZ_{\geq 0}}$ and $\delta_{x + e_1\bbZ_{\geq 0}}$, respectively. 

Our next result shows existence of a shift-covariant measurable selection of these maximal and minimal elements. Recall the notation $(\theta_x)_\#$ for the pushforward of $\theta_x$ onto the path measure spaces as defined in \eqref{eq:pathshift-pushforward}.
\begin{proposition}\label{prop:dmeas}
For each $x\in\mathbb{Z}^2$, $n\in\bbZ_{\geq 1}$, and $I\in\dyatup_n$, there exist jointly Borel random variables
\begin{align*}
(\w,\beta) \in  \Omega\times(0,\infty) \mapsto \mu_{x,I}^{\beta,\omega,-}, \mu_{x,I}^{\beta,\omega,+} \in \sM_1(\pathsp_x,\pathsa_{k:\infty}) 
\end{align*}
such that for all $\omega\in\Omega_0$, all $\beta\in(0,\infty)$ and all $x\in\bbZ^2$, the following hold.
\begin{enumerate}[label={\rm(\alph*)}, ref={\rm\alph*}]
\item\label{prop:dmeas:empty}
If $A_{x,I}\cap \ext\DLR_x^{\beta,\omega}=\varnothing$, then
\[
\mu^{\beta,\omega,+}_{x,I}=\delta_{x+e_2\mathbb{Z}_{\ge0}}
\qquad\text{and}\qquad
\mu^{\beta,\omega,-}_{x,I}=\delta_{x+e_1\mathbb{Z}_{\ge0}}.
\]
\item\label{prop:dmeas:inset}
If $A_{x,I}\cap \ext\DLR_x^{\beta,\omega}\neq\varnothing$, then
\[
\mu^{\beta,\omega,+}_{x,I}=\sup\bigl(A_{x,I}\cap \ext\DLR_x^{\beta,\omega}\bigr)
\qquad\text{and}\qquad
\mu^{\beta,\omega,-}_{x,I}=\inf\bigl(A_{x,I}\cap \ext\DLR_x^{\beta,\omega}\bigr).
\]
\item\label{prop:dmeas:mu-covariant}
For all $y\in\mathbb{Z}^2$, all $n\in\mathbb{N}$, and all $I\in\dyatup_n$,
\be\label{eq:dmeas-def-shift}
\mu^{\beta,\omega,\pm}_{x+y,I}
\;=\;
(\theta_y)_\#\,\mu^{\beta,T_y\omega,\pm}_{x,I}.
\ee
\item\label{prop:dmeas:dense}
With the notation
\[
\Dmeas_x^{\beta,\omega}
=\Bigl\{\mu^{\beta,\omega,+}_{x,I},\ \mu^{\beta,\omega,-}_{x,I}:\ I\in\dyatup_n,\ n\in\mathbb{N}\Bigr\},
\]
the set $\Dmeas_x^{\beta,\omega}$ is countable and dense in $\ext\DLR_x^{\beta,\omega}$.

\item\label{prop:dmeas:isolated}
If $\mu\in\ext\DLR_x^{\beta,\omega}$ is left-isolated or right-isolated relative to $\ext\DLR_x^{\beta,\omega}$,
then $\mu\in\Dmeas_x^{\beta,\omega}$. 

\item\label{prop:dmeas:Dmeas-covariant}
For all $y\in\mathbb{Z}^2$,
\[
\Dmeas_{x+y}^{\beta,\omega}\;=\;(\theta_y)_\#\,\Dmeas_x^{\beta,T_y\omega}.
\]
\end{enumerate}
\end{proposition}
\begin{remark}
With reference to part \eqref{prop:dmeas:isolated} of the proposition, we note that it is an open problem (even in the exactly solvable log-gamma polymer model) whether or not fully supported left- or right-isolated measures in $\ext \DLR_x^{\beta,\w}$ exist. In the context of the log-gamma polymer, where it is known that multiple extremal Gibbs measures \cite{Bat-Fan-Sep-25} with the same asymptotic direction exist, this question is related to the positive temperature version of the ``No 3 geodesics'' problem in first- and last-passage percolation; see \cite{Bus-25-, Cou-11, Jan-Ras-Sep-23} for discussion of the zero temperature case. 
\end{remark}

The proof of Proposition \ref{prop:dmeas} is fairly standard (if somewhat tedious) measure theory, so we defer the argument to Appendix \ref{app:meas}.
\subsection{Strong existence}
\subsubsection{Existence of the shift-covariant conditional quantile function}
Our next lemma constructs a shift-covariant version of the conditional quantile function (given the $\sigma$-algebra generated by the weights) of the random measure $\Pi_x^{\Bhat,\what}$, viewed as an element of the random compact, totally ordered metric space $\ext \DLR_x^{\beta,\w(\what)}$. We begin by proving Lemma \ref{lem:Bhat-ext}, which says that elements of $\cKhat^{\beta}$ generate extremal Gibbs measures.

\begin{proof}[Proof of Lemma \ref{lem:Bhat-ext}]
On the almost sure event where $\Bhat$ is a globally defined $\beta$-recovering cocycle, the family $\{\Pi_x^{\Bhat,\what} : x\in\bbZ^2\}$ is consistent by Theorem \ref{thm:DLR-cocycle}. Without loss of generality, we may also work on a shift-invariant event for which shift covariance also holds.  By Theorem \ref{thm:coalext}, to prove the result, it thus suffices to produce a shift-invariant full probability event on which a coalescing coupling of these measures exists. This is the content of \cite[Theorem A.2]{Jan-Ras-20-aop} (see Appendix A.2 of \cite{Jan-Ras-18-arxiv} for the details), which adapts the Licea-Newman argument \cite{Lic-New-96} to the setting of nearest-neighbor random walks in stationary random environments with weakly elliptic steps.
\end{proof}

Next, we choose a shift-covariant and proper version of the conditional distribution given $\kS$ which is compatible with the event on which the Gibbs measures generated by $\Bhat$ are extremal.

\begin{lemma}\label{lem:nu-exists}
Fix $\beta\in(0,\infty)$ and $\Bhat\in\cKhat^{\beta}$. On $(\Omhat, \kShat, \Phat)$, there exists a shift-invariant Borel event $\Omega_{\Bhat}\subset\Omega_0$ with $\bbP(\Omega_{\Bhat})=1$, a shift-invariant Borel event $\Omgood\subset \Omhatext$ with $\Phat(\Omgood)=1$, and a Borel-measurable version $\nu^{\w(\what)}$ of the conditional distribution  given the $\sigma$-algebra $\kS$ generated by the weight map $\w$,  which satisfy the following properties:
\begin{enumerate}  [label={\rm(\alph*)}, ref={\rm\alph*}, series= enum:xBhatproc]   \itemsep=3pt
\item\label{lem:nu-exists:Omega0} For all $\what\in\Omgood$ and all $z\in\Z^2$, $\w(\That_z\what)=T_z\w(\what)\in\Omega_{\Bhat}$.
\item\label{lem:nu-exists:proper} {\rm(}Properness and compatibility of $\nu${\rm)} For all $\omega\in\Omega_{\Bhat}$,
\[
\nu^\omega\bigl(\Omgood\cap\{\what : \w(\what) = \omega \}\bigr)=1.
\]
\item\label{lem:nu-exists:cov} {\rm(}Shift covariance of $\nu${\rm)} For all $\what\in \Omgood$, all $z\in\Z^2$, and all $A\in\kShat$,
\begin{equation}\label{eq:nucov}
\nu^{\w(\That_z\what)}(A)=\nu^{\w(\what)}(\That_z^{-1}A).
\end{equation}
\end{enumerate}
\end{lemma}

This is standard measure theory and we defer the proof of this lemma to Section \ref{sec:CDFmeas}.  One fact used extensively below is that for $\w\in\Omega_0$, by Theorem \ref{thm:cltot}, $\ext \DLR_x^{\beta,\omega}$ is closed and totally ordered.

Recall the set of measures $\Dmeas_x^{\beta,\omega}$ constructed in Proposition \ref{prop:dmeas}. For $\beta\in(0,\infty)$, let $\Bhat \in \cKhat^{\beta}$. For each $x \in \bbZ^2$ and $s \in (0,1)\cap \bbQ$, define \be
\cdfm_x^{\Bhat,s}(\what) = \begin{cases} \inf\{\mu \in \Dmeas_x^{\beta,\w(\what)} : \nu^{\w(\what)}(\Omgood,\, \Pi_x^{\Bhat} \preceq \mu) \geq s \} & \w(\what) \in \Omega_{\Bhat}, \\
\delta_{x + e_1\bbZ_{\geq 0}} & \w(\what) \notin \Omega_{\Bhat}.
\end{cases}
\label{eq:inf-rational-formula}
\ee
Our next proposition records basic measurability properties of these variables and the fact that we can extend them into a left-continuous process with right limits.

In the next result, $\Skor([0,1],\, \sM_1(\pathsp_x, \pathsa_{k:\infty}))$ and $\Skor([0,1],\, \ext \DLR_x^{\beta,\w(\what)})$ denote the Skorokhod spaces of left-continuous paths with right limits taking values in the non-random compact metric space $\sM_1(\pathsp_x,\pathsa_{k:\infty})$ and its random compact subset $\ext \DLR_x^{\beta,\w(\what)}$, respectively.

\begin{proposition}\label{prop:xBhatproc}
    For each $\beta\in(0,\infty)$ and each $\Bhat \in \cKhat^{\beta}$, the following hold:
\begin{enumerate}  [label={\rm(\alph*)}, ref={\rm\alph*}, series= enum:xBhatproc]   \itemsep=3pt
\item \label{prop:inf-rational-meas:meas} For each $s\in(0,1)\cap\bbQ$, $\cdfm_x^{\Bhat,s}$ is an $\kS$-measurable, $\sM_1(\pathsp_x,\pathsa_{k:\infty})$-valued random variable such that for all $\what \in \Omhat$, $\cdfm_x^{\Bhat,s}(\what)\in \ext \DLR_x^{\beta,\w(\what)}$. 
\item \label{prop:inf-rational-meas:fullinf} For all $s\in(0,1)\cap\bbQ$ and all $\what$ with $\w(\what)\in\Omega_{\Bhat}$,
\begin{align}
\cdfm_x^{\Bhat,s}(\what)=\inf\bigl\{\mu\in\ext\DLR_x^{\beta,\w(\what)}:\nu^{\w(\what)}(\Omgood, \Pi_x^{\Bhat}\preceq \mu)\geq s\bigr\}.\label{eq:fullinf}
\end{align}
\item \label{prop:inf-rational-meas:inf-formula} If we extend the definition to $s\in[0,1]$ by setting
\be
\cdfm_x^{\Bhat,s}(\what) =
\begin{cases}
\inf\bigl\{\cdfm_x^{\Bhat,r}(\what):r\in\bbQ\cap(0,1)\bigr\}, & s=0,\\
\sup\bigl\{\cdfm_x^{\Bhat,r}(\what) : r\in \bbQ\cap(0,1), r<s\bigr\}, & s\in(0,1],
\end{cases}\label{eq:inf-formula}
\ee
then the expression in  \eqref{eq:inf-formula} agrees with \eqref{eq:inf-rational-formula} for $s \in(0,1)\cap\bbQ$, for all $\what\in\Omhat.$
\end{enumerate}
The process $\cdfm_x^{\Bhat,\aabullet}$ defined through \eqref{eq:inf-formula} satisfies the following properties.
\begin{enumerate}  [label={\rm(\alph*)}, ref={\rm\alph*}, resume= enum:xBhatproc]   \itemsep=3pt
\item\label{prop:xBhatproc.meas} $\what\mapsto \cdfm_x^{\Bhat,\aabullet}(\what)$ is an $\kS$-measurable, $\Skor([0,1],\, \sM_1(\pathsp_x,\pathsa_{k:\infty}))$-valued random variable which almost surely takes values in $\Skor([0,1],\, \ext\DLR_{x}^{\beta,\w(\what)})$.
\item\label{prop:xBhatproc.inf} 
For all $s \in (0,1)$ and all $\what$ with $\w(\what)\in\Omega_{\Bhat}$, \eqref{eq:fullinf} holds. Consequently, $\cdfm_x^{\Bhat,r}(\what)\preceq\cdfm_x^{\Bhat,s}(\what)$ for $r\le s$ in $[0,1]$.
\item \label{prop:xBhatproc.key} 
For each $\what$ with $\w(\what)\in\Omega_{\Bhat}$ and each $\mu \in \ext\DLR_x^{\beta,\w(\what)}$,
\begin{align}
\bigl\{s\in (0,1) : \cdfm_x^{\Bhat,s}(\what) \preceq \mu\bigr\}=(0,1)\cap\bigl[0,\nu^{\w(\what)}(\Omgood, \Pi_x^{\Bhat} \preceq \mu)\bigr].
\label{eq:key-sym} 
\end{align}
\item \label{prop:xBhatproc:cov} 
Viewing $\cdfm_x^{\Bhat,s}$ as a function on the canonical space, for all $\omega\in\Omega_{\Bhat}$, all $x,y\in\Z^2$, and all $s\in[0,1]$,
\be\begin{aligned}\label{eq:cdfm-cov}
\cdfm_{x+y}^{\Bhat,s}(\omega)=(\theta_y)_\#\,\cdfm_x^{\Bhat,s}(T_y\omega).
	\end{aligned}\qedhere\ee
\end{enumerate}    
\end{proposition}
\begin{remark}
In view of $\kS$ measurability, we may write $\cdfm_x^{\Bhat,\aabullet}(\what)=\cdfm_x^{\Bhat,\aabullet}(\w(\what))$. This pushes $\cdfm_x^{\Bhat,\aabullet}$ onto the canonical space $(\Omega,\sF,\bbP)$. When working with the process $\cdfm_x^{\Bhat,\aabullet}$, we will typically work directly on the canonical space and only pull back to the extended space when working with objects that are defined there. In particular, we remind the reader that $\Bhat$ is a decoration to remind us of the origin of $\cdfm_x^{\Bhat,\aabullet}$; it does not carry any extra dependence on $\what$.
\end{remark}

We also defer the proof of this proposition to Section \ref{sec:CDFmeas}. Properness of the conditional distribution is used to ensure that if $\mu\in\DLR_x^{\beta,\w(\what)}$, then it lies in $\DLR_x^{\beta,\w(\what')}$ for $\nu^{\w(\what)}$-almost all $\what'$. The claim that \eqref{eq:inf-formula} agrees with \eqref{eq:inf-rational-formula} for rational $s>0$ is the step where we need to know that $\Dmeas_x^{\beta,\w(\what)}$ contains all fully supported right-isolated measures, if any exist.

With existence of the shift-covariant conditional quantile function in hand, we next use inverse CDF sampling and coupling arguments to prove strong existence.
\subsubsection{Inverse CDF sampling and strong existence}
Denote by $\leb$ the Lebesgue measure on $[0,1]$. As usual, we write integrals with respect to the Lebesgue measure using $ds$ instead of $\leb(ds)$. Our next result is the analogue in this setting of the classical fact that a random variable can be sampled by substituting a Uniform[0,1] random variable into its quantile function.
\begin{lemma}\label{lem:uniform-quantile}
    Fix $\beta>0$, $x \in \bbZ^2$, and $\Bhat \in \cKhat^{\beta}$. Then the distribution on $\Omega\times\sM_1(\pathsp_x, \pathsa_{k:\infty})$ of $(\w(\what),\Pi_x^{\Bhat,\what})$ under $\Phat(d\,\what)$ is the same as the distribution of $(\w,\cdfm_x^{\Bhat,s}(\w))$ under $\bbP(d\,\w)\otimes\leb(d\,s)$.
\end{lemma}
\begin{proof}
It suffices to prove that for all $A\in\sF$ and $C \in \mathscr{B}(\sM_1(\pathsp_x,\pathsa_{k:\infty}))$,
\be
\Phat(\w\in A,\, \Pi_x^{\Bhat}\in C) = \int_0^1 \bbP(A,\, \cdfm_x^{\Bhat,s}\in C)\,ds.\label{eq:quantile-goal}
\ee
Call a set $C\subset \sM_1(\pathsp_x,\pathsa_{k:\infty})$ \emph{lower} if
$\nu\preceq\mu\in C$ implies $\nu\in C$. We first show that to prove \eqref{eq:quantile-goal} for all $C\in \mathscr{B}(\sM_1(\pathsp_x,\pathsa_{k:\infty}))$, it is enough to
consider closed lower sets $C$. To see this, note that lower path intervals of the form $\{\gamma\in\pathsp_x:\gamma\preceq\pi\}$ are a $\pi$-system that generates the Borel $\sigma$-algebra on $\pathsp_x$
(see \cite[Lemma A.3]{Jan-Ras-Sep-25-strong-}). By the monotone class theorem \cite[Appendix Theorem 4.3]{Eth-Kur-86}, the Borel
$\sigma$-algebra of the weak topology on $\sM_1(\pathsp_x,\pathsa_{k:\infty})$ is then generated by the maps
\[\mu\mapsto F_\pi(\mu) = \mu\{\gamma\in\pathsp_x:\gamma\preceq\pi\},
\qquad \pi\in\pathsp_x.\]
The $\sigma$-algebra generated by a collection of real-valued maps is equal to the $\sigma$-algebra generated by super-level sets of those maps (because $\mathscr{B}(\bbR)$ is generated by sets of the form $[x,\infty)$). For $m\ge1$, $\pi_1,\dots,\pi_m\in\pathsp_x$, and $r_1,\dots,r_m\in [0,1]$, define
\[
C(\pi_1,\dots,\pi_m;r_1,\dots,r_m)
=
\bigcap_{j=1}^m\{\mu\in \sM_1(\pathsp_x,\pathsa_{k:\infty}): F_{\pi_j}(\mu)\ge r_j\}.
\]
The collection of such sets is a generating  $\pi$-system of $\mathscr{B}(\sM_1(\pathsp_x,\pathsa_{k:\infty}))$ consisting of closed lower sets. The reduction to closed lower sets then follows.

Equality in \eqref{eq:quantile-goal} holds trivially for all $A\in\sF$ if $C=\varnothing$. For any non-empty, closed, lower set $C$, by Proposition \ref{prop:meas-sup}, there is a Borel-measurable $\DLR_x^{\beta,\w}$-valued random variable $\mu_C^\w$ with the property that on $\Omega_0$, $\mu_C^\w = \sup\{ C \cap \ext \DLR_x^{\beta,\w}\} \in C \cap \ext \DLR_x^{\beta,\w}.$ This intersection is non-empty, because it necessarily contains $\delta_{x+e_2\bbZ_{\geq0}}.$  Recall the event $\Omgood\subset\Omhatext$ from Lemma \ref{lem:nu-exists}. By Lemma \ref{lem:Bhat-ext}, $\Pi_x^{\Bhat,\what}\in\ext\DLR_x^{\beta,\w(\what)}$ and is fully supported for every $\what\in\Omgood$. By total ordering of $\ext \DLR_x^{\beta,\w(\what)}$ and the hypothesis that $C$ is lower, a measure $\mu\in \ext \DLR_x^{\beta,\w}$ satisfies $\mu \in C$ if and only if $\mu \preceq \mu_C^\w$.

Consequently,
\be\begin{aligned}
\Phat(\w \in A, \Pi_x^{\Bhat}\in C) &= \Ehat[\one_A(\w)\,\nu^\w(\Pi_x^{\Bhat} \in C)] = \Ehat[\one_A(\w)\,\nu^\w(\Omgood,\,\Pi_x^{\Bhat} \preceq \mu_C^\w)] \\
&= \Ehat\left[\one_A(\w)\int_0^1 \one_{\{\cdfm_x^{\Bhat,s} \preceq \mu_C^\w\}}ds\right]= \int_0^1 \bbP(A, \cdfm_x^{\Bhat,s} \preceq \mu_C^\w) ds \\
&= \int_0^1 \bbP(A, \cdfm_x^{\Bhat,s} \in C) ds.
\end{aligned}\ee
The third equality uses \eqref{eq:key-sym}; the endpoints do not affect the integral. The result now follows.
\end{proof}
Next, we look into the consequences of existence of coalescing couplings. Recall Definition \ref{def:coalcoup-pair} of a coalescing coupling and that $\mu\coal \gamma$ means that there exists a coalescing coupling of two path measures $\mu$ and $\gamma$ rooted at possibly distinct vertices. By Proposition \ref{prop:coal-cond}, if the two measures are extremal fully supported DLR measures, this is equivalent to $\Busgeo_{\mu}=\Busgeo_{\gamma}$, where these cocycles are defined in \eqref{eq:global-ext-Busemann}.
\begin{lemma}\label{lem:quantile-equal}
Fix $x,y\in\bbZ^2$ and $\beta\in(0,\infty)$. Let $\Bhat\in\cKhat^{\beta}$ and define $\cdfm_x^{\Bhat,\aabullet}$ and $\cdfm_y^{\Bhat,\aabullet}$ as in Proposition \ref{prop:xBhatproc}. Then for all $\w\in\Omega_{\Bhat}$, the following holds: for all pairs of fully supported measures $(\mu,\gamma)\in \ext \DLR_x^{\beta,\w}\times \ext \DLR_y^{\beta,\w}$ with $\mu \coal \gamma$,
\[
\{s \in (0,1) : \cdfm_x^{\Bhat,s} \preceq \mu\} = \{s\in(0,1) : \cdfm_y^{\Bhat,s} \preceq \gamma\}.
\]
\end{lemma}
\begin{proof}
By Lemma \ref{lem:nu-exists}, $\nu^\w(\Omgood\cap\{\what:\w(\what)=\w\})=1$. On $\Omgood$, the family $\{\Pi_x^{\Bhat,\what} : x\in\bbZ^2\}$ admits a coalescing coupling and $\Pi_x^{\Bhat,\what}\in \ext \DLR_x^{\beta,\w(\what)}$ for each $x$. Moreover, by construction, each of these measures is fully supported.

We have
\be\begin{aligned}
\nu^\w(\Omgood, \Pi_x^{\Bhat} \preceq \mu) - \nu^\w(\Omgood, \Pi_y^{\Bhat} \preceq \gamma) 
\leq \nu^\w(\Omgood, \Pi_x^{\Bhat} \preceq \mu, \Pi_y^{\Bhat}\succneq \gamma)=0.
\end{aligned}\ee
The last probability is zero because $\nu^\w$ gives full mass to the fiber of $\w$ and, on this fiber, the event inside the probability is empty by Theorems \ref{thm:global-weakorder} and \ref{thm:global-strictorder}. A symmetric argument then implies that whenever $\mu \coal \gamma$, we have
\[
\nu^\w(\Omgood, \Pi_x^{\Bhat} \preceq \mu) = \nu^\w(\Omgood, \Pi_y^{\Bhat} \preceq \gamma).
\]
The claim now follows from Proposition \ref{prop:xBhatproc}\eqref{prop:xBhatproc.key}.
\end{proof}

\begin{lemma}\label{lm:coal}
Let $\w\in\Omega_{\Bhat},$ $\beta\in(0,\infty)$, $\Bhat \in \cKhat^\beta$, and $x,y\in \bbZ^2$. For all $s\in (0,1)$, $\cdfm_x^{\Bhat,s} \in \ext \DLR_x^{\beta,\w}$ and $\cdfm_y^{\Bhat,s} \in \ext \DLR_y^{\beta,\w}$, both are fully supported, and
$\cdfm_x^{\Bhat,s} \coal \cdfm_y^{\Bhat,s}$. 
\end{lemma}

\begin{proof}
By Lemma \ref{lem:nu-exists}, $\nu^\w(\Omgood\cap\{\what:\w(\what)=\w\})=1$. On this event, $\Pi_z^{\Bhat,\what}\in \ext \DLR_z^{\beta,\w}$ and is fully supported for every $z\in\bbZ^2$.  Next we argue as in the proof of Lemma \ref{lem:uniform-quantile}. Let $C\subset\sM_1(\pathsp_z,\pathsa_{\ell:\infty})$ be a non-empty closed lower set and let $\mu_C^\w=\sup\{C\cap\ext\DLR_z^{\beta,\w}\}\in C\cap\ext\DLR_z^{\beta,\w}$ be furnished by Proposition \ref{prop:meas-sup}. Then, by total ordering and Proposition \ref{prop:xBhatproc}\eqref{prop:xBhatproc.key},
\[
\nu^\w(\Omgood,\Pi_z^{\Bhat}\in C)=\nu^\w(\Omgood,\Pi_z^{\Bhat}\preceq \mu_C^\w)
=\int_0^1 \one_{\{\cdfm_z^{\Bhat,s}(\w)\preceq\mu_C^\w\}}\,ds
=\int_0^1 \one_{\{\cdfm_z^{\Bhat,s}(\w)\in C\}}\,ds.
\]
By the monotone class theorem, this identity now holds for all Borel $C$. Taking $C$ to be the set of fully supported extremal Gibbs measures, we have that for almost all $s\in(0,1)$ and $z\in\{x,y\}$, $\cdfm_z^{\Bhat,s}\in \ext \DLR_z^{\beta,\w}$ and is fully supported. By monotonicity in $s$ (Proposition \ref{prop:xBhatproc}\eqref{prop:xBhatproc.inf}) and closedness of $\ext \DLR_z^{\beta,\w}$ when $\w\in\Omega_0$, we see that for all $s\in(0,1)$ and $z\in\{x,y\}$, $\cdfm_z^{\Bhat,s}\in \ext \DLR_z^{\beta,\w}$ and is fully supported.

Fix $s\in(0,1)$ and abbreviate $\mu_x=\cdfm_x^{\Bhat,s}$ and $\mu_y=\cdfm_y^{\Bhat,s}$. By Propositions \ref{prop:global-ext}\eqref{prop:global-ext:consistent-Gibbs} and \ref{prop:coal-cond}, the measures $\gamma_y=\cDLR_y^{\mu_x}$ and $\gamma_x=\cDLR_x^{\mu_y}$ are fully supported elements of $\ext \DLR_y^{\beta,\w}$ and $\ext \DLR_x^{\beta,\w}$, respectively, and satisfy $\mu_x \coal \gamma_y$ and $\mu_y \coal \gamma_x$. Applying Lemma \ref{lem:quantile-equal} to the pairs $(\mu_x,\gamma_y)$ and $(\gamma_x,\mu_y)$ gives
\[
\{r\in(0,1):\cdfm_x^{\Bhat,r}\preceq \mu_x\}
=
\{r\in(0,1):\cdfm_y^{\Bhat,r}\preceq \gamma_y\}
\]
and
\[
\{r\in(0,1):\cdfm_x^{\Bhat,r}\preceq \gamma_x\}
=
\{r\in(0,1):\cdfm_y^{\Bhat,r}\preceq \mu_y\}.
\]
Since $s$ belongs to the left-hand set in the first identity and to the right-hand set in the second, we obtain
\[
\mu_y\preceq \gamma_y
\qquad\text{and}\qquad
\mu_x\preceq \gamma_x.
\]
By Theorem \ref{thm:global-weakorder}, this implies
\[
\Busgeo_{\mu_y}\preceq \Busgeo_{\gamma_y}=\Busgeo_{\mu_x}
\qquad\text{and}\qquad
\Busgeo_{\mu_x}\preceq \Busgeo_{\gamma_x}=\Busgeo_{\mu_y},
\]
where the equalities use Proposition \ref{prop:global-ext}\eqref{prop:global-ext:Busagree}. Hence $\Busgeo_{\mu_x}=\Busgeo_{\mu_y}$, and therefore $\mu_x\coal \mu_y$ by Proposition \ref{prop:coal-cond}.
\end{proof}

Fix $\beta\in(0,\infty)$ and some $\Bhat\in\cKhat^\beta$. With some abuse of notation, set for $s\in(0,1),$ 
\be
\Busgeo_{s}^{\Bhat}(\w) = \begin{cases}
0 & \w\notin\Omega_{\Bhat} \\
\Busgeo_{\cdfm_0^{\Bhat,s}(\w)}(\w) & \w\in\Omega_{\Bhat}
\end{cases}\label{eq:Busgeo_s}
\ee
where $\Busgeo_{\cdfm_0^{\Bhat,s}(\w)}(\w)$ is as defined in \eqref{eq:global-ext-Busemann}.

\begin{lemma}\label{lem:Busgeo-cocycle}
The map $(\w,s)\in \Omega\times(0,1)\mapsto \Busgeo_{s}^{\Bhat}(\w) \in \bbR^{\bbZ^2\times\bbZ^2}$ is jointly measurable. The field $\Busgeo_{s}^{\Bhat}(\w)$ satisfies the following properties:
\begin{enumerate}[label={\rm(\roman*)}, ref={\rm\roman*}] \itemsep=2pt
\item {\rm(}Cocycle{\rm)} For all $s\in(0,1)$, $\w\in\Omega$, and $x,y,z\in\bbZ^2$, 
\[
\Busgeo_s^{\Bhat}(\w,x,y) + \Busgeo_s^{\Bhat}(\w, y,z) = \Busgeo_s^{\Bhat}(\w, x,z).
\]
\item {\rm(}Recovery{\rm)} For all $s\in(0,1)$, $\w\in\Omega_{\Bhat}$, and all $x\in\bbZ^2$,
\[
\sum_{i=1}^2 e^{\beta \w_x - \beta \Busgeo_s^{\Bhat}(\w, x,x+e_i)} = 1.
\]
\item {\rm(}Shift covariance{\rm)} For all $s\in(0,1)$, $\w\in\Omega$, and all $x,y,z\in\bbZ^2$,
\[
\Busgeo_s^{\Bhat}(\w, x+z,y+z) = \Busgeo_s^{\Bhat}(T_z\w, x,y).
\]
\end{enumerate}
\end{lemma}
\begin{proof}
The claimed measurability follows from a standard argument. The inputs are the measurability of $\cdfm_0^{\Bhat,\aabullet}$ coming from Proposition \ref{prop:xBhatproc}, the fact that for $\w\in\Omega_{\Bhat}$, $\cdfm_0^{\Bhat,s}$ is extremal and fully supported for all $s\in(0,1)$ (by Lemma \ref{lm:coal}), the expression and finiteness of the limit in \eqref{eq:global-ext-Busemann}, the fact that the map
\[
(\w,\mu)\in\Omega\times \sM_1(\pathsp_x,\pathsa_{k:\infty}) \mapsto \bfE^{\mu}\left[\frac{\PF{y}{X_n}^{\beta,\w}}{\PF{x}{X_n}^{\beta,\w}}\right] \in \bbR
\]
is measurable for each $n \geq k= x\cdot(e_1+e_2)$ and $y\in \bbZ^2$, and the fact that the limsup of measurable functions is measurable. The cocycle and recovery properties follow from Proposition \ref{prop:global-ext}, so all that remains to check is shift covariance.

Recall that $\Omega_{\Bhat}$ is shift-invariant. By definition, $\Busgeo_s^\Bhat$ is constant and equal to $0$ and hence shift covariantly defined off of $\Omega_{\Bhat}$, so we check the condition on $\Omega_{\Bhat}$. We have for $\w\in\Omega_{\Bhat}$ and $x,y\in\bbZ^2$, 
\[
\begin{aligned}
\Busgeo_s^\Bhat(T_z\w, x,y)
&=\Busgeo_{\cdfm_0^{\Bhat,s}(T_z\w)}(T_z\w,x,y) =\Busgeo_{(\theta_z)_\#\cdfm_0^{\Bhat,s}(T_z\w)}(\w,x+z,y+z)\\
&=\Busgeo_{\cdfm_z^{\Bhat,s}(\w)}(\w,x+z,y+z) =\Busgeo_{\cdfm_0^{\Bhat,s}(\w)}(\w,x+z,y+z)\\
&=\Busgeo_s^\Bhat(\w,x+z,y+z).
\end{aligned}
\]
The second equality is the definition \eqref{eq:global-ext-Busemann} together with
\[
X_{n+z\cdot(e_1+e_2)}\circ\theta_z = X_n+z
\qquad\text{and}\qquad
\PF{u}{v}^{\beta,T_z\w}=\PF{u+z}{v+z}^{\beta,\w},
\]
the third equality is \eqref{eq:cdfm-cov}, and the fourth uses Lemma \ref{lm:coal}
together with Proposition \ref{prop:coal-cond}. This proves
\[
\Busgeo_s^\Bhat(\w,x+z,y+z)=\Busgeo_s^\Bhat(T_z\w, x,y). \qedhere
\]
\end{proof}
An immediate consequence of Proposition \ref{prop:global-ext} and Lemma \ref{lem:Bhat-ext}  is that for every $\what\in\Omhatext$ with $\w(\what)\in\Omega_0$ we have that, for all $x\in\Z^2$, $\Pi_x^{\Bhat,\what}\in \ext \DLR_x^{\beta,\w(\what)}$ is fully supported, and  
\be
\Busgeo_{\Pi_x^{\Bhat,\what}}(\w(\what),y,z)=\Bhat(\what,y,z)\quad\text{for all }y,z\in\bbZ^2.\label{eq:Bus=Busgeo}
\ee
To see the equality, note that, by Lemma \ref{lem:Bhat-ext}, the family $\{\Pi^{\Bhat,\what}_x:x\in\Z^2\}$ is consistent. Hence, Proposition \ref{prop:coal-cond} implies that \[\Busgeo_{\Pi^{\Bhat,\what}_x}=\Busgeo_{\Pi^{\Bhat,\what}_y}\quad\text{for all $x,y\in\Z^2$}.\] 
On the other hand, for any given $x\in\Z^2$, parts \eqref{prop:global-ext:consistent-Gibbs} and \eqref{prop:global-ext:agree} of Proposition \ref{prop:global-ext}, the definition of $\Pi^{\Bhat,\what}_x$, and the cocycle property of $\Busgeo_{\Pi_x^{\Bhat,\what}}$ and $\Bhat$ imply that the two agree on $\Z^2_{\ge x}\times\Z^2_{\ge x}$. The equality \eqref{eq:Bus=Busgeo} follows.

We next push Lemma \ref{lem:uniform-quantile} forward to the distribution of cocycles, using the formulas in \eqref{eq:Busgeo_s} and \eqref{eq:Bus=Busgeo}. Recall the notation $\leb$ for the Lebesgue measure on $[0,1].$

\begin{corollary}\label{cor:A_s=dB}
    Fix $\beta\in(0,\infty)$ and  $\Bhat\in\cKhat^\beta$. The distribution of $(\w(\what),\Bhat(\what))$ under $\Phat(d\what)$ is the same as the distribution of $(\w,\Busgeo^\Bhat_s(\w))$ under $\bbP(d\w)\otimes\leb(ds)$.
\end{corollary}

Our next result shows that the inverse CDF sampling results in a constant random variable.

\begin{lemma}\label{lem:quantile-constant}
    Fix $\beta\in(0,\infty)$ and let $\Bhat\in\cKhat^\beta$ satisfy $\Phat\{\what: \hhB(\Bhat, \what)=\Ehat[\hhB(\Bhat)]\}=1$. Then $\P\{\forall s,t\in(0,1):\Busgeo^\Bhat_s= \Busgeo^\Bhat_t\}=1$.
\end{lemma}

\begin{proof}
Abbreviate $h=\Ehat[\hhB(\Bhat)]$. By the cocycle shape theorem \cite[Theorem 4.4]{Jan-Ras-20-aop}, recorded above as \eqref{B-shape}, $\Phat$ almost surely,
\[
\lim_{n\to\infty}\max_{|x|_1\leq n}\frac{|\Bhat(\what, 0,x) + h\cdot x|}{n}  = 0.
\]
It then follows from Corollary \ref{cor:A_s=dB} that for $\bbP\otimes\leb$ almost every $(\w,s)$,
\[
\lim_{n\to\infty}\max_{|x|_1\leq n}\frac{|\Busgeo^\Bhat_s(\w,0,x) + h\cdot x|}{n}  = 0.
\] 
In particular, for almost all $(\w,s)$ and $j\in\{1,2\}$,  
\[
\lim_{n\to\infty}\frac{\sum_{k=1}^n\Busgeo^\Bhat_s(\w, (k-1)e_j,ke_j)}{n} = -\,h \cdot e_j
\]
By the recovery property in Lemma \ref{lem:Busgeo-cocycle}, we have that for all $s\in(0,1)$ and $\w\in\Omega_{\Bhat}$, $\Busgeo^\Bhat_s(\w, x,x+e_j) \geq \w_x$ for all $x\in\bbZ^2$ and $j\in\{1,2\}$.

Next, recall that $\Busgeo^\Bhat_s$ is shift-covariant. Applying Birkhoff's ergodic theorem with $s\in (0,1)$ fixed, the above convergence combined with the integrability of the negative part of $\Busgeo^\Bhat_s(0, e_j)$ coming from the lower bound by $\w_0$ implies that $\Busgeo^\Bhat_s(0,e_j)\in L^1(\bbP)$ for $j\in\{1,2\}$ (else the limit would be infinite almost surely). Now, we apply the cocycle shape theorem \cite[Theorem 4.4]{Jan-Ras-20-aop} (\eqref{B-shape} above) to $\Busgeo^\Bhat_s$ to conclude that for almost all $s\in(0,1)$, $\hhB(\Busgeo^\Bhat_s)=h$, $\bbP$-almost surely. In particular, the mean does not depend on $s$, for a full Lebesgue measure set of $s$.

For $\w\in\Omega_{\Bhat}$, by Theorem \ref{thm:global-weakorder} and Proposition \ref{prop:xBhatproc}\eqref{prop:xBhatproc.inf},  the real-valued maps
\[
s \mapsto \Busgeo^\Bhat_s(x,x+e_1) \quad \text{ and } \quad s\mapsto \Busgeo^\Bhat_s(x,x+e_2)
\]
are both monotone. We have just shown that these processes have constant means on a dense set. This implies the processes are constant.
\end{proof}
The previous result shows that the random variable obtained by using the inverse CDF method to sample from the conditional distribution given the weights is trivial. By a standard measure theoretic argument, it follows that the process we started with is measurable with respect to the weights (up to a set of measure zero). The proof is verbatim identical to that of \cite[Theorem 3.1]{Jan-Ras-Sep-25-strong-}.

\begin{proof}[Proof of Theorem \ref{thm:Bhat-B}]
Define $\Bus(x,y)$ = $\int_0^1\Busgeo^\Bhat_s(x,y)ds$ and call $\Bus(\w) = (\Bus(\w, x,y) : x,y\in\bbZ^2)$. By Lemma \ref{lem:quantile-constant},  for $\bbP$-almost all $\w$, for all $s\in(0,1)$, $\Busgeo^\Bhat_s(\w,x,y) = \Bus(\w, x,y)$. By Corollary \ref{cor:A_s=dB}, the joint distribution of $(\w(\what), \Bhat(\what))$ is the same as the distribution of $(\w,\Bus(\w))$. It follows that $\Phat$-almost surely, $\Bhat(\what) = \Bus(\w(\what))$ by a standard measure theoretic argument (e.g.~Lemma 2.2 in \cite{Kur-07}): factorize the joint distribution of $(\w(\what), \Bhat(\what))$ under $\Phat$ as the distribution $\bbP(d\w)$ of $\w$ together with a transition kernel $\eta(db\viiva\w)$ that represents the conditional distribution of $\Bhat$ given $\w$. Do the same on the other side of the equality in distribution to see that $\eta(db\viiva\w) = \delta_{\Bus(\w)}(db)$ $\bbP$-almost surely. Thus $\Phat\{\what : \Bhat(\what) = \Bus(\w(\what))\} = \Ehat[\Phat(\Bhat=\Bus\viiva\kS)]=1.$
\end{proof}

\subsection{Strong uniqueness}
We begin with the following observation, which is essentially immediate from the definition. Recall the partial order on cocycles recorded as Definition \ref{def:cocycle-order}.
\begin{lemma}\label{lem:Bus-monotone-as}
Fix $\beta\in(0,\infty)$ and let $\Bus^1,\Bus^2\in\cK^{\beta}$. Then exactly one of the following holds:
\[
\bbP(\Bus^1 \prec \Bus^2) = 1 \text{ or }\bbP(\Bus^1 = \Bus^2)=1 \text{ or }\bbP(\Bus^2 \prec \Bus^1)=1.
\]
\end{lemma}
\begin{proof}
By Lemma \ref{lem:Bhat-ext}, each cocycle generates extremal Gibbs measures almost surely. By Corollary \ref{cor:as-order-trichotomy} and \eqref{eq:Bus=Busgeo}, the three possibilities above are a trichotomy for each realization of the environment for which the cocycles generate extremal Gibbs measures. By shift covariance, these events are all shift-invariant. By ergodicity of $\bbP$, exactly one has probability one.
\end{proof}

\begin{proof}[Proof of Theorem \ref{thm:uniqueness}]
By the trichotomy in Lemma \ref{lem:Bus-monotone-as}, exactly one of
\[
\bbP(\Bus^1\prec\Bus^2)=1,\qquad \bbP(\Bus^1=\Bus^2)=1,\qquad\text{or}\qquad \bbP(\Bus^2\prec\Bus^1)=1
\]
holds. In the first case, ergodicity of $\bbP$, together with the definition
\eqref{h-def}, implies $\hhB(\Bus^1)\preceq\hhB(\Bus^2)$. Equality of the tilts
would force $\Bus^1=\Bus^2$ almost surely: indeed, then \eqref{h-def} gives $
\E[\Bus^1(0,e_1)-\Bus^2(0,e_1)]=\E[\Bus^2(0,e_2)-\Bus^1(0,e_2)]=0.$ Since both differences are nonnegative almost surely, they would have to vanish almost surely. Hence in the first case
$\hhB(\Bus^1)\precneq\hhB(\Bus^2)$. The third case is analogous. In the second
case, $\Bus^1=\Bus^2$ almost surely, so trivially $\hhB(\Bus^1)=\hhB(\Bus^2)$.
\end{proof}

We next prove the process-level statement, Theorem \ref{thm:Bus-process}.
\begin{proof}[Proof of Theorem \ref{thm:Bus-process}]
First, recall that the monotone limits \eqref{Busproc-def} defining the process exist by Corollary \ref{cor:Bus-process-order}. For each $h\in\sH^\beta$ and $\sigg\in\{-,+\}$ the limiting field is a shift-covariant $\beta$-recovering cocycle because those properties are preserved by taking limits. Lemma \ref{lem:Bus-process-extend} gives \eqref{thm:Bus-process:same}. In particular, for each $h\in\sH^\beta$, the common value agrees almost surely with the previously defined cocycle $\Bus^{\beta,h}$ (introduced after the proof of Proposition \ref{prop:sH-extreme}). This gives \eqref{thm:Bus-process:Bus} and \eqref{thm:Bus-process:mean}.

The monotonicity statement \eqref{thm:Bus-process:monotone}, with all weak inequalities, follows from Corollary \ref{cor:Bus-process-order} by passing to the defining monotone limits. We prove the claimed strict inequalities.
If $h\precneq h'$ are both in $\cHdense^\beta$, then the claim follows because, by Lemma \ref{lem:Bus-process-extend}, there is no sign distinction and, by Theorem \ref{thm:uniqueness}, $\Bus^h\prec \Bus^{h'}$, almost surely. In particular, the strict inequalities hold when $h$ is right-isolated and $h'$ is left-isolated in $\sH^\beta$. In the remaining case, there are infinitely many tilts in $\cHdense^\beta$ strictly between $h$ and $h'$. Let $h''\precneq h'''$ be two such tilts. Then $\Bus^h\preceq \Bus^{h''}\prec \Bus^{h'''}\preceq \Bus^{h'}$ and \eqref{thm:Bus-process:monotone} is proved.

The continuity statement \eqref{thm:Bus-process:cont} follows from the definitions and the density of $\cHdense^\beta$. It remains to prove uniqueness. 

Let $\widetilde \Bus^{\beta,h\sig}$ be another process as in the statement of the theorem. Let $\bar\sH_0^\beta$ be a countable dense subset of the dense set of $h\in\sH^\beta$ for which $\widetilde B^{h-},\widetilde B^{h+}\in\cK^\beta$ and $\E[\hhB(B^{h-})]=\E[\hhB(B^{h+})]=h$. Then, by \eqref{Bbar=B} and Theorem \ref{thm:uniqueness}, we have, on a full-probability event, for every $h\in\bar\sH_0^\beta$, $\Bus^{\beta,h-}=\Bus^{\beta,h+}=\widetilde\Bus^{\beta,h-}=\widetilde\Bus^{\beta,h+}=B^{\beta,h}$. Applying \eqref{thm:Bus-process:cont} to both processes extends this agreement to all $h\in\sH^\beta$ and both signs.
\end{proof}

We next show that the process constructs extremal Gibbs measures. The main non-trivial claim here is that the resulting measures are extremal for all $h$ and all signs $\sigg\in\{+,-\}$.

\begin{proof}[Proof of Corollary \ref{cor:Bus-process-DLR}]
Let $\Omega'$ be the intersection of $\Omega_0$ with a shift-invariant full-probability event on which the signed process from Theorem \ref{thm:Bus-process} has the stated cocycle, recovery, monotonicity, and continuity properties for all $h\in\sH^\beta$, $\sigg\in\{+,-\}$, and $x,y\in\bbZ^2$, and on which Lemma \ref{lem:Bhat-ext} applies to the countable collection $\{\Bus^{\beta,h}:h\in\cHdense^\beta\}$. Fix $\w\in\Omega'$.

First take $h\in\cHdense^\beta$. Lemma \ref{lem:Bhat-ext}, applied on the canonical space to $\Bus^{\beta,h}$, implies that the measures generated by $\Bus^{\beta,h}$ are fully supported elements of $\ext\DLR_x^{\beta,\w}$ for every $x$. For general $h\in\sH^\beta$ and $\sigg\in\{+,-\}$, choose a monotone sequence $h_n\in\cHdense^\beta$ from the defining approximation \eqref{Busproc-def} of $\Bus^{\beta,h\sig}$. For every finite admissible path $x_{k:m}$, Theorem \ref{thm:Bus-process}\eqref{thm:Bus-process:cont} gives $\Pi_x^{\beta,h_n,\w}(x_{k:m})\to \Pi_x^{\beta,h\sig,\w}(x_{k:m}).$ This implies weak convergence. Since $\ext\DLR_x^{\beta,\w}$ is closed when $\w\in\Omega_0$, the limit belongs to $\ext\DLR_x^{\beta,\w}$. Full support follows directly from \eqref{eq:Bus-process-DLR-definition}. This proves \eqref{cor:Bus-process-DLR:extreme}.

The consistency in \eqref{cor:Bus-process-DLR:consistent} follows from Theorem \ref{thm:DLR-cocycle}. With this consistency and the extremality from part \eqref{cor:Bus-process-DLR:extreme}, \eqref{APi=B} follows from Propositions \ref{prop:global-ext} and \ref{prop:coal-cond}, as in the proof of \eqref{eq:Bus=Busgeo}.

Finally, if $h\precneq h'$ in $\sH^\beta$, \eqref{APi=B} and Theorem \ref{thm:Bus-process}\eqref{thm:Bus-process:monotone} give the ordering
\[\Busgeo_{\Pi_x^{\beta,h-,\w}}\preceq\Busgeo_{\Pi_x^{\beta,h+,\w}}\prec\Busgeo_{\Pi_x^{\beta,h'-,\w}}\preceq\Busgeo_{\Pi_x^{\beta,h'+,\w}}\quad\forall x\in\Z^2,\] 
and then part \eqref{cor:Bus-process-DLR:monotone}  follows from Theorems  \ref{thm:global-weakorder} and \ref{thm:global-strictorder}.
\end{proof}

We close with the ergodic decomposition consequence, recorded above as Theorem \ref{thm:decomp}.

\begin{proof}[Proof of Theorem \ref{thm:decomp}]
The proof is identical to that of \cite[Theorem 3.4]{Jan-Ras-Sep-25-strong-}. First, consider the case in which $\Phat$ is ergodic under $\That$. Then $\hhB(\Bhat)$ is deterministic $\Phat$-almost surely. Denote this value by $h$. By Theorems \ref{thm:Bhat-B} and \ref{thm:uniqueness}, $h\in\sH^\beta$ and $\Bhat(\what)=\Bus^{\beta,h}(\w(\what))$, $\Phat$-almost surely. By Theorem \ref{thm:Bus-process}\eqref{thm:Bus-process:same}, $\Bus^{\beta,h-}(\w)=\Bus^{\beta,h+}(\w)$, $\bbP$-almost surely, and by Theorem \ref{thm:Bus-process}\eqref{thm:Bus-process:Bus}, this common cocycle is $\Bus^{\beta,h}(\w)$. Hence 
\[
\Phat\Bigl\{\what:\hhB(\Bhat,\what)\in\sH^\beta,\ \Bus^{\beta,\hhB(\Bhat,\what)-}(\w(\what))=\Bus^{\beta,\hhB(\Bhat,\what)+}(\w(\what))=\Bhat(\what)\Bigr\}=1.
\]
Since the above event has full probability under every ergodic probability measure, the ergodic decomposition theorem \cite[Appendix B]{Ras-Sep-15-ldp} tells us that it also has full probability under any $\That$-invariant probability measure.
\end{proof}

\section{\texorpdfstring{$L^1$}{L1} continuity of the Busemann process}\label{sec:L1-cont}
We recall the following result concerning weak-strong convergence from \cite{Jac-Mem-81-Stoch}.

  \begin{theorem}{\rm\cite[Theorem 2.1]{Jac-Mem-81-Stoch}.}\label{a:cor:jm8}
   Let $(\Omhat,\kShat)$ and $(S,\mathscr{B}(S))$ be two Polish spaces with their Borel $\sigma$-algebras. Let  $\{\mu^n\}_{n\ge1}$ be a sequence of probability measures on  $(\Omhat\times S, \kShat\otimes\mathscr{B}(S))$ that converges under the standard  weak topology: 
  $\mu^n\to\mu$ in $\cM_1(\Omhat\times S)$.   Assume further that the $\Omhat$-marginals are constant:   
  $\mu^n_{\Omhat}=\mu_{\Omhat}$ for all $n\ge1$.   Then 
  \[    \lim_{n\to\infty} \int_{\Omhat\times S} g\,d\mu^n =   \int_{\Omhat\times S} g\,d\mu \]  
for all bounded measurable functions $g:\Omhat\times S\to\R$ such that $x\mapsto g(\what,x)$ is a continuous function on $S$ for each $\what\in\Omhat$.   
\end{theorem}

One straightforward consequence of weak-strong convergence we will use is the following: if a sequence of random variables $X_n$ is coupled together on a single probability space $\Omhat$ along with another random variable $X$ and if the law of $(\what,X_n(\what))$ converges to the law of $(\what,X(\what))$ in the weak-strong topology which appears in Theorem \ref{a:cor:jm8}, then $X_n$ converges to $X$ in probability. To see this, take $g(\what,x) = d(x,X(\what))$, where $d$ is a bounded metric generating the topology of $S$.

We now turn to the proof of Theorem \ref{thm:L1-continuity-coupled}. The basic outline is the following: our hypotheses imply tightness of the Busemann cocycles and convergence of their means along subsequential weak limits. But because we have strong existence and uniqueness and can work on a common probability space, subsequential weak limits become in probability limits. This holds provided that we know that any weak limit has the property that its conditional mean vector is equal to its unconditional mean vector. Verifying this condition is why we have the extremality hypothesis. Once that is handled, the convergence can be upgraded to $L^1$ convergence using Scheff\'e's lemma.

\begin{proof}[Proof of Theorem \ref{thm:L1-continuity-coupled}]
Fix $x,y\in\bbZ^2$ and a nearest-neighbor lattice path $x=x_0,x_1,\dotsc,x_m=y$. By the cocycle property, $\Phat$-almost surely,
\[
\Bus^{\beta_n,h_n,\nu_n}(\w^n,x,y)-\Bus^{\beta_\infty,h_\infty,\nu_\infty}(\w^\infty,x,y)
=
\sum_{k=0}^{m-1}\Bigl(\Bus^{\beta_n,h_n,\nu_n}(\w^n,x_k,x_{k+1})-\Bus^{\beta_\infty,h_\infty,\nu_\infty}(\w^\infty,x_k,x_{k+1})\Bigr).
\]
If $x_{k+1}=x_k-e_i$, then
\[
\Bus^{\beta_n,h_n,\nu_n}(\w^n,x_k,x_{k+1})=-\Bus^{\beta_n,h_n,\nu_n}(\w^n,x_{k+1},x_k)
\]
and similarly for $\Bus^{\beta_\infty,h_\infty,\nu_\infty}(\w^\infty)$. Therefore, by the triangle inequality, it suffices to prove the claim for $y=x+e_j$ with $j\in\{1,2\}$, and, by shift covariance, to only consider $x=0$.

Fix a sequence $(\beta_n,h_n)\to(\beta_\infty,h_\infty)$ as in the statement. Let
\[
S=\R^{\bbZ^2\times\bbZ^2},
\qquad
E=[0,\infty)\times\Omega\times S,
\]
and for $n\in\bbN\cup\{\infty\}$ define the \emph{temperature}
\[
\alpha_n=
\begin{cases}
\beta_n^{-1}, & \beta_n<\infty,\\
0, & \beta_n=\infty,
\end{cases}
\]
and
\[
\Bus_n(\what)=\Bus^{\beta_n,h_n,\nu_n}(\w^n(\what))\in S.
\]
Let $\mu_n$ be the law of $(\what,\alpha_n,\w^n,\Bus_n)$ on $\Omhat\times E$. Denote by $T_{x,y}$ the coordinate shift by $x$ in the first coordinate and $y$ in the second on $S=\bbR^{\bbZ^2\times\bbZ^2}$.

By assumption \eqref{thm:L1-continuity-coupled:weights-L1}, the family
$\{\w_0^n\}_{n\in\bbN\cup\{\infty\}}$ is uniformly integrable.
For each $i\in\{1,2\}$ and each $n\in\bbN\cup\{\infty\}$, $\beta_n$-recovery gives
$0\le \Bus_n(0,e_i)^-\le (\w_0^n)^-$, so
$\{\Bus_n(0,e_i)^-\}_{n\in\bbN\cup\{\infty\}}$ is uniformly integrable. Therefore, $\sup_n\Ehat[\Bus_n(0,e_i)^-]<\infty$.
Since
\begin{equation}\label{eq:Bus-means-converge}
\Ehat[\Bus_n(0,e_i)]=-\,h_n\cdot e_i \to -\,h_\infty\cdot e_i=\Ehat[\Bus_\infty(0,e_i)],
\end{equation}
it follows that $\sup_n\Ehat[|\Bus_n(0,e_i)|]<\infty$ for $i\in\{1,2\}$. By shift covariance and the cocycle property, the same holds for $\Bus_n(x,y)$ for all $x,y\in\bbZ^2$. 

By \eqref{w-cov} and the $\That$ invariance of $\Phat$,
\[
\Ehat[|\w_x^n-\w_x^\infty|]=\Ehat[|\w_0^n-\w_0^\infty|]\to0
\qquad\text{for each }x\in\bbZ^2.
\]
Hence $\w^n\to\w^\infty$ in $\Phat$-probability as $\Omega$-valued random variables. In particular, the laws of $\w^n$ on $\Omega$ are tight, and consequently the family $\{\mu_n\}$ is tight.

Take a weakly convergent subsequence, still denoted by $\mu_n$, with limit $\mu$, and let
\[
\widehat W(\what,\alpha,\omega,b)=\what,\qquad A(\what,\alpha,\omega,b)=\alpha,\qquad W(\what,\alpha,\omega,b)=\omega,\qquad \Bhat(\what,\alpha,\omega,b)=b
\]
be the coordinate maps on $\Omhat\times E$. Let $\widetilde T$ be the continuous automorphism group on $\Omhat\times E$ defined by 
\[
\widetilde T_z(\what,\alpha,\omega,b)=\bigl(\That_z\what,\alpha,T_z\omega,T_{z,z}b\bigr).
\]
Since each $\mu_n$ is $\widetilde T$-invariant, so is $\mu$, and since each $\mu_n$ is supported on the closed set of $(\what,\alpha,\omega,b)\in\Omhat\times E$ on which the coordinate map $\Bhat$ is a cocycle, $\Bhat$ is $\mu$-almost surely a cocycle. We have that the law of $(\widehat W,A,W)$ under $\mu_n$ converges weakly to the law of $\bigl(\what,\alpha_\infty,\w^\infty(\what)\bigr)$ under $\Phat$, because $(\what,\alpha_n,\w^n)\to(\what,\alpha_\infty,\w^\infty)$ in $\Phat$-probability. Therefore, the first three-coordinate marginal of $\mu$ is the law of $\bigl(\what,\alpha_\infty,\w^\infty(\what)\bigr)$ under $\Phat$. In particular,
\[
A=\alpha_\infty\quad\text{and}\quad W=\w^\infty(\widehat W),\qquad \mu\text{-almost surely.}
\]
The law of $W$ under $\mu$ is therefore $\nu_\infty$.

Define
\[
\mathfrak{r}(\alpha,a_1,a_2)=
\begin{cases}
-\alpha\log\bigl(e^{-a_1/\alpha}+e^{-a_2/\alpha}\bigr), & \alpha>0,\\
a_1\wedge a_2, & \alpha=0.
\end{cases}
\]
The map $\mathfrak{r}$ is continuous on $[0,\infty)\times\R^2$, and $\beta_n$-recovery is equivalent to
\[
\mathfrak{r}\bigl(A,\Bhat(0,e_1)-W_0,\Bhat(0,e_2)-W_0\bigr)=0,
\qquad \mu_n\text{-almost surely.}
\]
The set
\[
C=\bigl\{(\what,\alpha,\omega,b):\mathfrak{r}\bigl(\alpha,b(0,e_1)-\omega_0,b(0,e_2)-\omega_0\bigr)=0\bigr\}
\]
is closed, and $\mu_n(C)=1$ for all $n$. Hence
\[
\mathfrak{r}\bigl(A,\Bhat(0,e_1)-W_0,\Bhat(0,e_2)-W_0\bigr)=0,
\qquad \mu\text{-almost surely.}
\]
Since $A=\alpha_\infty$ and $W=\w^\infty(\widehat W)$, $\Bhat$ is a $\beta_\infty$-recovering cocycle for $W$, $\mu$-almost surely. For $i\in\{1,2\}$, let $g_i(\what,\alpha,\omega,b)=(b(0,e_i)-\omega_0)^+$. Then $g_i$ is continuous and nonnegative on $\Omhat\times E$. Since $\mu_n(C)=1$ for all $n$ and $\mu(C)=1$, recovery implies that $g_i=b(0,e_i)-\omega_0\ge0$, $\mu_n$-almost surely and $\mu$-almost surely. Hence Portmanteau, combined with the convergence $\Ehat[\w_0^n]\to\Ehat[\w_0^\infty]$ coming from \eqref{thm:L1-continuity-coupled:weights-L1} and $h_n\to h_\infty$ from \eqref{thm:L1-continuity-coupled.params}, gives
\be\label{eq:Bus-portmanteau}
0 \leq \int \bigl(\Bhat(0,e_i)-W_0\bigr)\,d\mu \le \varliminf_{n\to\infty}\int \bigl(\Bhat(0,e_i)-W_0\bigr)\,d\mu_n = -h_{\infty}\cdot e_i - \Ehat[\w_0^\infty].
\ee
It follows that $\Bhat(0,e_i)\in L^1(\mu)$. Hence, $\Bhat$ is a shift-covariant $\beta_\infty$-recovering $L^1(\mu)$ cocycle on an extended space satisfying all of the hypotheses of Section \ref{sec:probsp}, including (by assumption) Condition \ref{cond:moment-mixing}, whose environment marginal is $\nu_\infty$. Thus, Lemma \ref{lem:h-superdiff} applies.

Set $\bar h=\int\hhB(\Bhat)\,d\mu.$ 
Since $W_0$ has the same law under $\mu$ as $\w_0^\infty$ under $\Phat$, \eqref{eq:Bus-portmanteau} gives $\bar h\cdot e_i\ge h_\infty\cdot e_i$ for $i\in\{1,2\}$.  By Lemma \ref{lem:h-superdiff}, $-\hhB(\Bhat)\in\partial\fe^{\beta_\infty,\nu_\infty}(\Uset)$, $\mu$-almost surely. Lemma \ref{lm:m'=m} then implies that $\bar h=h_\infty$. Since we assumed $-h_\infty\in\ext\partial\fe^{\beta_\infty,\nu_\infty}(\xi)$ for some $\xi\in\ri\Uset$, Lemma \ref{lm:hext} (or Lemma \ref{lem:h-superdiff}\eqref{lem:h-superdiff-extreme}) implies that $\hhB(\Bhat)=h_\infty$, $\mu$-almost surely.
If $\beta_\infty<\infty$, Theorem \ref{thm:decomp}, applied to $(\Omhat\times E,\mathscr B(\Omhat\times E),\mu,\widetilde T,W)$, gives $\Bhat=\Bus^{\beta_\infty,h_\infty,\nu_\infty}(W)$, $\mu$-almost surely. If $\beta_\infty=\infty$, the zero-temperature analogue in Condition \ref{cond:zero-temp-strong} gives $\Bhat=\Bus^{\beta_\infty,h_\infty,\nu_\infty}(W)$, $\mu$-almost surely.

Consequently every weak limit point of $\mu_n$ is the law of $\bigl(\what,\alpha_\infty,\w^\infty(\what),\Bus^{\beta_\infty,h_\infty,\nu_\infty}(\w^\infty(\what))\bigr),$ and hence $\mu_n$ converges weakly to this law. Projecting onto the first and fourth coordinates gives weak convergence of the laws of $(\what,\Bus_n)$ to the law of $\bigl(\what,\Bus^{\beta_\infty,h_\infty,\nu_\infty}(\w^\infty(\what))\bigr)$ under $\Phat$ on $\Omhat\times S$. Since the first coordinate is $\what$ in both cases, Theorem \ref{a:cor:jm8} gives weak-strong convergence on $(\Omhat,\kShat,\Phat)$. In particular,
\[
\Bus^{\beta_n,h_n,\nu_n}(\w^n,0,e_j)\to \Bus^{\beta_\infty,h_\infty,\nu_\infty}(\w^\infty,0,e_j)
\qquad\text{in }\Phat\text{-probability.}
\]

By recovery, $0\le \Bus_n(0,e_j)^-\le (\w_0^n)^-$, so the family
$\{\Bus_n(0,e_j)^-\}_{n\in\bbN}$ is uniformly integrable. Since
$\Bus_n(0,e_j)\to \Bus_\infty(0,e_j)$ in $\Phat$-probability, we obtain
$\Ehat[\Bus_n(0,e_j)^-]\to\Ehat[\Bus_\infty(0,e_j)^-]$. Together with
\eqref{eq:Bus-means-converge}, this gives
\[
\Ehat[|\Bus_n(0,e_j)|]
=
\Ehat[\Bus_n(0,e_j)]+2\Ehat[\Bus_n(0,e_j)^-]
\to
\Ehat[\Bus_\infty(0,e_j)]+2\Ehat[\Bus_\infty(0,e_j)^-]
=
\Ehat[|\Bus_\infty(0,e_j)|].
\]
Since $\Bus_n(0,e_j)\to \Bus_\infty(0,e_j)$ in $\Phat$-probability, Scheff\'e's lemma gives
\[
\Ehat[|\Bus_n(0,e_j)-\Bus_\infty(0,e_j)|]\to0. \qedhere
\]
\end{proof}

\begin{proof}[Proof of Corollary \ref{cor:L1-continuity-perturb}]
The setting of the corollary is that of Theorem \ref{thm:L1-continuity-coupled}, with $\beta_n=\beta$ for all $n$. We verify hypotheses \eqref{thm:L1-continuity-coupled:weights-L1} and \eqref{thm:L1-continuity-coupled.params} of the theorem. Since $|\eta_0|\le M$, $\Ehat[|\w_0^n-\w_0^\infty|]\le |\epsilon_n|M\to0$. 

For any $u\le v$, every admissible path from $u$ to $v$ has length $|v-u|_1$, and the perturbation changes its energy by at most $|\epsilon_n|M|v-u|_1$. Hence, for both $\beta<\infty$ and $\beta=\infty$,
$\bigl|\FE{u}{v}^{\beta,\w^n}-\FE{u}{v}^{\beta,\w^\infty}\bigr|\le |\epsilon_n|M|v-u|_1$.
Applying \eqref{sh-th} to both environments $\w^n$ and $\w^\infty$ 
gives 
\[\sup_{\xi\in\Uset}\abs{\fe^{\beta,\nu_n}(\xi)-\fe^{\beta,\nu_\infty}(\xi)}\le\abs{\epsilon_n}M.\] 
Since $-h_n\in\partial\fe^{\beta,\nu_n}(\xi)$ and $\fe^{\beta,\nu_\infty}$ is differentiable at $\xi\in\ri\Uset$, Lemma \ref{lem:supergrad-converge} gives $h_n\to h_\infty$. Finally, differentiability at $\xi$ and Proposition \ref{prop:sH-extreme} when $\beta<\infty$ (and \cite[Proposition 3.5]{Jan-Ras-Sep-25-strong-} when $\beta=\infty$) imply $(\beta,h_\infty)\in\sH_{(0,\infty]}^{\nu_\infty,\textup{ext}}$. Theorem \ref{thm:L1-continuity-coupled} now applies and gives the result.
\end{proof}

\section{Asymptotics of Gibbs-DLR measures}\label{sec:zero-temp}

We begin by proving continuity in probability of the positive-temperature extremal Gibbs measures with respect to the parameters $(\beta,h,\nu,\w)$.

\begin{proof}[Proof of Corollary \ref{cor:Gibbs-measure-continuity}]
Fix $x\in\bbZ^2$ and a finite admissible path $x_{k:m}$ with $x_k=x$. 
Assumption \eqref{thm:L1-continuity-coupled:weights-L1} in Theorem \ref{thm:L1-continuity-coupled}, together with shift covariance and the conclusion of the theorem imply that the exponent in \eqref{Pindef} converges in $\Phat$-probability to the corresponding exponent with $n=\infty$. Thus the $\Pi_x^{\beta_n,h_n,\nu_n,\w^n}$-probabilities of cylinder events converge in $\Phat$-probability. Since finite cylinders are convergence determining and there are countably many finite paths rooted at $x$, the claimed weak convergence follows.
\end{proof}

Next, we turn to the zero temperature limit.

\begin{proof}[Proof of Corollary \ref{cor:cylinder-rate}] 
By the definition of $\Pi_x^{\beta_n,h_n,\nu_n,\w^n}$ and the cocycle property, 
\begin{align*}
-\frac{1}{\beta_n}\log \Pi_x^{\beta_n,h_n,\nu_n,\w^n}(x_{k:m})
=
\Bus^{\beta_n,h_n,\nu_n}(\w^n,x,x_m)-\sum_{j=k}^{m-1}\w^n_{x_j}
=
\sum_{j=k}^{m-1}\bigl(\Bus^{\beta_n,h_n,\nu_n}(\w^n,x_j,x_{j+1})-\w^n_{x_j}\bigr).
\end{align*}
For each $j$, Theorem \ref{thm:L1-continuity-coupled} gives $\Bus^{\beta_n,h_n,\nu_n}(\w^n,x_j,x_{j+1})\to \Bus^{\infty,h_\infty,\nu_\infty}(\w^\infty,x_j,x_{j+1})$ in $L^1(\Phat)$, while assumption \eqref{thm:L1-continuity-coupled:weights-L1} in the theorem and shift covariance give $\w^n_{x_j}\to\w^\infty_{x_j}$ in $L^1(\Phat)$. This proves the first claim.  The $\beta=\infty$ recovery relation gives $\Bus^{\infty,h_\infty,\nu_\infty}(\w^\infty,x_j,x_{j+1})\ge \w^\infty_{x_j}$ for each $j$, and equality throughout is the definition of the path being a $\Bus^{\infty,h_\infty,\nu_\infty}(\w^\infty)$-geodesic.
\end{proof}

Finally, we turn to the proof of Corollary \ref{cor:subseq-quenched-ldp}.

\begin{proof}[Proof of Corollary \ref{cor:subseq-quenched-ldp}]
There are countably many pairs $(x,x_{k:m})$ consisting of a root $x\in\bbZ^2$ and an admissible finite path segment started at $x$. Applying Corollary \ref{cor:cylinder-rate} and a diagonal argument gives a deterministic subsequence $\{n_\ell\}_{\ell\ge1}$ and an event $\Omhat_1$ of full $\Phat$-probability on which, for every such pair $(x,x_{k:m})$,
\be
-\frac{1}{\beta_{n_\ell}}\log \Pi_x^{\beta_{n_\ell},h_{n_\ell},\nu_{n_\ell},\w^{n_\ell}(\what)}(x_{k:m})
\to
I_x^{h_\infty,\nu_\infty,\w^\infty(\what)}(x_{k:m})\quad\text{as }\ell\to\infty.\label{eq:LDP-lim}
\ee
Shrinking $\Omhat_1$ if necessary, we may also assume that for every $\what\in\Omhat_1$, the cocycle and recovery properties of $\Bus^{\beta_n,h_n,\nu_n}(\w^n(\what))$ hold at all sites and all $n\in\bbN\cup\{\infty\}$. We verify the large deviation principle on $\pathsp_x$ for $\what\in\Omhat_1$ by applying the Dawson-G\"artner theorem to the finite prefix projections. 

Fix $\what\in\Omhat_1$ and $x\in\bbZ^2$ with $k=x\cdot(e_1+e_2)$. For $m\ge k$, let $E_m^x$ be the finite set of admissible path segments $x_{k:m}$ started at $x$, equipped with the discrete topology, and let $p_m:\pathsp_x\to E_m^x$ be given by $p_m(x_{k:\infty})=x_{k:m}.$ Then $\pathsp_x$ is the projective limit of the system $\{E_m^x\}_{m\ge k}$ under these truncation maps. 

For $m\ge k$, define
\[
I_x^{h_\infty,\nu_\infty,\w^\infty(\what)}(x_{k:m})
=
\Bus^{\infty,h_\infty,\nu_\infty}(\w^\infty(\what),x,x_m)-\sum_{j=k}^{m-1}\w^\infty_{x_j}(\what).
\]
By the cocycle property,
\[
I_x^{h_\infty,\nu_\infty,\w^\infty(\what)}(x_{k:m})
=
\sum_{j=k}^{m-1}\bigl(\Bus^{\infty,h_\infty,\nu_\infty}(\w^\infty(\what),x_j,x_{j+1})-\w^\infty_{x_j}(\what)\bigr),
\]
and the recovery property makes this sequence nondecreasing in $m$. Thus the second equality in \eqref{eq:quenched-rate-fn} holds and says that
\[
I_x^{h_\infty,\nu_\infty,\w^\infty(\what)}(x_{k:\infty})=\sup_{m\ge k}I_x^{h_\infty,\nu_\infty,\w^\infty(\what)}(x_{k:m}).
\]
With the notation used in this paper, for each $m\ge k$ and each $x_{k:m}\in E_m^x$,
\[
(p_m)_\#\Pi_x^{\beta_{n_\ell},h_{n_\ell},\nu_{n_\ell},\w^{n_\ell}(\what)}(x_{k:m})
=
\Pi_x^{\beta_{n_\ell},h_{n_\ell},\nu_{n_\ell},\w^{n_\ell}(\what)}(x_{k:m}),
\]
so the convergence in \eqref{eq:LDP-lim} above is an LDP for the measures $(p_m)_\#\Pi_x^{\beta_{n_\ell},h_{n_\ell},\nu_{n_\ell},\w^{n_\ell}(\what)}$ on the finite space $E_m^x$ with normalization $\beta_{n_\ell}$ and rate function $I_x^{h_\infty,\nu_\infty,\w^\infty(\what)}(x_{k:m})$. The Dawson-G\"artner theorem \cite[Theorem 4.6.1]{Dem-Zei-98} now yields an LDP for $\Pi_x^{\beta_{n_\ell},h_{n_\ell},\nu_{n_\ell},\w^{n_\ell}(\what)}$ on $\pathsp_x$ with normalization $\beta_{n_\ell}$ and rate function \eqref{eq:quenched-rate-fn}.

Since each term in the series above is nonnegative, $I_x^{h_\infty,\nu_\infty,\w^\infty(\what)}(x_{k:\infty})=0$ if and only if
\[
\Bus^{\infty,h_\infty,\nu_\infty}(\w^\infty(\what),x_j,x_{j+1})=\w^\infty_{x_j}(\what)
\qquad\text{for all }j\ge k,
\]
which is exactly the condition that $x_{k:\infty}$ is a $\Bus^{\infty,h_\infty,\nu_\infty}(\w^\infty(\what))$-geodesic.
\end{proof}

\appendix
\section{Measurability}\label{app:meas}
In this section, we collect the purely technical, measure theoretic, aspects of our proofs. 
In the arguments that follow, we work with a fixed $x\in\bbZ^2$ and denote the hyperspace of non-empty compact subsets of $(\sM_1(\pathsp_x,\pathsa_{k:\infty}),$ $d_{\sM_1})$ equipped with the Hausdorff metric, by $(\sK_x,d_{\sK_x})$. Explicitly, the Hausdorff metric is given for $K_1,K_2 \in \sK_x$ by
\begin{align*}
	d_{\sK_x}(K_1,K_2) = \max\{\max_{\mu\in K_2}d_{\sM_1}(\mu,K_1),\max_{\mu\in K_1}d_{\sM_1}(\mu,K_2)\}.
\end{align*}
It is a standard fact that $(\sK_x,d_{\sK_x})$ is compact because $(\sM_1(\pathsp_x,\pathsa_{k:\infty}),d_{\sM_1})$ is \cite[Theorem 3.2.4]{Bee-93}.

\subsection{Proof of Proposition \ref{prop:dmeas}}
Recall the shift-invariant and Borel-measurable event $\Omega_0$ from Lemma \ref{lem:Omega0}. By Theorem \ref{thm:cltot}, on $\Omega_0$, for all $\beta\in(0,\infty)$ and all $x\in\bbZ^2$, $\ext \DLR_x^{\beta,\w}$ is a closed subset of the compact metric space $\sM_1(\pathsp_x,\pathsa_{k:\infty})$ and totally ordered under stochastic domination. The goal of this section is to construct for each $\beta\in(0,\infty)$ a Borel-measurable version of the map $\w \mapsto \sup A \cap \ext \DLR_x^{\beta,\w}$, where $A$ is a deterministic closed subset of $\sM_1(\pathsp_x,\pathsa_{k:\infty})$. We begin by proving that $\w \mapsto A \cap \ext \DLR_x^{\beta,\w}$ is a Borel correspondence taking values in the hyperspace of compact subsets of $\sM_1$, up to the full measure set $\Omega_0$.

\subsubsection{Measurability of $\DLR_x^{\beta}$} We first argue that the graph of $\DLR_x^{\beta,\w}$ is Borel-measurable.

\begin{lemma}\label{lem:DLRBorelGraph} 
The graph
\[
\Gr(\DLR_x)=\{(\w,\beta,\mu) \in \Omega\times(0,\infty)\times\sM_1(\pathsp_x,\pathsa_{k:\infty}) :  \mu\in \DLR_x^{\beta,\w}\}
\]
is Borel-measurable in the product space $\Omega\times(0,\infty)\times\sM_1(\pathsp_x,\pathsa_{k:\infty})$.
\end{lemma}
\begin{proof}
Recall that the topology on $\sM_1(\pathsp_x,\pathsa_{k:\infty})$ is the coarsest topology in which the evaluation maps sending a finite path segment $x_{m:n}$ accessible from $x$ to the probability of that path segment, $\mu(x_{m:n})$, are continuous.

For each $\w \in \Omega$, the map $(\beta,\mu) \mapsto$ $\mu(x_{m:n}) - \poly{x}{x_n}^{\beta,\w}(x_{m:n})\mu(x_n)$ is continuous. For each $(\beta,\mu) \in (0,\infty)\times \sM_1(\pathsp_x,\pathsa_{k:\infty})$, the map $\w\mapsto \mu(x_{m:n}) - \poly{x}{x_n}^{\beta, \w}(x_{m:n})\mu(x_n)$ is Borel-measurable in $\Omega$. It then follows that
\[
(\w,\beta,\mu) \in \Omega\times(0,\infty)\times\sM_1(\pathsp_x,\pathsa_{k:\infty}) \mapsto \mu(x_{m:n}) - \poly{x}{x_n}^{\beta, \w}(x_{m:n})\mu(x_n)
\]
is a Carath\'eodory function and hence is jointly Borel-measurable by Lemma 4.51 in \cite{Ali-Bor-06}. Therefore,
\[
\{(\w,\beta,\mu) : \mu\in \DLR_x^{\beta, \w}\} = \bigcap_{x_{m:n}} \{(\w,\beta,\mu) : \mu(x_{m:n}) - \poly{x}{x_n}^{\beta, \w}(x_{m:n})\mu(x_n) = 0\}
\]
is Borel-measurable.
\end{proof}
\begin{corollary}\label{cor:DLRmeas}
The map $\Omega\times(0,\infty)\to \sK_x$ given by $(\w,\beta) \mapsto \DLR_x^{\beta, \w}$ is $(\sF\times\mathscr{B}((0,\infty)),\mathscr{B}(\sK_x))$-measurable.
\end{corollary}
\begin{proof} 
Recall that $\DLR_x^{\beta,\w}$ is compact and non-empty for each $(\w,\beta)\in\Omega\times(0,\infty)$. By Theorem III.2 in Chapter 3 of \cite{Cas-Val-77}, measurability of $(\w,\beta) \mapsto \DLR_x^{\beta,\w}$ is equivalent to Borel measurability of $\{(\w,\beta) : \DLR_x^{\beta,\w} \cap C \neq \emptyset\}$ for all closed sets $C\subset \sM_1(\pathsp_x,\pathsa_{k:\infty})$. This is the projection
\[
\PrOmB(\{(\w,\beta,\mu) : \mu \in \DLR_x^{\beta,\w}\cap C\}).
\]
By Lemma \ref{lem:DLRBorelGraph}, $\Gr(\DLR_x^{\beta})\cap(\Omega \times (0,\infty) \times C) = \{(\w,\beta,\mu) : \mu \in \DLR_x^{\beta,\w}\cap C\}$ is Borel. For each $(\w,\beta)\in\Omega\times(0,\infty)$, the section $\{\mu : \mu \in\DLR_x^{\beta,\w}\cap C\}$ is compact. Since $\Omega\times(0,\infty)$ and $\sM_1(\pathsp_x,\pathsa_{k:\infty})$ are both Polish spaces, it then follows from the Arsenin-Kunugui Theorem \cite[Theorem 18.18]{Kec-95} that this projection is Borel.
\end{proof}

\subsubsection{Measurability of $\ext \DLR_x^{\beta}$} Measurability of $(\w,\beta) \mapsto A \cap \ext \DLR_x^{\beta,\w}$ will be argued similarly, modulo some minor technical points. We start with an intermediate lemma which encodes the (non-)extremality condition.

We will use the following characterization of a point of a compact convex set in a separable convex metric subspace of a topological vector space being a proper mixture. Because the proof is entirely general, we state the result generally. The proof is almost immediate from the definitions.
\begin{lemma}\label{lem:notext}
Let $(E,d)$ be a separable convex metric subspace of a topological vector space and let $K$ be a compact convex subset of $E$. Denote by $\ext K$ the extreme points of $K$ and let $\{y_i : i \in \bbN\}$ be a dense subset of $E$. The following are equivalent:
\begin{enumerate}[label={\rm(\alph*)}, ref={\rm\alph*}] \itemsep=2pt
\item\label{lem:notext:def} $x \in K\ \setminus\ext K$.
\item\label{lem:notext:seq} There exist $\epsilon \in (0,1)\cap\bbQ$, sequences of natural numbers $n_k, m_k \to \infty$, and a sequence of rational numbers $\alpha_k \in [\epsilon,1-\epsilon]\cap\bbQ$ for which (i) $d(y_{m_k},K)\to 0$ and $d(y_{n_k},K)\to 0$, (ii) $d(x, \alpha_k y_{m_k} + (1-\alpha_k)y_{n_k}) \to 0$, and (iii) $d(y_{m_k}, y_{n_k}) \geq \epsilon$ for all $k\in\bbN$.
\item\label{lem:notext:set} There exists $\epsilon \in (0,1)\cap\bbQ$ such that for all $N \in \mathbb{N}$ and all $\delta\in(0,1)\cap\bbQ$, there exist $n,m \geq N$  and $\alpha \in [\epsilon,1-\epsilon]\cap \bbQ$ for which (i) $d(y_{m},K)\leq \delta$ and $d(y_{n},K)\leq \delta $, (ii) $d(x, \alpha y_{m} + (1-\alpha)y_{n})\leq \delta$, and (iii) $d(y_{m}, y_{n}) \geq \epsilon$.
\end{enumerate}
\end{lemma}
\begin{proof}
\eqref{lem:notext:seq} and \eqref{lem:notext:set} are equivalent by basic set theoretic manipulations, so we just check the equivalence of \eqref{lem:notext:def} and \eqref{lem:notext:seq}.

Suppose first that $x \in K\ \setminus\ext K$ and let $x_1,x_2 \in K$ with $x_1\neq x_2$ and $\alpha \in (0,1)$ be such that $x = \alpha x_1 + (1-\alpha) x_2$. Let $\epsilon \in(0, \min\{d(x_1,x_2), \alpha, (1-\alpha)\}/2)\cap \bbQ$. Take sequences $n_k,m_k\to\infty$ for which $y_{m_k} \to x_1$ and $y_{n_k} \to x_2$ and let $\alpha_k \in (0,1)\cap \bbQ$ converge to $\alpha$. For all sufficiently large $k$, we must have $d(y_{m_k},y_{n_k}) \geq \epsilon$ and $\alpha_k \in [\epsilon,1-\epsilon]$ so without loss of generality, we can discard all of the terms which fail these conditions. The convergence $y_{m_k} \to x_1$ implies $d(y_{m_k},K)\to 0$ and similarly $d(y_{n_k},K)\to 0.$ Finally, because addition and scalar multiplication are continuous, we have $d(x,\alpha_k y_{m_k} + (1-\alpha_k)y_{n_k})\to 0$.

Conversely, suppose that \eqref{lem:notext:seq} holds. Because $K$ is compact, for each $k \in \bbN$, there exists $x_{m_k} \in K$ for which $d(y_{m_k},K)= d(y_{m_k},x_{m_k})$. A similar statement holds for $y_{n_k}$. We may therefore extract a subsequence $k_j$ along which $y_{m_{k_j}}$ and $y_{n_{k_j}}$ both converge to points $x^{(1)}$ and $x^{(2)}$ in $K$ respectively and $\alpha_{k_j} \to \alpha \in [\epsilon,1-\epsilon]$. It follows from the hypotheses that $x = \alpha x^{(1)} + (1-\alpha) x^{(2)} \in K$, and $d(x^{(1)},x^{(2)})\geq \epsilon$. Therefore $x \in K\ \setminus \ext K$.
\end{proof}

With the previous result in mind, fix a countable dense subset $(\nu_i : i \in \bbN)$ of $\sM_1(\pathsp_x,\pathsa_{k:\infty})$. Denote by
\[
\ext \DLR_{x,0}^{\beta,\w} = \begin{cases}
    \ext \DLR_x^{\beta,\w} & \w \in \Omega_0 \\
    \{\delta_{x + e_1 \bbZ_{\geq 0}}, \delta_{x + e_2 \bbZ_{\geq 0}}\} & \w \notin \Omega_0.
\end{cases}
\]
The point of this re-definition is to ensure that the set is closed for all $\w\in\Omega$.
\begin{lemma}\label{lem:extDLR-meas}
For each $x \in \bbZ^2$, the map $\Omega\times(0,\infty)\mapsto \sK_x$ given by $(\w,\beta) \mapsto \ext \DLR_{x,0}^{\beta,\w}$ is $(\sF\times\mathscr{B}((0,\infty)),\mathscr{B}(\sK_x))$-measurable.
\end{lemma}
\begin{proof}
Using Lemma \ref{lem:notext}, we have that
\[
\{(\w, \beta, \mu): \mu \in \DLR_x^{\beta, \w} \ \setminus\ \ext \DLR_{x}^{\beta, \w}\}
\]
is equal to
\be\begin{aligned}
\bigcup_{\epsilon\in (0,1)\cap \bbQ} &\bigcap_{N\in\bbN}\bigcap_{\delta \in (0,1) \cap \bbQ} \bigcup_{(n,m)\in A(\epsilon,N)} \bigcup_{\alpha \in [\epsilon,1-\epsilon]\cap \bbQ}\bigl\{(\w, \beta, \mu) : d_{\sM_1}(\nu_m,\DLR_x^{\beta, \w}) \leq \delta, \\
&\qquad\qquad\qquad\qquad\qquad
d_{\sM_1}(\nu_n, \DLR_x^{\beta, \w}) \leq \delta, \quad d_{\sM_1}(\mu, \alpha \nu_m + (1-\alpha) \nu_n)\leq \delta\bigr\},
\end{aligned}\label{eq:nextmeas}\ee
where $A(\epsilon,N) = \{(n,m) \in \bbN\times \bbN : n,m \geq N, d_{\sM_1}(\nu_n,\nu_m) \geq \epsilon\}$ is countable and deterministic.

The map $\sM_1(\pathsp_x,\pathsa_{k:\infty})\times \sK_x\mapsto\mathbb{R}$ that sends $(\nu,K)$ to $d_{\sM_1}(\nu,K)$ is $1$-Lipschitz, and therefore the map sending $(\w,\beta,\nu) \in \Omega\times(0,\infty)\times\sM_1(\pathsp_x,\pathsa_{k:\infty})$ to $d_{\sM_1}(\nu,\DLR_x^{\beta,\w})$ is Borel-measurable. It follows that each of the sets in the expression in \eqref{eq:nextmeas} above is Borel-measurable. Therefore,
\[
\Gr(\ext \DLR_x)=\{(\w,\beta,\mu) : \mu \in \DLR_x^{\beta, \w}\}\ \setminus\ \{(\w,\beta,\mu) : \mu \in \DLR_x^{\beta, \w}\ \setminus\  \ext \DLR_x^{\beta, \w}\}
\]
is Borel. Since $\Omega_0$ is Borel, the graph of $(\w,\beta)\mapsto \ext \DLR_{x,0}^{\beta,\w}$ is obtained from $\Gr(\ext \DLR_x)$ by replacing the fibers over $(\Omega\ \setminus\ \Omega_0)\times(0,\infty)$ with the fixed two-point set $\{\delta_{x+e_1\bbZ_{\ge0}},\delta_{x+e_2\bbZ_{\ge0}}\}$. Hence it is Borel. Next we fix a closed set $C$ and argue as in the proof of Corollary \ref{cor:DLRmeas}. For each $(\omega,\beta)\in \Omega \times (0,\infty)$, $\ext \DLR_{x,0}^{\beta,\w}$ is compact, so it suffices to show that if $C$ is a deterministic closed subset of $\sM_1(\pathsp_x,\pathsa_{k:\infty})$, then $\{(\w,\beta) : \ext \DLR_{x,0}^{\beta,\w} \cap C \neq \varnothing\}$ is Borel. We again lift and project:
\begin{align*}
\{(\w,\beta) : \ext \DLR_{x,0}^{\beta,\w} \cap C \neq \varnothing\} = \PrOmB(\{(\w,\beta,\mu) : \mu\in \ext \DLR_{x,0}^{\beta,\w} \cap C\}).
\end{align*}
As before, for each $(\w,\beta)\in\Omega\times(0,\infty)$, the section $\{\mu : \mu \in\ext \DLR_{x,0}^{\beta,\w}\cap C\}$ is compact. Since $\Omega\times(0,\infty)$ and $\sM_1(\pathsp_x,\pathsa_{k:\infty})$ are both Polish spaces, it then follows from the Arsenin-Kunugui Theorem \cite[Theorem 18.18]{Kec-95} that this projection is Borel.
\end{proof}
The next corollary follows immediately from the previous result.
\begin{corollary}\label{cor:GammaAmeas}
Fix $x\in\bbZ^2$. Let $A\subset \sM_1(\pathsp_x,\pathsa_{k:\infty})$ be closed and define
\[
\Gamma_{x,A}^{\beta,\w}=\begin{cases}
    A\cap \ext \DLR_x^{\beta,\w} & \w \in \Omega_0, \\
    A \cap \{\delta_{x+e_2\bbZ_{\geq 0}}, \delta_{x + e_1 \bbZ_{\geq 0}}\} & \w \notin \Omega_0. \end{cases}
\]
Then
\begin{enumerate}[label={\rm(\alph*)}, ref={\rm\alph*}]
\item The graph
\[
\Gr(\Gamma_{x,A})=\{(\omega,\beta,\mu)\in\Omega\times(0,\infty)\times \sM_1(\pathsp_x,\pathsa_{k:\infty}):\ \mu\in \Gamma_{x,A}^{\beta,\w}\}
\]
is a Borel subset of $\Omega\times(0,\infty)\times \sM_1(\pathsp_x,\pathsa_{k:\infty})$.
\item The compact set-valued correspondence
$(\omega,\beta) \mapsto\Gamma_{x,A}^{\beta,\w}$ is Borel.
\end{enumerate}
\end{corollary}

Denote by $f_n:\bbX_0\to\{0,1\}$ an enumeration of the indicator functions of all order-intervals of finite paths rooted at the origin, i.e., all functions of the form
\be
f(\gamma) = \one_{\{\pi_{0:m} \preceq \gamma_{0:m}\}}\label{eq:orderint}
\ee
where $\pi_{0:m}$ is any path with $\pi_0=0$. We have the following lemma, which is immediate from the definitions.
\begin{lemma}\label{lem:Phi-separates}
With
$\Phi:\sM_1(\pathsp_0,\pathsa_0)\to [0,1]$ defined by
\[
\Phi(\mu)\;=\;\sum_{n=1}^\infty 2^{-n}\int f_n\,d\mu,
\]
$\Phi$ is continuous and satisfies $\Phi(\mu) \leq \Phi(\nu)$ whenever $\mu \preceq \nu$, and $\Phi(\mu) < \Phi(\nu)$ when $\mu \precneq \nu$.
\end{lemma}

\begin{proposition}\label{prop:meas-sup}
For each $x\in\bbZ^2$ and closed $A\subset \sM_1(\pathsp_x,\pathsa_{k:\infty})$, there exists a Borel-measurable map
\[
(\w,\beta)\in\Omega\times(0,\infty)\ \longmapsto\ \mu_{A}^{\beta,\w}\in\sM_1(\pathsp_x,\pathsa_{k:\infty})
\]
such that the following hold:
\begin{enumerate}[label={\rm(\alph*)}, ref={\rm\alph*}]
\item\label{prop:meas-sup:empty} If $\w\in \Omega_0$ and $A\cap \ext \DLR_x^{\beta,\w}=\varnothing$, then $\mu_{A}^{\beta,\w}= \delta_{x+e_2\bbZ_{\geq 0}}$.
\item \label{prop:meas-sup:inset} If $\w \in \Omega_0$ and $A\cap \ext \DLR_x^{\beta,\w}\neq\varnothing$, then 
\[
\mu_{A}^{\beta,\w}\;=\;\sup\{ A\cap \ext \DLR_x^{\beta,\w}\} \in A\cap \ext \DLR_x^{\beta,\w}.
\]
\end{enumerate}
\end{proposition}

\begin{proof}Without loss of generality, we take $x=0$. $\{(\w,\beta) : A\cap \ext \DLR_0^{\beta,\w}=\varnothing\}$ is Borel, as is $\Omega_0\times \sM_1(\pathsp_0,\pathsa_{0:\infty})$, so we may take \eqref{prop:meas-sup:empty} as a definition. 

Let
$B\;=\;\{(\w,\beta)\in\Omega\times(0,\infty):\ \Gamma_{0,A}^{\beta,\w}\neq\varnothing\}.$ On $B$, define the value function $v(\w,\beta)\;=\;\sup_{\mu\in\Gamma_{0,A}^{\beta,\w}}\Phi(\mu)$ and the argmax correspondence 
\[
M(\w,\beta)\;=\;\bigl\{\mu\in\Gamma_{0,A}^{\beta,\w}:\ \Phi(\mu)=v(\w,\beta)\bigr\}.
\]
Off of $B$, we set $v(\w,\beta)=0$ and $M(\w,\beta)=\{\delta_{e_2\bbZ_{\geq 0}}\}$. Because $\Gamma_{0,A}$ is measurable and compact-valued and $\Phi$ is continuous, $M$ is never empty. The measurable maximum theorem \cite[Theorem 18.19]{Ali-Bor-06} implies that $v$ is measurable and there exists a measurable selection
$(\w,\beta) \mapsto \mu_{A}^{\beta,\w}$ such that
$\mu_A^{\beta,\w}\in M(\w,\beta)$ for all $(\w,\beta)\in \Omega\times(0,\infty)$.  For $\w\in\Omega_0$, the maximizer is unique by Lemma \ref{lem:Phi-separates} and the fact that $\ext \DLR_0^{\beta,\w}$ is totally ordered. Recall the discussion at the start of Section \ref{sec:lriso} which shows that every non-empty totally ordered subset of $\sM_1(\pathsp_0,\pathsa_{0:\infty})$ has a unique supremum, which lies in the set if it is weakly closed.  Hence we can conclude that $M(\w,\beta)=\{\sup(A \cap \ext \DLR_0^{\beta,\w})\}$. 
\end{proof}

\subsubsection{Completing the proof}
With the previous result in hand, we now complete the proof of Proposition \ref{prop:dmeas}.
\begin{proof}[Proof of Proposition \ref{prop:dmeas}]
Let $A_{0,I}\subset\sM_1(\pathsp_0,\pathsa_{0:\infty})$ be the closed set
\[
A_{0,I}=\Bigl\{\mu\in\sM_1(\pathsp_0,\pathsa_{0:\infty}):\ \mu(\pi^{0,n,j})\in I^j\ \ \forall\,j=1,\dots,2^n\Bigr\}.
\]
Apply Proposition~\ref{prop:meas-sup} with $A=A_{0,I}$ to obtain a Borel-measurable map
\[
(\omega,\beta)\longmapsto \mu^{\beta,\omega,+}_{0,I}\in\sM_1(\pathsp_0,\pathsa_{0:\infty})
\]
such that for $\omega\in\Omega_0$,
\[
\mu^{\beta,\omega,+}_{0,I}=
\begin{cases}
\delta_{e_2\mathbb{Z}_{\ge0}} & \text{if }A_{0,I}\cap\ext\DLR_0^{\beta,\omega}=\varnothing,\\
\sup\bigl(A_{0,I}\cap\ext\DLR_0^{\beta,\omega}\bigr) & \text{if }A_{0,I}\cap\ext\DLR_0^{\beta,\omega}\neq\varnothing.
\end{cases}
\]
$\mu^{\beta,\omega,-}_{0,I}$ can be constructed similarly. For $x\in\mathbb{Z}^2$ and $\sigg\in\{+,-\}$, we take \eqref{eq:dmeas-def-shift} as the definition of $\mu^{\beta,\omega,\sigg}_{x,I}$ so that \eqref{prop:dmeas:mu-covariant} holds by construction. \eqref{prop:dmeas:Dmeas-covariant} follows immediately from \eqref{prop:dmeas:mu-covariant}, by definition.

Fix $\omega\in\Omega_0$ and $\beta\in(0,\infty)$.
By shift-invariance of $\Omega_0$, $T_x\omega\in\Omega_0$. The DLR equations are shift-covariant, so we have
\begin{align*}
\DLR_x^{\beta,\omega}=(\theta_x)_\#\,\DLR_0^{\beta,T_x\omega},
\qquad\text{ and }\qquad
\ext\DLR_x^{\beta,\omega}=(\theta_x)_\#\,\ext\DLR_0^{\beta,T_x\omega}.
\end{align*}
By our convention that $\pi^{x,n,j}=\theta_x\pi^{0,n,j}$, we also have $A_{x,I}=(\theta_x)_\#\,A_{0,I}.$ We conclude that
\begin{align*}
A_{x,I}\cap\ext\DLR_x^{\beta,\omega}
\;=\;
(\theta_x)_\#\bigl(A_{0,I}\cap\ext\DLR_0^{\beta,T_x\omega}\bigr),
\end{align*}
and consequently 
\begin{align*}
\sup\bigl\{A_{x,I}\cap\ext\DLR_x^{\beta,\omega}\bigr\}
&=(\theta_x)_{\#}\,\sup\bigl\{A_{0,I}\cap\ext\DLR_0^{\beta,T_x\omega}\bigr\},\\
\inf\bigl\{A_{x,I}\cap\ext\DLR_x^{\beta,\omega}\bigr\}
&=(\theta_x)_{\#}\,\inf\bigl\{A_{0,I}\cap\ext\DLR_0^{\beta,T_x\omega}\bigr\}.
\end{align*}
Combining these identities with \eqref{eq:dmeas-def-shift} yields \eqref{prop:dmeas:empty} and \eqref{prop:dmeas:inset}.

We next prove \eqref{prop:dmeas:dense}. Countability of $\Dmeas_x^{\beta,\omega}$ is immediate since $\bigcup_{n\ge1}\dyatup_n$ is countable. To prove density, let $\mu\in\ext\DLR_x^{\beta,\omega}$ and let $U$ be any open neighborhood of $\mu$ in
$\sM_1(\pathsp_x,\pathsa_{k:\infty})$.
There exist $n\in\mathbb{N}$ and open
intervals $J^j\subset(0,1)$, $j=1,\dots,2^n$, such that
\[
\mu\in \Bigl\{\nu:\ \nu(\pi^{x,n,j})\in J^j\ \ \forall\,j\in\{1,\dots,2^n\}\Bigr\}\ \subset\ U.
\]
Using density of dyadic rationals, choose closed dyadic intervals $I^j\in\dy$ with
$\mu(\pi^{x,n,j})\in I^j\subset J^j$ for all $j$, and set $I=(I^1,\dots,I^{2^n})$. Then $A_{x,I}\subset U$ and $\mu\in A_{x,I}\cap\ext\DLR_x^{\beta,\omega}$, so
$A_{x,I}\cap\ext\DLR_x^{\beta,\omega}\neq\varnothing$ and consequently
$\mu^{\beta,\omega,\pm}_{x,I}\in A_{x,I}\subset U$ by \eqref{prop:dmeas:inset}.
Thus $\Dmeas_x^{\beta,\omega}\cap U\neq\varnothing$ for every neighborhood $U$ of $\mu$.

To show \eqref{prop:dmeas:isolated}, we consider the case where $\mu\in \ext \DLR_x^{\beta,\w}$ is right-isolated, the left-isolated case being similar.
By Lemma~\ref{lem:RI}, there exist $\nu\in \ext \DLR_x^{\beta,\w}$ with $\mu\precneq\nu$, a finite segment $x_{k:n}$ with $x_k=x$,
and $\varepsilon_0>0$ such that
\begin{equation}\label{eq:dmeas-gap}
\Bigl\{\gamma\in \ext \DLR_x^{\beta,\w}:\ \mu\precneq\gamma\ \text{ and }\ \gamma(x_{k:n}\preceq X_{k:n})<\mu(x_{k:n}\preceq X_{k:n})+\varepsilon_0\Bigr\}
=\varnothing.
\end{equation}
Let $m=n-k$ and define the index set
\[
J=\Bigl\{j\in\{1,\dots,2^m\}:\ x_{k:n}\preceq \pi^{x,m,j}\Bigr\}.
\]
Then, for any $\gamma\in\sM_1(\pathsp_x,\pathsa_{k:\infty})$,
\[
\gamma(x_{k:n}\preceq X_{k:n})=\sum_{j\in J}\gamma(\pi^{x,m,j}).
\]
For each $j\in\{1,\dots,2^m\}$, choose a dyadic interval $I^j\in\dy$ containing $\mu(\pi^{x,m,j})$. If $j\in J$, additionally require that
\[
\sup I^j<\mu(\pi^{x,m,j})+\frac{\varepsilon_0}{2|J|}.
\]
Let $I=(I^1,\dots,I^{2^m})\in\dyatup_m$. Then $\mu\in A_{x,I}$. Moreover, for any $\gamma\in A_{x,I}$,
\[
\gamma(x_{k:n}\preceq X_{k:n})
=\sum_{j\in J}\gamma(\pi^{x,m,j})
\le\sum_{j\in J}\sup I^j
<\sum_{j\in J}\mu(\pi^{x,m,j})+\frac{\varepsilon_0}{2}
=\mu(x_{k:n}\preceq X_{k:n})+\frac{\varepsilon_0}{2}.
\]
In particular, 
\[
A_{x,I}\cap \ext \DLR_x^{\beta,\w}
\subset
\bigl\{\gamma\in \ext \DLR_x^{\beta,\omega}:\ \gamma(x_{k:n}\preceq X_{k:n})<\mu(x_{k:n}\preceq X_{k:n})+\varepsilon_0\bigr\}.
\]
By \eqref{eq:dmeas-gap}, there is no $\gamma\in A_{x,I}\cap \ext \DLR_x^{\beta,\w}$ with $\mu\precneq \gamma$. Since $\mu\in A_{x,I}\cap \ext \DLR_x^{\beta,\w}$, it follows that
\[
\mu=\sup\bigl\{A_{x,I}\cap \ext \DLR_x^{\beta,\w}\bigr\}=\mu^{\beta,\w,+}_{x,I}\in \Dmeas_x^{\beta,\w}. \qedhere
\]
\end{proof}

\subsection{Measurability and properties of the conditional quantile function}\label{sec:CDFmeas}
In this section, we prove Proposition \ref{prop:xBhatproc}, which constructs the shift-covariant conditional quantile function and describes its basic properties. We start with the existence of the shift-covariant regular conditional distribution in Lemma \ref{lem:nu-exists}.

\begin{proof}[Proof of Lemma \ref{lem:nu-exists}]
Denote by $\nu^{\w(\what)}$ a Borel-measurable and proper (in the language of \cite{Bog-07}) version of the conditional distribution of $\Phat$ on $(\Omhat, \kShat)$ given the $\sigma$-algebra $\kS$ generated by the weight map $\w$. Properness of $\nu^{\w(\what)}$ means that $\nu^{\w(\what)}$ is supported on the fiber of $\w(\what)$, $\nu^{\w(\what)}(\{\what' : \w(\what') = \w(\what) \})=1$, for all $\what$ in a full measure Borel set, which we denote by $\Omhat_{\textup{fib}}$. See Example 10.4.11 in \cite{Bog-07} for existence of such a measure. Without loss of generality, replace $\Omhat_{\textup{fib}}$ by $\bigcap_{z\in\bbZ^2} \That_z\Omhat_{\textup{fib}}$ to obtain a shift-invariant event.
Let
\[
\Omega_{\textup{fib}}=\{\w\in\Omega:\nu^\w(\{\what' : \w(\what')=\w\})=1\}.
\]
Then $\bbP(\Omega_{\textup{fib}})=1$. By \eqref{w-cov}, there is a shift-invariant Borel event $\Omhat_{\textup{w-cov}}$ of full $\Phat$-measure on which $\w(\That_z\what)=T_z\w(\what)$ for all $z\in\bbZ^2$. 

Next, we fix a countable $\pi$-system $\sC$ generating $\kShat$. Fix any $B\in\kS$ and $A\in\sC$ and compute using the shift covariance in \eqref{w-cov} and shift invariance of $\Phat$:
\be\begin{aligned}
\Ehat[\one_B\, \nu(A)\circ \That_z] &= \Ehat\bigl[(\one_B\circ \That_{-z})\, \nu(A)\bigr] =\Ehat\bigl[\one_{T_zB}\, \nu(A)\bigr]=\Ehat\bigl[\one_{T_zB}\, \one_A\bigr]\\ &=\Ehat\bigl[(\one_B\circ \That_{-z})\, \one_A\bigr]
= \Ehat[\one_B\,\one_{\That_z^{-1}A}] = \Ehat[\one_B\,\nu(\That_z^{-1}A)].
\end{aligned}\ee

It follows that for each $A\in \sC$, $\nu^{\w(\That_z\what)}(A)=\nu^{\w(\what)}(\That_z^{-1}A)$, $\Phat$-almost surely. Take a shift-invariant Borel event $\Omhat_{\textup{cov}}$ of full $\Phat$-measure on which this holds for all $A \in \sC$ and $z \in \bbZ^2$. By the $\pi$-$\lambda$ theorem, \eqref{eq:nucov} then holds for all $A \in \kShat$ and $\what\in\Omhat_{\textup{cov}}$.

Now define
\[
\Omhat_{\textup{pre}}=\Omhatext \cap \Omhat_{\textup{fib}}\cap \Omhat_{\textup{cov}}\cap\Omhat_{\textup{w-cov}}\cap \bigcap_{z\in\bbZ^2}\That_z \w^{-1}(\Omega_0).
\]
Since $\Phat(\Omhat_{\textup{pre}})=1$, the set
\[
\Omega_{\textup{full}}'=\{\omega\in\Omega:\nu^\omega(\Omhat_{\textup{pre}})=1\}\cap\Omega_{\textup{fib}}
\]
has full $\bbP$-measure. Let $\Omega_{\textup{full}}=\bigcap_{z\in\bbZ^2}T_z\Omega_{\textup{full}}'$, set $\Omega_{\Bhat}=\Omega_0\cap\Omega_{\textup{full}}$, and define
\[
\Omgood=\Omhat_{\textup{pre}}\cap \w^{-1}(\Omega_{\Bhat}).
\]
Then $\Omega_{\Bhat}$ is a shift-invariant Borel subset of $\Omega_0$ with full $\bbP$-measure, and $\Omgood\subset\Omhatext$ is Borel and has full $\Phat$-measure. It is shift-invariant because $\Omhat_{\textup{pre}}$ is shift-invariant, $\Omega_{\Bhat}$ is shift-invariant, and $\w(\That_z\what)=T_z\w(\what)$ on $\Omhat_{\textup{pre}}$. This also gives \eqref{lem:nu-exists:Omega0}, and \eqref{lem:nu-exists:cov} follows from $\Omgood\subset\Omhat_{\textup{cov}}$. Finally, if $\omega\in\Omega_{\Bhat}$, then $\nu^\omega(\Omhat_{\textup{pre}})=1$ and $\nu^\omega(\{\what':\w(\what')=\omega\})=1$. Since the fiber of $\omega$ is contained in $\w^{-1}(\Omega_{\Bhat})$, \eqref{lem:nu-exists:proper} follows.
\end{proof}

\begin{proof}[Proof of Proposition \ref{prop:xBhatproc}]
We start with claim \eqref{prop:inf-rational-meas:meas}. The event
\[\{\what' : \what' \in \Omgood,\, \Pi_x^{\Bhat,\what'} \preceq \mu_{x,I}^{\beta,\w(\what),\sigg}\}\]
is measurable for each fixed $\sigg\in\{+,-\}$ and each fixed $I \in \dyatup_m$, where $m\in\bbZ_{\geq 1}$. This is because ordering of two measures can be checked by comparing the probabilities assigned to the countable collection of increasing events $A$ with $A\in\pathsa_{k:n}$ for some $n>k.$ Below, we use the countable, separating collection of increasing events given by order intervals of the form 
\be\label{eq:path-interval}\{\gamma \in \pathsp_x : \pi_{k:n} \preceq \gamma_{k:n}\},\ee 
where $\pi_{k:n}$ is any finite path with $\pi_k=x$.

Properness of $\nu^{\w(\what)}$ says that $\nu^{\w(\what)}\left\{\what' : \what'\in\Omgood,\, \w(\what')=\w(\what) \right\}=1$ for all $\what$ with $\w(\what)\in\Omega_{\Bhat}$. Thus if $\mu \in \ext \DLR_x^{\beta,\w(\what)}$, then $\mu\in\ext \DLR_x^{\beta,\w(\what')}$ for $\nu^{\w(\what)}$-almost all $\what'$. Therefore, for $\nu^{\w(\what)}$-almost all $\what'$, $\Pi_x^{\Bhat,\what'}$ is extremal and this measure and $\mu$ are ordered. To see that the infimum is measurably defined, recall that for $\w(\what)\in\Omega_{\Bhat}$, $\ext \DLR_x^{\beta,\w(\what)}$ is compact and totally ordered. We can thus define the infimal measure by specifying its values on the countable collection of order intervals in \eqref{eq:path-interval}. For such an order interval $A$, we have
\[
\cdfm_x^{\Bhat,s}(\what)(A) =
\inf\{\mu(A) :\,\mu \in \Dmeas_x^{\beta,\w(\what)}, \nu^{\w(\what)}(\Omgood,\, \Pi_x^{\Bhat} \preceq \mu) \geq s \},
\]
which is $\kS$-measurable as an infimum of countably many $\kS$-measurable real random variables. Next, observe that when $\w(\what)\in\Omega_{\Bhat}$, the set of measures in \eqref{eq:inf-rational-formula} is non-empty because it contains the trivial measure $\delta_{x+e_1\bbZ_{\geq 0}}$. On $\Omega_0$, $\ext \DLR_x^{\beta,\w(\what)}$ is closed and totally ordered. Off of $\Omega_{\Bhat}$, the random variable is one of the trivial measures and so is automatically extremal. \eqref{prop:inf-rational-meas:meas} is proved.

We next check \eqref{prop:inf-rational-meas:fullinf}. Note that if $r < s$ and both are rational, then using definition \eqref{eq:inf-rational-formula}, we have $\cdfm_x^{\Bhat,r}(\what) \preceq \cdfm_x^{\Bhat,s}(\what)$. We have the inequality for $\w(\what)\in\Omega_{\Bhat}$ and $s \in (0,1)\cap\bbQ$,
\begin{align*}
\cdfm_x^{\Bhat,s}(\what) &=  \inf\{\mu \in \Dmeas_x^{\beta,\w(\what)} : \nu^{\w(\what)}(\Omgood,\, \Pi_x^{\Bhat} \preceq \mu) \geq s \} \\
&\geq \inf\{\mu \in \ext \DLR_x^{\beta,\w(\what)} : \nu^{\w(\what)}(\Omgood,\, \Pi_x^{\Bhat} \preceq \mu) \geq s \} = \wt{\bfm}^s.
\end{align*}
To prove the reverse inequality in \eqref{prop:inf-rational-meas:fullinf}, fix $\what$ with $\w(\what)\in\Omega_{\Bhat}$ and
$\mu\in\ext\DLR_x^{\beta,\w(\what)}$ such that
\[
\nu^{\w(\what)}\bigl(\Omgood,\ \Pi_x^{\Bhat}\preceq \mu\bigr)\ge s.
\]
Because $\Dmeas_x^{\beta,\w(\what)}$ is dense and contains all right-isolated elements of $\ext \DLR_x^{\beta,\w(\what)}$, there exists a  (weakly) decreasing sequence $(\mu_n)_{n\ge1}\subset
\Dmeas_x^{\beta,\w(\what)}$ such that $\mu\preceq \mu_n$ for all $n$ and $\mu_n\to \mu$. Then 
\[
\nu^{\w(\what)}\bigl(\Omgood,\ \Pi_x^{\Bhat}\preceq \mu_n\bigr)\ge
\nu^{\w(\what)}\bigl(\Omgood,\ \Pi_x^{\Bhat}\preceq \mu\bigr)\ge s
\qquad\text{for all }n.
\]
Therefore each $\mu_n$ is admissible in the infimum in \eqref{eq:inf-rational-formula}, so
$\cdfm_x^{\Bhat,s}(\what)\preceq \mu_n$ for all $n$.  Letting $n\to\infty$ gives $\cdfm_x^{\Bhat,s}(\what)\preceq \mu$. Taking the infimum over all such
$\mu$ yields $\cdfm_x^{\Bhat,s}(\what)\preceq \wt{\bfm}^s$, which combined with the previously proved
$\cdfm_x^{\Bhat,s}(\what)\succeq \wt{\bfm}^s$ implies \eqref{eq:fullinf} for $s\in(0,1)\cap\bbQ$.

We turn to part \eqref{prop:inf-rational-meas:inf-formula}. To prove the claim that \eqref{eq:inf-rational-formula} and \eqref{eq:inf-formula} agree, first note that equality is clear when $\w(\what)\notin\Omega_{\Bhat}$. Fix
$s\in(0,1)\cap\bbQ$ and $\what$ with $\w(\what)\in\Omega_{\Bhat}$. Abbreviate
\[
F_{\what}(\mu)
=\nu^{\w(\what)}\bigl(\Omgood,\ \Pi_x^{\Bhat}\preceq \mu\bigr),
\qquad \mu\in \ext\DLR_x^{\beta,\w(\what)}.
\]
By construction $F_{\what}$ is nondecreasing in $\mu$ and for each rational
$r\in(0,1)$ we have $F_{\what}(\cdfm_x^{\Bhat,r}(\what))\ge r$. To see this, take any sequence $\mu_n \in \ext \DLR_x^{\beta,\w(\what)}$ with $F_{\what}(\mu_n)\ge r$ which decreases to $\cdfm_x^{\Bhat,r}(\what)$. Continuity of measure then implies that $F_{\what}(\cdfm_x^{\Bhat,r}(\what))\ge r$.

Next, let $r_n$ be any strictly increasing sequence of rationals in $(0,1)$ that converges to $s$. Define
\[
\underline{\bfm}^s(\what)=\sup\bigl\{\cdfm_x^{\Bhat,r}(\what): r\in\bbQ\cap(0,1),\ r<s\bigr\}
=\sup_{n}\cdfm_x^{\Bhat,r_n}(\what).
\]
Monotonicity in $r$ implies $\cdfm_x^{\Bhat,r_n}(\what)\nearrow \underline{\bfm}^s(\what)$, where the limit and supremum are in the compact totally ordered set $\ext\DLR_x^{\beta,\w(\what)}$. Since $F_{\what}$ is nondecreasing and
$\underline{\bfm}^s(\what)\succeq \cdfm_x^{\Bhat,r_n}(\what)$ for each $n$,
\[
F_{\what}\bigl(\underline{\bfm}^s(\what)\bigr)\ \ge\ F_{\what}\bigl(\cdfm_x^{\Bhat,r_n}(\what)\bigr)\ \ge\ r_n
\qquad\text{for all }n.
\]
Letting $n\to\infty$ yields $F_{\what}(\underline{\bfm}^s(\what))\ge s$, so
$\underline{\bfm}^s(\what)$ is admissible for the infimum in \eqref{eq:fullinf}. Hence
\[
\cdfm_x^{\Bhat,s}(\what)\ \preceq\ \underline{\bfm}^s(\what).
\]
The reverse inequality $\underline{\bfm}^s(\what)\preceq \cdfm_x^{\Bhat,s}(\what)$ follows from
$r<s\Rightarrow \cdfm_x^{\Bhat,r}(\what)\preceq \cdfm_x^{\Bhat,s}(\what)$. Thus,
\[
\cdfm_x^{\Bhat,s}(\what)
=\sup\bigl\{\cdfm_x^{\Bhat,r}(\what): r\in\bbQ\cap(0,1),\ r<s\bigr\}.
\]
This proves \eqref{prop:inf-rational-meas:inf-formula}.

We now establish the remaining claims for the process $s\mapsto \cdfm_x^{\Bhat,s}$ defined by
\eqref{eq:inf-formula} for all $s\in[0,1]$, starting with \eqref{prop:xBhatproc.meas}.

For each fixed $s\in[0,1]$, the random variable $\cdfm_x^{\Bhat,s}$ is $\kS$-measurable as an
infimum or supremum of countably many $\kS$-measurable random variables
$\cdfm_x^{\Bhat,r}$, $r\in(0,1)\cap\bbQ$. The map
\[
\what \longmapsto \bigl(\cdfm_x^{\Bhat,r}(\what)\bigr)_{r\in\bbQ\cap[0,1]}
\]
is therefore $\kS$-measurable into the countable product
$\bigl(\sM_1(\pathsp_x,\pathsa_{k:\infty})\bigr)^{\bbQ\cap[0,1]}$.

For each fixed $\what$ with $\w(\what)\in\Omega_{\Bhat}$, $s\mapsto \cdfm_x^{\Bhat,s}(\what)$ is nondecreasing in $s$ and,
by definition \eqref{eq:inf-formula}, is left-continuous on $(0,1]$ and right-continuous at $0$. Since $\ext\DLR_x^{\beta,\w(\what)}$ is
compact and totally ordered, for each $s\in[0,1)$ the right limit exists and equals
$\lim_{t\searrow s}\cdfm_x^{\Bhat,t}(\what)$, again as an element of
$\ext\DLR_x^{\beta,\w(\what)}$. Hence $\cdfm_x^{\Bhat,\aabullet}(\what)\in
\Skor([0,1],\ext\DLR_x^{\beta,\w(\what)})$ for all $\what$ with $\w(\what)\in\Omega_{\Bhat}$. 

Since the Borel $\sigma$-algebra on $\Skor([0,1],\sM_1(\pathsp_x,\pathsa_{k:\infty}))$ is generated
by the coordinate maps at rational times, the previous two paragraphs combine to imply that
$\what\mapsto \cdfm_x^{\Bhat,\aabullet}(\what)$ is $\kS$-measurable as a
$\Skor([0,1],\sM_1(\pathsp_x,\pathsa_{k:\infty}))$-valued random variable. This completes the proof of
\eqref{prop:xBhatproc.meas}.

Next, consider \eqref{prop:xBhatproc.inf}.
Fix $\what$ with $\w(\what)\in\Omega_{\Bhat}$ and $s\in(0,1)$. Let  
\[
A_s(\what)=\Bigl\{\mu\in\ext\DLR_x^{\beta,\w(\what)}:\ 
\nu^{\w(\what)}\bigl(\Omgood,\ \Pi_x^{\Bhat}\preceq \mu\bigr)\ge s\Bigr\}.
\]
If $\mu\in A_s(\what)$ and $r<s$ then $\mu\in A_r(\what)$, so $\cdfm_x^{\Bhat,r}(\what)\preceq \mu$
for all rationals $r<s$. Taking the supremum over rationals $r<s$ gives
$\cdfm_x^{\Bhat,s}(\what)\preceq \mu$. Since $\mu\in A_s(\what)$ was arbitrary,
$\cdfm_x^{\Bhat,s}(\what)\preceq \inf A_s(\what)$.

Conversely, for any increasing sequence $r_n\nearrow s$ in $\bbQ\cap(0,1)$ with $r_n<s$, the same
argument as above shows
\[
\nu^{\w(\what)}\bigl(\Omgood,\ \Pi_x^{\Bhat}\preceq \cdfm_x^{\Bhat,s}(\what)\bigr)
\ \ge\ r_n
\qquad\text{for all }n.
\]
Letting $n\to\infty$ yields that $\cdfm_x^{\Bhat,s}(\what)\in A_s(\what)$, so
$\inf A_s(\what)\preceq \cdfm_x^{\Bhat,s}(\what)$. Hence $\cdfm_x^{\Bhat,s}(\what)=\inf A_s(\what)$,
which is exactly \eqref{eq:fullinf} for all $s\in(0,1)$. Monotonicity
$\cdfm_x^{\Bhat,r}(\what)\preceq \cdfm_x^{\Bhat,s}(\what)$ for $r\le s$ follows immediately.

We turn next to \eqref{prop:xBhatproc.key}. Fix $\what$ with $\w(\what)\in\Omega_{\Bhat}$ and $\mu\in\ext\DLR_x^{\beta,\w(\what)}$. With $F_{\what}$ as above,
the inclusion $(0,1)\cap[0,F_{\what}(\mu)]\subset\{s\in(0,1):\cdfm_x^{\Bhat,s}(\what)\preceq \mu\}$ is immediate:
if $s\in(0,1)$ and $s\le F_{\what}(\mu)$ then $\mu\in A_s(\what)$, and hence
$\cdfm_x^{\Bhat,s}(\what)=\inf A_s(\what)\preceq \mu$.

For the reverse inclusion, suppose $s\in(0,1)$ and $\cdfm_x^{\Bhat,s}(\what)\preceq \mu$. By definition of
$\cdfm_x^{\Bhat,s}(\what)$ as $\inf A_s(\what)$, there exists a (weakly) decreasing sequence
$(\mu_n)_{n\ge1}\subset A_s(\what)$ with $\mu_n\to \cdfm_x^{\Bhat,s}(\what)$.
Then the events $\{\Pi_x^{\Bhat}\preceq \mu_n\}$ decrease to
$\{\Pi_x^{\Bhat}\preceq \cdfm_x^{\Bhat,s}(\what)\}$, so by continuity of measure from above, 
\[
F_{\what}\bigl(\cdfm_x^{\Bhat,s}(\what)\bigr)
=\lim_{n\to\infty}F_{\what}(\mu_n)\ \ge\ s.
\]
Since $F_{\what}$ is nondecreasing and $\mu\succeq \cdfm_x^{\Bhat,s}(\what)$, we obtain
$F_{\what}(\mu)\ge F_{\what}(\cdfm_x^{\Bhat,s}(\what))\ge s$, i.e.\ $s\le F_{\what}(\mu)$.
Thus $\{s\in(0,1):\cdfm_x^{\Bhat,s}(\what)\preceq \mu\}\subset(0,1)\cap[0,F_{\what}(\mu)]$, proving
\eqref{eq:key-sym}.

\smallskip
We turn to \eqref{prop:xBhatproc:cov}. We first prove the lifted identity on $\Omgood$. Fix $\what\in\Omgood$, $x\in\Z^2$, and $s\in(0,1)$. Take $\mu\in\ext\DLR_x^{\beta,\w(\That_y\what)}$ and define
$\mu'=(\theta_y)_\#\,\mu$.
By the shift covariance in \eqref{prop:dmeas:mu-covariant} and \eqref{prop:dmeas:Dmeas-covariant} of Proposition~\ref{prop:dmeas}, 
$\mu\in\Dmeas_x^{\beta,\w(\That_y\what)}$ if and only if $\mu'\in\Dmeas_{x+y}^{\beta,\w(\what)}$.
The same shift map gives a bijection from $\DLR_x^{\beta,\w(\That_y\what)}$ to $\DLR_{x+y}^{\beta,\w(\what)}$, and this bijection preserves convex combinations. Hence it also gives a bijection between the corresponding sets of extremal measures.
Furthermore, shift covariance of $\Pi^{\Bhat}$ gives
\[
\Pi_{x+y}^{\Bhat}(\what')\preceq \mu'
\quad\Longleftrightarrow\quad
\Pi_x^{\Bhat}(\That_y\what')\preceq \mu.
\]
Since $\Omgood$ is shift-invariant, the preceding equivalence gives
\[
\Omgood\cap\{\Pi_{x+y}^{\Bhat}\preceq\mu'\}=\That_y^{-1}(\Omgood\cap\{\Pi_x^{\Bhat}\preceq\mu\}).
\]
Using \eqref{eq:nucov}, we obtain
\[
\nu^{\w(\what)}\bigl(\Omgood,\ \Pi_{x+y}^{\Bhat}\preceq \mu'\bigr)
=
\nu^{\w(\That_y\what)}\bigl(\Omgood,\ \Pi_x^{\Bhat}\preceq \mu\bigr).
\]
Consequently, the admissible sets in \eqref{eq:fullinf} for $(x+y,\what)$ and $(x,\That_y\what)$
correspond under $\mu\mapsto (\theta_y)_\#\mu$, and taking infima preserves this correspondence.
Therefore,
\be
\cdfm_{x+y}^{\Bhat,s}(\what)=(\theta_y)_\#\,\cdfm_x^{\Bhat,s}(\That_y\what),
\label{eq:cdfm-cov-lifted}
\ee
for all $\what\in\Omgood$ and $s\in(0,1)$. Now fix $\omega\in\Omega_{\Bhat}$. By Lemma \ref{lem:nu-exists}, there exists $\what\in\Omgood$ with $\w(\what)=\omega$. Since $\w(\That_y\what)=T_y\omega$ and $\cdfm_x^{\Bhat,s}$ is a function of the weights, \eqref{eq:cdfm-cov-lifted} gives \eqref{eq:cdfm-cov} for $s\in(0,1)$. The endpoint cases follow from the definition \eqref{eq:inf-formula} and the fact that $(\theta_y)_\#$ preserves infima and suprema in the stochastic order. This completes the proof of Proposition~\ref{prop:xBhatproc}.
\end{proof}

\section{Auxiliary Results}\label{app:aux}
\subsection{Forward measurability}\label{app:forward-measurable}
We show that if the weights are i.i.d., any shift-covariant, recovering $L^1(\bbP)$ cocycle is forward-measurable. The proof is the positive-temperature analogue of the forward-measurability argument for coalescing geodesics in Lemmas 4.6 and 4.7 of \cite{Jan-Ras-Sep-25-strong-}. 

\begin{lemma}\label{lem:forward-measurable}
Assume $\beta\in(0,\infty)$ and that Condition \ref{cond:moment-mixing}\eqref{cond:moment-mixing-iid} holds. Then every $B\in\cK^\beta$ is forward-measurable: if $u\in\bbZ^2$ and $\{x,y\}\subset u+\bbZ_{\geq0}^2$, then $B(x,y)$ is measurable with respect to $\sigma\{\w_z:z\in u+\bbZ_{\geq0}^2\}$, up to $\bbP$-null sets.
\end{lemma}

\begin{proof}
Fix $B\in\cK^\beta$. Let $\Omega'\subset\Omega_0$ be a shift-invariant full-probability event on which $B$ is a recovering shift-covariant cocycle and, for every $u\in\bbZ^2$, the measure $\Pi_u^B$ generated by $B$ is fully supported, belongs to $\ext\DLR_u^{\beta,\w}$, and the family $\{\Pi_u^B:u\in\bbZ^2\}$ is consistent. Such an event exists by Lemma \ref{lem:Bhat-ext}, applied on the canonical space.

By shift covariance, it is enough to treat $u=0$. Couple two environments $(\w,\wt\w)$ so that $\wt\w_z=\w_z$ for $z\cdot e_2\ge0$ and the weights with $z\cdot e_2<0$ are resampled independently. Work on the full-probability event where $\w,\wt\w\in\Omega'$. Since the two environments agree on $\{z:z\cdot e_2\ge0\}$, we have that $\ext \DLR_{ke_1}^{\w} = \ext \DLR_{k e_1}^{\wt \w}$ for $k \geq 0$. Hence $\Pi_{ke_1}^B(\w)$ and $\Pi_{ke_1}^B(\wt\w)$ are comparable by Theorem \ref{thm:cltot}.

If either $\bbP(\Pi^B_0(\w)\precneq\Pi^B_0(\wt\w)) >0$ or $\bbP(\Pi^B_0(\w)\succneq\Pi^B_0(\wt\w)) >0$, then by exchangeability, both inequalities must hold.
This, shift covariance, and the ergodic theorem for the $e_1$-shift imply that if $\Pi_0^B(\w)\ne\Pi_0^B(\wt\w)$ with positive probability, then, $\P$-almost surely, there exist integers $\ell>k\ge 0$ such that 
\be
\Pi_{ke_1}^B(\w)\precneq\Pi_{ke_1}^B(\wt\w)
\qquad\text{and}\qquad
\Pi_{\ell e_1}^B(\w)\succneq\Pi_{\ell e_1}^B(\wt\w).
\ee
 Applying Theorem \ref{thm:global-weakorder} to the first strict inequality in the common environment ahead of $ke_1$ shows that the corresponding representatives at root $\ell e_1$ satisfy $\Pi_{\ell e_1}^B(\w)\preceq\Pi_{\ell e_1}^B(\wt\w)$, because the families generated by $B(\w)$ and $B(\wt\w)$ are, respectively, consistent. This contradicts the second strict inequality. Hence $\Pi_0^B(\w)=\Pi_0^B(\wt\w)$ almost surely under this coupling. Standard measure theory (for example, \cite[Lemma A.2]{Kur-07}) gives a $\sigma\{\w_z:z\cdot e_2\ge0\}$-measurable version of $\Pi_0^B$. Repeating the same argument with $e_1$ and $e_2$ interchanged gives a $\sigma\{\w_z:z\cdot e_1\ge0\}$-measurable version. Since the weights are independent, the intersection of these two completed product $\sigma$-algebras is the completed $\sigma\{\w_z:z\in\bbZ_{\geq0}^2\}$. Thus $\Pi_0^B$ is measurable with respect to the weights in $\bbZ_{\geq0}^2$. By \eqref{eq:DLR-transition} and the cocycle property, this implies that for $x,y\in\Z_{\ge0}^2$, $B(x,y)$ is $\sigma\{\w_z:z\in\Z_{\ge0}^2\}$-measurable.
 \end{proof}

\subsection{Convex Analysis}\label{app:convex}
The proof of Lemma \ref{lem:h-superdiff} relies on a number of convex analytic auxiliary results about functions $\fe:\R_{\ge0}^2\to\R$ that are concave and positively homogeneous of degree one. 

As a finite concave function, $\fe$ is continuous on the convex open set $\ri\Uset$ and lower semicontinuous on $\Uset$.
In particular, the function is bounded below on $\Uset$.
By Theorem 10.3 in \cite{Roc-70}, $\fe$ has a unique continuous extension $\bar\fe$ from $\ri\Uset$ to the whole $\Uset$. (It is not hard to see that this extension agrees with the upper semicontinuous regularization of $\fe$. See, e.g., page 726 in \cite{Ras-Sep-14}.) For $\xi \in \bbR_{\geq 0}^2$ with $\xi \neq 0$, set $\bar\fe(\xi)=|\xi|\bar\fe(\xi/|\xi|)$, where $|\xi|=|\xi\cdot e_1|+|\xi\cdot e_2|$, and $\bar\fe(0)=0$. This gives the continuous extension of $\fe$ from $\R_{>0}^2$ to $\R_{\ge0}^2$. This extension is concave and positively homogeneous of degree one.

For $m\in\bbR^2$, define the restricted Legendre transform
\[
\fe^{\star}(m) = \sup_{\xi\in\Uset}\{\fe(\xi)-m\cdot\xi\}.
\]
Lower semicontinuity of $\fe$ implies that $\fe(e_i)\le\bar\fe(e_i)$ for $i\in\{1,2\}$ and, consequently, $\fe^\star=\bar\fe^\star$.

\begin{lemma}\label{lem:superdiff-sufficient}
Suppose $\fe:\R_{\ge0}^2\to\R$ is concave and positively homogeneous of degree one. Then for $m\in\bbR^2$:
\begin{enumerate}[label={\rm(\alph*)}, ref={\rm\alph*}] \itemsep=3pt
\item\label{lem:superdiff-sufficient.a} If $m\in\partial\fe(\xi)$ for $\xi\in\R_{\ge0}^2$, then $m\cdot\xi=\fe(\xi)$ and $m\cdot\zeta\ge\fe(\zeta)$ for all $\zeta\in\R_{\ge0}^2$.
\item\label{lem:superdiff-sufficient.b} 
If $m\in\partial\fe(\Uset)$ and $m\cdot\zeta=\fe(\zeta)$ for some $\zeta\in\Uset$, then $m\in\partial\fe(\zeta)$.
\item\label{lem:superdiff-sufficient.c} If $m\in\partial\fe(\Uset)$, then $\fe^\star(m)=0$.
\item\label{lem:superdiff-sufficient.d} If $\fe^\star(m)=0$, then $m\in\partial\bar\fe(\Uset)$.
\end{enumerate}
\end{lemma}

\begin{proof}
By the supergradient inequality,
\[
\fe(\zeta)\le \fe(\xi)+m\cdot(\zeta-\xi)\qquad\text{for all }\zeta\in\bbR_{\geq 0}^2.
\]
Taking $\zeta=c\xi$ and using positive homogeneity gives
\[
c\fe(\xi)\le \fe(\xi)+(c-1)m\cdot\xi\qquad\text{for all }c>0,
\]
hence $m\cdot\xi=\fe(\xi)$. Substituting back yields $\fe(\zeta)\le m\cdot\zeta$ for all $\zeta\in\R_{\ge0}^2$. Part \eqref{lem:superdiff-sufficient.a} is proved. For \eqref{lem:superdiff-sufficient.b}, if $m\cdot\zeta=\fe(\zeta)$ for some $\zeta\in\Uset$, then \eqref{lem:superdiff-sufficient.a} implies
\[
\fe(\eta)\le m\cdot\eta=\fe(\zeta)+m\cdot(\eta-\zeta)\qquad\text{for all }\eta\in\R^2_{\ge0},
\]
so $m\in\partial\fe(\zeta)$. For \eqref{lem:superdiff-sufficient.c}, if $m\in\partial\fe(\xi)$ for $\xi\in\Uset$, then \eqref{lem:superdiff-sufficient.a} gives $\fe(\zeta)-m\cdot\zeta\le0$ for all $\zeta\in\Uset$ with equality at $\zeta=\xi$, so $\fe^\star(m)=0$. Finally, $\fe^\star(m)=0$ implies $\bar\fe^\star(m)=0$ and since $\bar\fe$ is continuous on $\Uset$, the supremum in the definition of $\bar\fe^\star(m)$ is attained at some $\xi\in\Uset$, and then
\[
\bar\fe(\zeta)-m\cdot\zeta\le 0=\bar\fe(\xi)-m\cdot\xi\qquad\text{for all }\zeta\in\Uset.
\]
Homogeneity extends the inequality to all $\zeta\in\R_{\ge0}^2$. It then follows that $m\in\partial\bar\fe(\xi)\subset\partial\bar\fe(\Uset)$.
\end{proof}

\begin{lemma}\label{lm:m'=m}
    Suppose $\fe:\R_{\ge0}^2\to\R$ is concave and positively homogeneous of degree one. Let $m\in\partial\fe(\xi)$ for $\xi\in\ri\Uset$ and take $m'$ in the convex hull of $\partial\fe(\Uset)$. If $m'\le m$ coordinatewise, then $m'=m$.
\end{lemma}

\begin{proof}
By Lemma \ref{lem:superdiff-sufficient}\eqref{lem:superdiff-sufficient.a}, every $m''\in\partial\fe(\Uset)$ satisfies $m''\cdot\zeta\ge\fe(\zeta)$ for all $\zeta\in\R_{\ge0}^2$.
The same holds for every convex mixture of such vectors. Hence
\begin{align}\label{aux2415}
m'\cdot\zeta\ge\fe(\zeta)\qquad\text{ for all }\zeta\in\R_{\ge0}^2.
\end{align}
On the other hand, since $m\in\partial\fe(\xi)$, the same lemma implies that $m\cdot\zeta\ge\fe(\zeta)$ for all $\zeta\in\R_{\ge0}^2$, with equality at $\zeta=\xi$. 
Together with $m'\le m$ coordinatewise, this yields
$\fe(\xi)\le m'\cdot\xi\le m\cdot\xi=\fe(\xi)$,
so $m'\cdot\xi=m\cdot\xi$. Since $\xi\in\ri\Uset$ and $m'\le m$ coordinatewise, we conclude that $m'=m$.
\end{proof}

\begin{lemma}\label{lm:hext} 
Suppose $\fe:\R_{\ge0}^2\to\R$ is concave and positively homogeneous of degree one. Let $C$ be the convex hull of $\partial\fe(\Uset)$. Then, for each $\xi\in\Uset$, $\partial\fe(\xi)$ is an exposed face of $C$. Furthermore, 
\[\ext C=\bigcup_{\xi\in\Uset}\ext\partial\fe(\xi),\]
with the understanding that for $\xi\in\{e_1,e_2\}$, $\partial\fe(\xi)$ could be empty. 
\end{lemma}

\begin{proof}
For any $\xi\in\Uset$, $\partial\fe(\xi)\subset C$ and, by Lemma \ref{lem:superdiff-sufficient}\eqref{lem:superdiff-sufficient.a}, every $m\in\partial\fe(\xi)$ satisfies $m\cdot\xi=\fe(\xi)$. Conversely,
by \eqref{aux2415}, every $m\in C$ satisfies
$m\cdot \zeta \ge \fe(\zeta)$ for all $\zeta\in\R_{\ge0}^2$ and if also $m\cdot\xi=\fe(\xi)$ then, by definition, $m\in\partial\fe(\xi)$. Thus, we have shown that for every $\xi\in\Uset$,
\[\partial\fe(\xi)=\{m\in C:m\cdot\xi=\fe(\xi)\}
=C\cap\{m:m\cdot\xi=\fe(\xi)\}.\]
Consequently, $\partial\fe(\xi)$ is an exposed face of $C$.

Since an extreme point of a face is an extreme point of $C$, we have $\bigcup_{\xi\in\Uset}\ext\partial\fe(\xi)\subset \ext C$.
On the other hand, we have, by definition, $\partial\fe(\Uset)=\bigcup_{\xi\in\Uset}\partial\fe(\xi)$. Hence, 
$\ext C\subset\bigcup_{\xi\in\Uset}\ext\partial\fe(\xi)$.
\end{proof}

We now turn to the proof of Lemma \ref{lem:h-superdiff}. Recall from Remark \ref{rk:extension} that $\fe^\beta$ is continuous on $\R_{\ge0}^2$, as it is the continuous extension of the limit in Lemma \ref{lem:shape}.

\begin{proof}[Proof of Lemma \ref{lem:h-superdiff}] Recovery and the cocycle properties combine to imply that 
\be\begin{cases}
\displaystyle \max_{x \in n \Uset \cap \bbZ_{\geq 0}^2}\bigl\{\FE{0}{x}^{\infty} - \Bhat(0,x) \bigr\} =0,& \beta =\infty, \\
\log \displaystyle \sum_{x\in n \Uset \cap \bbZ_{\geq0}^2}e^{\beta(\FE{0}{x}^{\beta}-\Bhat(0,x))}=0,& \beta<\infty.
\end{cases}\label{eq:Bus-iteration}
\ee
Divide by $n$ and send $n\to\infty$. We claim that the expressions above converge to the identity
\be
0=\sup_{\xi\in\Uset}\{\fe^\beta(\xi)+\hhB(\Bhat)\cdot\xi\},\label{eq:suff-target}
\ee
which implies the claim in \eqref{lem:h-superdiff-general} by Lemma \ref{lem:superdiff-sufficient}\eqref{lem:superdiff-sufficient.d}. Then parts \eqref{lem:h-superdiff-mean} and \eqref{lem:h-superdiff-extreme} follow from Lemma \ref{lm:hext}. 

The $\geq$ direction in \eqref{eq:suff-target} follows by first restricting to sites $x_n$ with $x_n/n\cdot e_1 \in [\delta,1-\delta]$ for $\delta\in(0,1/2)$, applying Lemma \ref{lem:shape} and \eqref{B-shape}, and then removing the restriction using the continuity of $\fe^\beta(\xi)+\hhB(\Bhat)\cdot\xi$ on $\Uset$. 

For the $\leq$ direction, fix $\delta\in(0,1/2)$. On the middle region $\delta n\le x\cdot e_1\le (1-\delta)n$, Lemma \ref{lem:shape} and \eqref{B-shape} apply directly and we only need to prove the bound for the two endpoint strips. Consider $0\le x\cdot e_i\le\delta n$ and $x \cdot (e_1+e_2) = n$, and let $y=x+\lceil\delta n\rceil e_i$. Then $y/(n+\lceil\delta n\rceil)$ is bounded away from the endpoints $\{e_1,e_2\}$. By extending paths from $x$ to $y$,
\[
\FE{0}{x}^{\beta,\w}-\Bhat(0,x)\le \FE{0}{y}^{\beta,\w}-\Bhat(0,x)+\sum_{k\le\lceil\delta n\rceil}|\w_{x+ke_i}|.
\]
The class $\sL$ condition \eqref{eq:class-L} makes the last term negligible after dividing by $n$ and first sending $n\to\infty$ and then $\delta\searrow0$. Lemma \ref{lem:shape} applies to the first term on the right, and \eqref{B-shape} applies to the cocycle term.  Thus the contribution of this endpoint strip is bounded, up to an error that vanishes with $\delta$, by the value of $\fe^\beta(\xi)+\hhB(\Bhat)\cdot\xi$ near $e_{3-i}$.  Letting $\delta\searrow0$ and using continuity gives the $\leq$ direction in \eqref{eq:suff-target}.
%
%
%
%
\end{proof}

We end with another convex analysis lemma.

\begin{lemma}\label{lem:supergrad-converge}
	Let $\Lambda_n:\bbR_{>0}^2\to\bbR$, $n\in\bbN\cup\{\infty\}$, be concave and positively homogeneous of degree one, and assume that $\Lambda_n\to\Lambda_\infty$ pointwise on $\ri\Uset=]e_2,e_1[$. Fix $\xi\in\ri\Uset$ and vectors $h_n\in\bbR^2$ such that
\[
-h_n\in\partial\Lambda_n(\xi)
\qquad\text{for all }n\in\bbN.
\]
If $\Lambda_\infty$ is differentiable at $\xi$, then $h_n\to -\nabla\Lambda_\infty(\xi)$.
\end{lemma}
\begin{proof}
Write $\theta=\xi\cdot e_1\in(0,1)$ and $g_n(t)=\Lambda_n(te_1+(1-t)e_2)$ for $t\in(0,1)$. Then $g_n$ is concave for each $n$, and $g_n\to g_\infty$ pointwise on $(0,1)$. The supergradient relation and homogeneity give $h_n\cdot\xi=-\Lambda_n(\xi)$.
Also, for $0<\delta<\theta\wedge(1-\theta)$, applying the supergradient inequality at
\[
\zeta_\pm=(\theta\pm\delta)e_1+(1-\theta\mp\delta)e_2
\]
gives
\[
\frac{g_n(\theta-\delta)-g_n(\theta)}{\delta}
\le h_n\cdot(e_1-e_2)
\le
\frac{g_n(\theta)-g_n(\theta+\delta)}{\delta}.
\]
Fix $\delta>0$ and let $n\to\infty$. Pointwise convergence yields 
\[
\frac{g_\infty(\theta-\delta)-g_\infty(\theta)}{\delta}
\le
\varliminf_{n\to\infty}h_n\cdot(e_1-e_2)
\le
\varlimsup_{n\to\infty}h_n\cdot(e_1-e_2)
\le
\frac{g_\infty(\theta)-g_\infty(\theta+\delta)}{\delta}.
\]
Since $\Lambda_\infty$ is differentiable at $\xi$, the function $g_\infty$ is differentiable at $\theta$, and letting $\delta\searrow0$ gives
\[
h_n\cdot(e_1-e_2)\to -\nabla\Lambda_\infty(\xi)\cdot(e_1-e_2).
\]
Also,
\[
h_n\cdot\xi=-\Lambda_n(\xi)\to-\Lambda_\infty(\xi)=-\nabla\Lambda_\infty(\xi)\cdot\xi.
\]
These two convergences determine the limits of the two coordinates of $h_n$, so $h_n\to-\nabla\Lambda_\infty(\xi)$.
\end{proof}

\section*{Statements and Declarations}

\textbf{Competing interests.}
The authors have no relevant financial or non-financial interests to disclose.

\textbf{Data availability.}
No datasets were generated or analyzed during the current study.

\textbf{Use of AI-assistance.}
The authors used OpenAI's ChatGPT and Codex large language models as writing tools during the preparation of this manuscript. This use consisted of assistance with proofreading; wording, proof organization, and notation suggestions; LaTeX editing and typesetting help; reference suggestions; and identifying points in proofs which required further careful manual review. The authors made all final decisions about the mathematical content and exposition and take full responsibility for the content of the manuscript.

\bibliographystyle{plain}
\bibliography{masterbib}

@article{Bas-Hof-Sly-22,
	author = {Basu, Riddhipratim and Hoffman, Christopher and Sly, Allan},
	doi = {10.1007/s00220-021-04246-0},
	fjournal = {Communications in Mathematical Physics},
	issn = {0010-3616},
	journal = {Comm. Math. Phys.},
	mrclass = {82C43},
	mrnumber = {4365136},
	number = {1},
	pages = {1--30},
	title = {Nonexistence of bigeodesics in planar exponential last passage percolation},
	url = {https://doi.org/10.1007/s00220-021-04246-0},
	volume = {389},
	year = {2022}}

@Article{Bus-25-,
  author   = {Ofer Busani},
  journal  = {Probab. Th. Rel. Fields},
  title    = {Non-existence of three non-coalescing infinite geodesics with the same direction in the directed landscape.},
  year     = {2025},
  note     = {\href{https://doi.org/10.1007/s00440-025-01451-z}{\tt Online first}},
}

@article{Bus-Sep-Sor-24,
	author = {Busani, Ofer and Sepp\"{a}l\"{a}inen, Timo and Sorensen, Evan},
	doi = {10.1214/23-aop1655},
	fjournal = {The Annals of Probability},
	issn = {0091-1798},
	journal = {Ann. Probab.},
	mrclass = {60K35 (60K37)},
	mrnumber = {4698024},
	number = {1},
	pages = {1--66},
	title = {The stationary horizon and semi-infinite geodesics in the directed landscape},
	url = {https://doi.org/10.1214/23-aop1655},
	volume = {52},
	year = {2024}}

@article{Jan-Ras-Sep-23,
	author = {Janjigian, Christopher and Rassoul-Agha, Firas and Sepp\"{a}l\"{a}inen, Timo},
	doi = {10.4171/jems/1246},
	fjournal = {Journal of the European Mathematical Society (JEMS)},
	issn = {1435-9855},
	journal = {J. Eur. Math. Soc. (JEMS)},
	mrclass = {60K35 (37H10 53C22 60K37 82B43)},
	mrnumber = {4612098},
	number = {7},
	pages = {2573--2639},
	title = {Geometry of geodesics through {B}usemann measures in directed last-passage percolation},
	url = {https://doi.org/10.4171/jems/1246},
	volume = {25},
	year = {2023}}

@article{Gou-07,
	author = {Gou\'{e}r\'{e}, Jean-Baptiste},
	doi = {10.1214/105051607000000113},
	fjournal = {The Annals of Applied Probability},
	issn = {1050-5164},
	journal = {Ann. Appl. Probab.},
	mrclass = {60K35 (92D25)},
	mrnumber = {2344307},
	mrreviewer = {Marina Vachkovskaia},
	number = {4},
	pages = {1273--1305},
	title = {Shape of territories in some competing growth models},
	url = {https://doi.org/10.1214/105051607000000113},
	volume = {17},
	year = {2007}}

@article{Bus-Sep-22-ejp,
	author = {Busani, Ofer and Sepp\"{a}l\"{a}inen, Timo},
	doi = {10.1214/21-ejp731},
	fjournal = {Electronic Journal of Probability},
	journal = {Electron. J. Probab.},
	mrclass = {60K35 (60K37)},
	mrnumber = {4372098},
	pages = {Paper No. 14, 40},
	title = {Non-existence of bi-infinite polymers},
	url = {https://doi.org/10.1214/21-ejp731},
	volume = {27},
	year = {2022}}

@book{Cas-Val-77,
	author = {Castaing, C. and Valadier, M.},
	mrclass = {46G99 (26A51 28A05 49A50 54C60)},
	mrnumber = {0467310},
	mrreviewer = {Vladimir L. Levin},
	pages = {vii+278},
	publisher = {Springer-Verlag, Berlin-New York},
	series = {Lecture Notes in Mathematics, Vol. 580},
	title = {Convex analysis and measurable multifunctions},
	url = {https://mathscinet.ams.org/mathscinet-getitem?mr=0467310},
	year = {1977}}

@article{Jan-Nur-Ras-22,
	author = {Janjigian, Christopher and Nurbavliyev, Sergazy and Rassoul-Agha, Firas},
	doi = {10.1214/21-aihp1200},
	fjournal = {Annales de l'Institut Henri Poincar\'{e} Probabilit\'{e}s et Statistiques},
	issn = {0246-0203},
	journal = {Ann. Inst. Henri Poincar\'{e} Probab. Stat.},
	mrclass = {60K35 (60K37)},
	mrnumber = {4421616},
	number = {2},
	pages = {1010--1040},
	title = {A shape theorem and a variational formula for the quenched {L}yapunov exponent of random walk in a random potential},
	url = {https://mathscinet.ams.org/mathscinet-getitem?mr=4421616},
	volume = {58},
	year = {2022}}

@Article{Jan-Ras-Sep-23-1F1S-,
  author        = {Christopher Janjigian and Firas Rassoul-Agha and Timo Sepp\"al\"ainen},
  title         = {Ergodicity and synchronization of the {K}ardar-{P}arisi-{Z}hang equation.},
  year          = {2023},
  note          = {Preprint (\href{https://arxiv.org/abs/2211.06779}{\tt arXiv:2211.06779})},
}

@book{Eth-Kur-86,
	author = {Ethier, Stewart N. and Kurtz, Thomas G.},
	doi = {10.1002/9780470316658},
	isbn = {0-471-08186-8},
	mrclass = {60J25 (60B10 60F05 60F17 60G44 60J80)},
	mrnumber = {838085},
	mrreviewer = {S. R. S. Varadhan},
	pages = {x+534},
	publisher = {John Wiley \& Sons, Inc., New York},
	series = {Wiley Series in Probability and Mathematical Statistics: Probability and Mathematical Statistics},
	title = {{M}arkov {P}rocesses: {C}haracterization and {C}onvergence},
	url = {https://mathscinet.ams.org/mathscinet-getitem?mr=838085},
	year = {1986}}

@article{Fan-Sep-20,
	author = {Wai-Tong (Louis) Fan and Timo Sepp\"al\"ainen},
	doi = {10.2140/pmp.2020.1.55},
	journal = {Prob. Math. Phys.},
	number = {1},
	pages = {55--100},
	publisher = {Mathematical Sciences Publishers},
	title = {Joint distribution of {B}usemann functions in the exactly solvable corner growth model},
	url = {https://doi.org/10.2140%2Fpmp.2020.1.55},
	volume = {1},
	year = 2020}

@article{Bal-Bus-Sep-20,
	author = {Bal\'{a}zs, M\'{a}rton and Busani, Ofer and Sepp\"{a}l\"{a}inen, Timo},
	doi = {10.1017/fms.2020.31},
	fjournal = {Forum of Mathematics. Sigma},
	journal = {Forum Math.\ Sigma},
	mrclass = {60K35 (60K37)},
	mrnumber = {4176750},
	pages = {Paper No. e46, 34},
	title = {Non-existence of bi-infinite geodesics in the exponential corner growth model},
	url = {https://mathscinet.ams.org/mathscinet-getitem?mr=4176750},
	volume = {8},
	year = {2020}}

@article{Jan-Ras-20-aop,
	author = {Janjigian, Christopher and Rassoul-Agha, Firas},
	doi = {10.1214/19-AOP1375},
	fjournal = {The Annals of Probability},
	issn = {0091-1798},
	journal = {Ann. Probab.},
	mrclass = {60K35 (60K37)},
	mrnumber = {4089495},
	number = {2},
	pages = {778--816},
	title = {Busemann functions and {G}ibbs measures in directed polymer models on {$\mathbb Z^2$}},
	url = {https://mathscinet.ams.org/mathscinet-getitem?mr=4089495},
	volume = {48},
	year = {2020}}

@article{Jan-Ras-20-jsp,
	author = {Janjigian, Christopher and Rassoul-Agha, Firas},
	doi = {10.1007/s10955-020-02541-z},
	fjournal = {Journal of Statistical Physics},
	issn = {0022-4715},
	journal = {J. Stat. Phys.},
	mrclass = {60 (82D60)},
	mrnumber = {4099996},
	number = {3},
	pages = {672--689},
	title = {Uniqueness and {E}rgodicity of {S}tationary {D}irected {P}olymers on {$\mathbb{Z}^2$}},
	url = {https://mathscinet.ams.org/mathscinet-getitem?mr=4099996},
	volume = {179},
	year = {2020}}

@article{Bak-Li-19,
	author = {Bakhtin, Yuri and Li, Liying},
	doi = {10.1002/cpa.21779},
	fjournal = {Communications on Pure and Applied Mathematics},
	issn = {0010-3640},
	journal = {Comm. Pure Appl. Math.},
	mrclass = {82D60 (35K59 35Q53 60Fxx 60K37)},
	mrnumber = {3911894},
	number = {3},
	pages = {536--619},
	title = {Thermodynamic limit for directed polymers and stationary solutions of the {B}urgers equation},
	url = {https://mathscinet.ams.org/mathscinet-getitem?mr=3911894},
	volume = {72},
	year = {2019}}

@article{Car-Sou-17,
	author = {Cardaliaguet, Pierre and Souganidis, Panagiotis E.},
	doi = {10.1016/j.crma.2017.06.001},
	fjournal = {Comptes Rendus Math\'{e}matique. Acad\'{e}mie des Sciences. Paris},
	issn = {1631-073X},
	journal = {C. R. Math. Acad. Sci. Paris},
	mrclass = {35B27 (35K55 35R60 49L20)},
	mrnumber = {3673054},
	number = {7},
	pages = {786--794},
	title = {On the existence of correctors for the stochastic homogenization of viscous {H}amilton-{J}acobi equations},
	url = {https://mathscinet.ams.org/mathscinet-getitem?mr=3673054},
	volume = {355},
	year = {2017}}

@Article{Jan-Ras-18-arxiv,
  author        = {Janjigian, Christopher and Rassoul-Agha, Firas},
  title         = {{B}usemann functions and {G}ibbs measures in directed polymer models on $\mathbb{Z}^2$.},
  year          = {2018},
  note          = {Extended version (\href{https://arxiv.org/abs/1810.03580v2}{\tt arXiv:1810.03580v2})},
}

@book{Com-17,
	author = {Comets, Francis},
	doi = {10.1007/978-3-319-50487-2},
	isbn = {978-3-319-50486-5; 978-3-319-50487-2},
	mrclass = {60K37 (60F10 60H05 60J10 82-01 82B41 82D60)},
	mrnumber = {3444835},
	mrreviewer = {Julien Poisat},
	note = {Lecture notes from the 46th Probability Summer School held in Saint-Flour, 2016},
	pages = {xv+199},
	publisher = {Springer, Cham},
	series = {Lecture Notes in Mathematics},
	title = {Directed polymers in random environments},
	url = {https://doi.org/10.1007/978-3-319-50487-2},
	volume = {2175},
	year = {2017}}

@book{New-97,
	author = {Newman, Charles M.},
	doi = {10.1007/978-3-0348-8912-4},
	isbn = {3-7643-5777-0},
	mrclass = {82B44 (60-02 60K35 82-02 82D30)},
	mrnumber = {1480664},
	mrreviewer = {Massimo Campanino},
	pages = {viii+88},
	publisher = {Birkh{\"a}user Verlag, Basel},
	series = {Lectures in Mathematics ETH Z{\"u}rich},
	title = {Topics in disordered systems},
	url = {https://doi.org/10.1007/978-3-0348-8912-4},
	year = {1997}}

@article{Aiz-Weh-90,
	author = {Aizenman, Michael and Wehr, Jan},
	fjournal = {Communications in Mathematical Physics},
	issn = {0010-3616},
	journal = {Comm. Math. Phys.},
	mrclass = {82B05 (82B26 82B44)},
	mrnumber = {1060388},
	mrreviewer = {Flora Koukiou},
	number = {3},
	pages = {489--528},
	title = {Rounding effects of quenched randomness on first-order phase transitions},
	url = {http://projecteuclid.org/euclid.cmp/1104200601},
	volume = {130},
	year = {1990}}

@book{Phe-01,
	author = {Phelps, Robert R.},
	doi = {10.1007/b76887},
	edition = {Second},
	isbn = {3-540-41834-2},
	mrclass = {46-01 (46A55 46J10)},
	mrnumber = {1835574},
	mrreviewer = {T. S. S. R. K. Rao},
	pages = {viii+124},
	publisher = {Springer-Verlag, Berlin},
	series = {Lecture Notes in Mathematics},
	title = {Lectures on {C}hoquet's theorem},
	url = {https://doi.org/10.1007/b76887},
	volume = {1757},
	year = {2001}}

@book{Bog-07,
	author = {Bogachev, V. I.},
	doi = {10.1007/978-3-540-34514-5},
	isbn = {978-3-540-34513-8; 3-540-34513-2},
	mrclass = {28-02 (28Axx 28Cxx 46G12 60G42 60G44)},
	mrnumber = {2267655},
	mrreviewer = {Ren{\'e} L. Schilling},
	pages = {Vol. I: xviii+500 pp., Vol. II: xiv+575},
	publisher = {Springer-Verlag, Berlin},
	title = {Measure theory. {V}ol. {I}, {II}},
	url = {https://doi.org/10.1007/978-3-540-34514-5},
	year = {2007}}

@article{Geo-Ras-Sep-17-ptrf-1,
	author = {Georgiou, Nicos and Rassoul-Agha, Firas and Sepp{\"a}l{\"a}inen, Timo},
	doi = {10.1007/s00440-016-0729-x},
	fjournal = {Probability Theory and Related Fields},
	issn = {0178-8051},
	journal = {Probab. Theory Related Fields},
	mrclass = {60K35},
	mrnumber = {3704768},
	number = {1-2},
	pages = {177--222},
	title = {Stationary cocycles and {B}usemann functions for the corner growth model},
	url = {http://dx.doi.org/10.1007/s00440-016-0729-x},
	volume = {169},
	year = {2017}}

@article{Geo-Ras-Sep-17-ptrf-2,
	author = {Georgiou, Nicos and Rassoul-Agha, Firas and Sepp{\"a}l{\"a}inen, Timo},
	doi = {10.1007/s00440-016-0734-0},
	fjournal = {Probability Theory and Related Fields},
	issn = {0178-8051},
	journal = {Probab. Theory Related Fields},
	mrclass = {60K35},
	mrnumber = {3704769},
	number = {1-2},
	pages = {223--255},
	title = {Geodesics and the competition interface for the corner growth model},
	url = {http://dx.doi.org/10.1007/s00440-016-0734-0},
	volume = {169},
	year = {2017}}

@Article{Ahl-Hof-16-,
  author        = {Daniel Ahlberg and Christopher Hoffman},
  title         = {Random coalescing geodesics in first-passage percolation.},
  year          = {2016},
  note          = {Preprint (\href{https://arxiv.org/abs/1609.02447}{\tt arXiv:1609.02447})},
}

@article{Bak-16,
	author = {Bakhtin, Yuri},
	doi = {10.1214/16-EJP4413},
	fjournal = {Electronic Journal of Probability},
	issn = {1083-6489},
	journal = {Electron. J. Probab.},
	mrclass = {37L30 (35Q53 35R60 37L55 60G55 60K35)},
	mrnumber = {3508684},
	pages = {Paper No. 37, 50},
	title = {Inviscid {B}urgers equation with random kick forcing in noncompact setting},
	url = {http://dx.doi.org/10.1214/16-EJP4413},
	volume = {21},
	year = {2016}}

@article{Geo-etal-15,
	author = {Georgiou, Nicos and Rassoul-Agha, Firas and Sepp{{\"a}}l{{\"a}}inen, Timo and Yilmaz, Atilla},
	doi = {10.1214/14-AOP933},
	fjournal = {The Annals of Probability},
	issn = {0091-1798},
	journal = {Ann. Probab.},
	mrclass = {60K35 (60K37)},
	mrnumber = {3395462},
	mrreviewer = {Dimitri Petritis},
	number = {5},
	pages = {2282--2331},
	title = {Ratios of partition functions for the log-gamma polymer},
	url = {http://dx.doi.org/10.1214/14-AOP933},
	volume = {43},
	year = {2015}}

@article{Ras-Sep-Yil-13,
	author = {Rassoul-Agha, Firas and Sepp{{\"a}}l{{\"a}}inen, Timo and Yilmaz, Atilla},
	coden = {CPAMA},
	doi = {10.1002/cpa.21417},
	fjournal = {Communications on Pure and Applied Mathematics},
	issn = {0010-3640},
	journal = {Comm. Pure Appl. Math.},
	mrclass = {60K37 (60F10 82B41)},
	mrnumber = {2999296},
	mrreviewer = {Dimitris Cheliotis},
	number = {2},
	pages = {202--244},
	title = {Quenched free energy and large deviations for random walks in random potentials},
	url = {http://dx.doi.org/10.1002/cpa.21417},
	volume = {66},
	year = {2013}}

@book{Auf-Dam-Han-17,
	author = {Antonio Auffinger and Jack Hanson and Michael Damron},
	isbn = {978-1-4704-4183-8},
	pages = {161},
	publisher = {American Mathematical Society, Providence, RI},
	series = {University Lecture Series},
	title = {50 years of first passage percolation},
	volume = {68},
	year = {2017}}

@article{Bak-13,
	author = {Bakhtin, Yuri},
	fjournal = {The Annals of Probability},
	issn = {0091-1798},
	journal = {Ann. Probab.},
	mrclass = {60H15 (37L55 60G55 60K37)},
	mrnumber = {3112935},
	mrreviewer = {Meihua Yang},
	number = {4},
	pages = {2961--2989},
	title = {The {B}urgers equation with {P}oisson random forcing},
	volume = {41},
	year = {2013}}

@article{Gar-Mar-05,
	author = {Garet, Olivier and Marchand, R{{\'e}}gine},
	doi = {10.1214/105051604000000503},
	fjournal = {The Annals of Applied Probability},
	issn = {1050-5164},
	journal = {Ann. Appl. Probab.},
	mrclass = {60K35 (82B43)},
	mrnumber = {2115045 (2006j:60106)},
	mrreviewer = {Marina Vachkovskaia},
	number = {1A},
	pages = {298--330},
	title = {Coexistence in two-type first-passage percolation models},
	url = {http://dx.doi.org/10.1214/105051604000000503},
	volume = {15},
	year = {2005}}

@article{Sep-12-corr,
	author = {Sepp{{\"a}}l{{\"a}}inen, Timo},
	coden = {APBYAE},
	doi = {10.1214/10-AOP617},
	fjournal = {The Annals of Probability},
	issn = {0091-1798},
	journal = {Ann. Probab.},
	mrclass = {60K35 (60K37 82B41 82D60)},
	mrnumber = {2917766},
	mrreviewer = {Jonathon R. Peterson},
	note = {Corrected version available at \href{https://arxiv.org/abs/0911.2446}{\tt arXiv:0911.2446}},
	number = {1},
	pages = {19--73},
	title = {Scaling for a one-dimensional directed polymer with boundary conditions},
	volume = {40},
	year = {2012}}

@article{Hof-08,
	author = {Hoffman, Christopher},
	doi = {10.1214/07-AAP510},
	fjournal = {The Annals of Applied Probability},
	issn = {1050-5164},
	journal = {Ann. Appl. Probab.},
	mrclass = {60K35 (82B43)},
	mrnumber = {2462555 (2010c:60295)},
	number = {5},
	pages = {1944--1969},
	title = {Geodesics in first passage percolation},
	url = {http://dx.doi.org/10.1214/07-AAP510},
	volume = {18},
	year = {2008}}

@book{Ras-Sep-15-ldp,
	author = {Rassoul-Agha, Firas and Sepp{{\"a}}l{{\"a}}inen, Timo},
	isbn = {978-0-8218-7578-0},
	mrclass = {60-01 (60F10 60J10 60K35 60K37 82B20)},
	mrnumber = {3309619},
	pages = {xiv+318},
	publisher = {American Mathematical Society, Providence, RI},
	series = {Graduate Studies in Mathematics},
	title = {A course on large deviations with an introduction to {G}ibbs measures},
	volume = {162},
	year = {2015}}

@article{Bur-Kea-89,
	author = {Burton, R. M. and Keane, M.},
	coden = {CMPHAY},
	fjournal = {Communications in Mathematical Physics},
	issn = {0010-3616},
	journal = {Comm. Math. Phys.},
	mrclass = {60K35 (82A43)},
	mrnumber = {990777 (90g:60090)},
	mrreviewer = {G. R. Grimmett},
	number = {3},
	pages = {501--505},
	title = {Density and uniqueness in percolation},
	url = {http://projecteuclid.org/euclid.cmp/1104178143},
	volume = {121},
	year = {1989}}

@article{Ras-Sep-14,
	author = {Rassoul-Agha, Firas and Sepp{{\"a}}l{{\"a}}inen, Timo},
	doi = {10.1007/s00440-013-0494-z},
	fjournal = {Probability Theory and Related Fields},
	issn = {0178-8051},
	journal = {Probab. Theory Related Fields},
	mrclass = {Preliminary Data},
	mrnumber = {3176363},
	number = {3-4},
	pages = {711--750},
	title = {Quenched point-to-point free energy for random walks in random potentials},
	url = {http://dx.doi.org/10.1007/s00440-013-0494-z},
	volume = {158},
	year = {2014}}

@article{Dam-Han-14,
	author = {Damron, Michael and Hanson, Jack},
	doi = {10.1007/s00220-013-1875-y},
	fjournal = {Communications in Mathematical Physics},
	issn = {0010-3616},
	journal = {Comm. Math. Phys.},
	mrclass = {94A17 (05Cxx 81P16 81P45)},
	mrnumber = {3152744},
	number = {3},
	pages = {917--963},
	title = {Busemann functions and infinite geodesics in two-dimensional first-passage percolation},
	url = {http://dx.doi.org/10.1007/s00220-013-1875-y},
	volume = {325},
	year = {2014}}

@article{Cha-94,
	author = {Chang, Cheng Shang},
	coden = {JPRBAM},
	fjournal = {Journal of Applied Probability},
	issn = {0021-9002},
	journal = {J. Appl. Probab.},
	mrclass = {60K25 (68M20 90B22)},
	mrnumber = {1303943 (95m:60136)},
	mrreviewer = {R. Subramanian},
	number = {4},
	pages = {1128--1133},
	title = {On the input-output map of a {$G/G/1$} queue},
	volume = {31},
	year = {1994}}

@inproceedings{New-95,
	address = {Basel},
	author = {Newman, Charles M.},
	booktitle = {Proceedings of the {I}nternational {C}ongress of {M}athematicians, {V}ol.\ 1, 2 ({Z}{\"u}rich, 1994)},
	mrclass = {60K35 (82B43)},
	mrnumber = {1404001},
	mrreviewer = {Yu Zhang},
	pages = {1017--1023},
	publisher = {Birkh{\"a}user},
	title = {A surface view of first-passage percolation},
	year = {1995}}

@article{Bak-Cat-Kha-14,
	author = {Bakhtin, Yuri and Cator, Eric and Khanin, Konstantin},
	fjournal = {Journal of the American Mathematical Society},
	issn = {0894-0347},
	journal = {J. Amer. Math. Soc.},
	mrclass = {Preliminary Data},
	mrnumber = {3110798},
	number = {1},
	pages = {193--238},
	title = {Space-time stationary solutions for the {B}urgers equation},
	volume = {27},
	year = {2014}}

@article{Lic-New-96,
	author = {Licea, Cristina and Newman, Charles M.},
	coden = {APBYAE},
	doi = {10.1214/aop/1042644722},
	fjournal = {The Annals of Probability},
	issn = {0091-1798},
	journal = {Ann. Probab.},
	mrclass = {60K35 (60D05 82B44)},
	mrnumber = {1387641},
	mrreviewer = {Olle H{{\"a}}ggstr{{\"o}}m},
	number = {1},
	pages = {399--410},
	title = {Geodesics in two-dimensional first-passage percolation},
	volume = {24},
	year = {1996}}

@article{Cou-11,
	author = {Coupier, David},
	doi = {10.1214/ECP.v16-1656},
	fjournal = {Electronic Communications in Probability},
	issn = {1083-589X},
	journal = {Electron. Commun. Probab.},
	mrclass = {60K35},
	mrnumber = {2836758},
	pages = {517--527},
	title = {Multiple geodesics with the same direction},
	volume = {16},
	year = {2011}}

@article{Com-Yos-06,
	author = {Comets, Francis and Yoshida, Nobuo},
	coden = {APBYAE},
	fjournal = {The Annals of Probability},
	issn = {0091-1798},
	journal = {Ann. Probab.},
	mrclass = {60K37 (60J60 60K35 82C44 82D60)},
	mrnumber = {2271480 (2007m:60305)},
	mrreviewer = {Marina Vachkovskaia},
	number = {5},
	pages = {1746--1770},
	title = {Directed polymers in random environment are diffusive at weak disorder},
	volume = {34},
	year = {2006}}

@article{Mar-04,
	author = {Martin, James B.},
	coden = {APBYAE},
	fjournal = {The Annals of Probability},
	issn = {0091-1798},
	journal = {Ann. Probab.},
	mrclass = {60K35},
	mrnumber = {2094434},
	mrreviewer = {Luis G. Gorostiza},
	number = {4},
	pages = {2908--2937},
	title = {Limiting shape for directed percolation models},
	volume = {32},
	year = {2004}}

@article{Hus-Hen-85,
	author = {Huse, David A. and Henley, Christopher L.},
	journal = {Phys. Rev. Lett.},
	month = {Jun},
	number = {25},
	numpages = {3},
	pages = {2708--2711},
	publisher = {American Physical Society},
	title = {Pinning and Roughening of Domain Walls in Ising Systems Due to Random Impurities},
	volume = {54},
	year = {1985}}

@book{Roc-70,
	address = {Princeton, N.J.},
	author = {Rockafellar, R. Tyrrell},
	mrclass = {26.52 (46.00)},
	mrnumber = {0274683 (43 \#445)},
	mrreviewer = {Ky Fan},
	pages = {xviii+451},
	publisher = {Princeton University Press},
	series = {Princeton Mathematical Series, No. 28},
	title = {Convex analysis},
	year = {1970}}

@book{Geo-88,
	address = {Berlin},
	author = {Georgii, Hans-Otto},
	isbn = {0-89925-462-4},
	mrclass = {82A05 (60K35 82-01 82A25 82A68)},
	mrnumber = {956646 (89k:82010)},
	mrreviewer = {Nicolae Angelescu},
	pages = {xiv+525},
	publisher = {Walter de Gruyter \& Co.},
	series = {de Gruyter Studies in Mathematics},
	title = {Gibbs measures and phase transitions},
	volume = {9},
	year = {1988}}

@book{Dem-Zei-98,
	address = {New York},
	author = {Dembo, Amir and Zeitouni, Ofer},
	edition = {Second},
	isbn = {0-387-98406-2},
	mrclass = {60F10},
	mrnumber = {1619036 (99d:60030)},
	pages = {xvi+396},
	publisher = {Springer-Verlag},
	series = {Applications of Mathematics (New York)},
	title = {Large deviations techniques and applications},
	volume = {38},
	year = {1998}}

@Article{Ale-23,
  author        = {Alexander, Kenneth S.},
  journal       = {Electron. J. Probab.},
  title         = {Geodesics, bigeodesics, and coalescence in first passage percolation in general dimension},
  year          = {2023},
  issn          = {1083-6489},
  pages         = {Paper No. 160, 83},
  volume        = {28},
  doi           = {10.1214/23-ejp1011},
  fjournal      = {Electronic Journal of Probability},
  mrclass       = {60K35 (82B43)},
  mrnumber      = {4673327},
  mrreviewer    = {Elcio\ Lebensztayn},
  url           = {https://doi.org/10.1214/23-ejp1011},
}

@Article{Gro-Jan-Ras-25-jsp,
  author   = {Groathouse, Sean and Janjigian, Christopher and Rassoul-Agha, Firas},
  journal  = {J. Stat. Phys.},
  title    = {Non-existence of non-trivial bi-infinite geodesics in {G}eometric {L}ast {P}assage {P}ercolation},
  year     = {2025},
  issn     = {0022-4715,1572-9613},
  number   = {6},
  pages    = {Paper No. 81},
  volume   = {192},
  doi      = {10.1007/s10955-025-03462-5},
  fjournal = {Journal of Statistical Physics},
  mrclass  = {60K35 (60K37)},
  mrnumber = {4917163},
  url      = {https://doi.org/10.1007/s10955-025-03462-5},
}

@article{Kur-07,
	author = {Kurtz, Thomas G.},
	doi = {10.1214/EJP.v12-431},
	fjournal = {Electronic Journal of Probability},
	issn = {1083-6489},
	journal = {Electron. J. Probab.},
	mrclass = {60H99 (60H10 60H15 60H20 60H25)},
	mrnumber = {2336594},
	mrreviewer = {Rainer\ Buckdahn},
	pages = {951--965},
	title = {The {Y}amada-{W}atanabe-{E}ngelbert theorem for general stochastic equations and inequalities},
	url = {https://doi.org/10.1214/EJP.v12-431},
	volume = {12},
	year = {2007}}

@Article{Jan-Ras-Sep-25-strong-,
  author = {Christopher Janjigian and Firas Rassoul-Agha and Timo Sepp\"al\"ainen},
  title  = {Strong existence and uniqueness of the tilt-indexed {B}usemann process in the planar corner growth model},
  year   = {2025},
  note   = {Preprint (\href{https://arxiv.org/abs/2507.08757}{\tt arXiv:2507.08757})},
}

@Article{Bat-Fan-Sep-25,
  author   = {Bates, Erik and Fan, Wai-Tong and Sepp\"al\"ainen, Timo},
  journal  = {Electron. J. Probab.},
  title    = {Intertwining the {B}usemann process of the directed polymer model},
  year     = {2025},
  issn     = {1083-6489},
  pages    = {Paper No. 50, 80},
  volume   = {30},
  doi      = {10.1214/25-ejp1310},
  fjournal = {Electronic Journal of Probability},
  mrclass  = {60K35 (60K37)},
  mrnumber = {4888159},
  url      = {https://doi.org/10.1214/25-ejp1310},
}

@InCollection{Jac-Mem-81,
  author     = {Jacod, Jean and M\'emin, Jean},
  booktitle  = {Seminar on {P}robability, {XV} ({U}niv. {S}trasbourg, {S}trasbourg, 1979/1980) ({F}rench)},
  publisher  = {Springer, Berlin},
  title      = {Sur un type de convergence interm\'ediaire entre la convergence en loi et la convergence en probabilit\'e},
  year       = {1981},
  isbn       = {3-540-10689-8},
  pages      = {529--546},
  series     = {Lecture Notes in Math.},
  volume     = {850},
  mrclass    = {60B10},
  mrnumber   = {622586},
  mrreviewer = {P.\ G\'erard},
}

@Article{Jac-Mem-81-Stoch,
  author     = {Jacod, Jean and M\'emin, Jean},
  journal    = {Stochastics},
  title      = {Existence of weak solutions for stochastic differential equations with driving semimartingales},
  year       = {1980/81},
  issn       = {0090-9491},
  number     = {4},
  pages      = {317--337},
  volume     = {4},
  doi        = {10.1080/17442508108833169},
  fjournal   = {Stochastics},
  mrclass    = {60H10 (60G48)},
  mrnumber   = {609691},
  mrreviewer = {J.\ M.\ Stoyanov},
  url        = {https://doi.org/10.1080/17442508108833169},
}

@Book{Kec-95,
  author     = {Kechris, Alexander S.},
  publisher  = {Springer-Verlag, New York},
  title      = {Classical descriptive set theory},
  year       = {1995},
  isbn       = {0-387-94374-9},
  series     = {Graduate Texts in Mathematics},
  volume     = {156},
  doi        = {10.1007/978-1-4612-4190-4},
  mrclass    = {03E15 (03-01 03-02 04A15 28A05 54H05 90D44)},
  mrnumber   = {1321597},
  mrreviewer = {Jakub\ Jasi\'nski},
  pages      = {xviii+402},
  url        = {https://doi.org/10.1007/978-1-4612-4190-4},
}

@book{Ali-Bor-06,
	author = {Aliprantis, Charalambos D. and Border, Kim C.},
	edition = {Third},
	isbn = {978-3-540-32696-0; 3-540-32696-0},
	mrclass = {46-01 (28D05 46N10 47-01 54-01 60B10 60J05)},
	mrnumber = {2378491},
	note = {A hitchhiker's guide},
	pages = {xxii+703},
	publisher = {Springer, Berlin},
	title = {Infinite dimensional analysis},
	year = {2006}}

@Book{Bee-93,
  author     = {Beer, Gerald},
  publisher  = {Kluwer Academic Publishers Group, Dordrecht},
  title      = {Topologies on closed and closed convex sets},
  year       = {1993},
  isbn       = {0-7923-2531-1},
  series     = {Mathematics and its Applications},
  volume     = {268},
  doi        = {10.1007/978-94-015-8149-3},
  mrclass    = {49-02 (46A99 46B99 46N10 49J45 54-01 90C31)},
  mrnumber   = {1269778},
  mrreviewer = {P.\ S.\ Kenderov},
  pages      = {xii+340},
  url        = {https://doi.org/10.1007/978-94-015-8149-3},
}

@Article{Gol-79,
  author     = {Goldstein, Sheldon},
  journal    = {Z. Wahrsch. Verw. Gebiete},
  title      = {Maximal coupling},
  year       = {1978/79},
  issn       = {0044-3719},
  number     = {2},
  pages      = {193--204},
  volume     = {46},
  doi        = {10.1007/BF00533259},
  fjournal   = {Zeitschrift f\"ur Wahrscheinlichkeitstheorie und Verwandte Gebiete},
  mrclass    = {60G05 (60J05)},
  mrnumber   = {516740},
  mrreviewer = {H.\ Kesten},
  url        = {https://doi.org/10.1007/BF00533259},
}

@Article{Str-65,
  author        = {Strassen, V.},
  journal       = {Ann. Math. Statist.},
  title         = {The existence of probability measures with given marginals},
  year          = {1965},
  issn          = {0003-4851},
  pages         = {423--439},
  volume        = {36},
  doi           = {10.1214/aoms/1177700153},
  fjournal      = {Annals of Mathematical Statistics},
  mrclass       = {60.05 (60.20)},
  mrnumber      = {177430},
  mrreviewer    = {J.\ Wolfowitz},
  url           = {https://doi.org/10.1214/aoms/1177700153},
}

@Article{Emr-Jan-Sep-25,
  author   = {Emrah, Elnur and Janjigian, Christopher and Sepp\"al\"ainen, Timo},
  journal  = {Comm. Math. Phys.},
  title    = {Anomalous geodesics in the inhomogeneous corner growth model},
  year     = {2025},
  issn     = {0010-3616,1432-0916},
  number   = {12},
  pages    = {Paper No. 316, 65},
  volume   = {406},
  doi      = {10.1007/s00220-025-05470-8},
  fjournal = {Communications in Mathematical Physics},
  mrclass  = {60K35 (60K25 60K37)},
  mrnumber = {4991160},
  url      = {https://doi.org/10.1007/s00220-025-05470-8},
}

@Article{Bas-Mj-26,
  author   = {Basu, Riddhipratim and Mj, Mahan},
  journal  = {Adv. Math.},
  title    = {Geodesic trees and exceptional directions in {FPP} on hyperbolic groups},
  year     = {2026},
  issn     = {0001-8708,1090-2082},
  pages    = {Paper No. 111030, 52},
  volume   = {498},
  doi      = {10.1016/j.aim.2026.111030},
  fjournal = {Advances in Mathematics},
  mrclass  = {60K35 (20F65 20F67 51F99 60J50 82B43)},
  mrnumber = {5074842},
  url      = {https://doi.org/10.1016/j.aim.2026.111030},
}

@Article{Gro-Jan-Ras-25-tams,
  author     = {Groathouse, Sean and Janjigian, Christopher and Rassoul-Agha, Firas},
  journal    = {Trans. Amer. Math. Soc.},
  title      = {Existence of generalized {B}usemann functions and {G}ibbs measures for random walks in random potentials},
  year       = {2025},
  issn       = {0002-9947,1088-6850},
  number     = {12},
  pages      = {8487--8563},
  volume     = {378},
  doi        = {10.1090/tran/9452},
  fjournal   = {Transactions of the American Mathematical Society},
  mrclass    = {60K35 (60K37)},
  mrnumber   = {4982325},
  mrreviewer = {Thomas\ Polaski},
  url        = {https://doi.org/10.1090/tran/9452},
}

\end{document}